\documentclass[twoside,11pt,reqno]{amsart}
\usepackage{amsmath,amssymb,amscd,mathrsfs,epic,bm,wasysym,latexsym,tikz,tikz-cd,mathrsfs,cite,hyperref,longtable}
\usepackage{pb-diagram}
\usepackage{stmaryrd}

\makeatletter

\numberwithin{equation}{section}

\allowdisplaybreaks

\newtheorem{Proposition}[equation]{Proposition}
\newtheorem{Lemma}[equation]{Lemma}
\newtheorem{Theorem}[equation]{Theorem}
\newtheorem{Corollary}[equation]{Corollary}
\newtheorem{MainTheorem}{Theorem}

\theoremstyle{definition}  

\newtheorem{Example}[equation]{Example}

\let\<\langle
\let\>\rangle

\newcommand\Comment[2][\relax]{\space\par\medskip\noindent%
   \fbox{\begin{minipage}{\textwidth}\textbf{Comment\ifx\relax#1\else---#1\fi}\newline%
        #2\end{minipage}}\medskip
}

\def\bi{{\bm{i}}}
\def\bj{{\bm{j}}}

\def\bolde{{\bm{e}}}
\def\bk{{\bm{k}}}

\def\b1{{\bm{1}}}

\def\bd{{\bm{d}}}

\def\bzero{{\bm{0}}}

\def\bolde{\text{\bf e}}

\def\bLa{{\bm{\Lambda}}}
\def\beps{{\bm{\varepsilon}}}
\def\bde{{\bm{\delta}}}
\def\bga{{\bm{\gamma}}}

\def\pmod#1{\text{ }(\text{\rm mod } #1)\,}

\newcommand{\KP}{\operatorname{KP}}

\newcommand{\Ext}{\operatorname{Ext}}
\newcommand{\EXT}{\operatorname{Ext}}

\newcommand{\im}{\operatorname{im}}
\newcommand{\id}{\operatorname{id}}

\def\sgn{\mathtt{sgn}}

\def\e{\mathtt{e}}

\newcommand{\head}{\operatorname{head}}

\newcommand{\Z}{\mathbb{Z}}

\def\eps{{\varepsilon}}
\def\phi{{\varphi}}

\newcommand{\F}{{\mathbb F}}

\newcommand{\ga}{\gamma}

\newcommand{\la}{\lambda}
\newcommand{\La}{\Lambda}
\newcommand{\al}{\alpha}
\newcommand{\be}{\beta}

\def\Si{\mathfrak{S}}
\newcommand{\si}{\sigma}
\newcommand{\om}{\omega}

\newcommand{\de}{\delta}
\newcommand{\De}{\Delta}

\renewcommand{\th}{\theta}

\def\id{\mathop{\mathrm {id}}\nolimits}

\newcommand{\Ind}{{\mathrm {Ind}}}

\newcommand{\C}{{\mathbb C}}

\newcommand{\D}{{\mathscr D}}

\newcommand{\di}{\mathrm{div}}

\def\b{\mathfrak{b}}
\def\k{\Bbbk}

\def\y{y}

\def\height{{\operatorname{ht}}}

\def\im{{\mathrm{im}\,}}

\def\into{{\hookrightarrow}}
\def\Mod#1{#1\!\operatorname{-Mod}}

\def\iso{\stackrel{\sim}{\longrightarrow}}

\def\col{{\rm col}}
\def\row{{\tt row}}

\def\HOM{\operatorname{Hom}}

\def\words{I}

\def\cc{{\tt c}}

\newcommand{\NH}{\mathcal{NH}}

{\catcode`\|=\active
  \gdef\set#1{\mathinner{\lbrace\,{\mathcode`\|"8000%
  \let|\midvert #1}\,\rbrace}}
}
\def\midvert{\egroup\mid\bgroup}

\colorlet{darkgreen}{green!50!black}
\tikzset{dots/.style={very thick,loosely dotted},
         greendot/.style={fill,circle,color=darkgreen,inner sep=1.5pt,outer sep=0},
         blackdot/.style={fill,circle,color=black,inner sep=1.5pt,outer sep=0},
         graydot/.style={fill,circle,color=gray,inner sep=1.1pt,outer sep=0}
}
\def\greendot(#1,#2){\node[greendot] at(#1,#2){}}
\def\blackdot(#1,#2){\node[blackdot] at(#1,#2){}}
\def\graydot(#1,#2){\node[graydot] at(#1,#2){}}

\newenvironment{braid}{
  \begin{tikzpicture}[baseline=6mm,black,line width=1pt, scale=0.32,
                      draw/.append style={rounded corners},
                      every node/.append style={font=\fontsize{5}{5}\selectfont}]%
  }{\end{tikzpicture}
}

\def\Grid(#1,#2){
  \draw[very thin,gray,step=2mm] (0,0)grid(#1,#2);
  \draw[very thin,darkgreen,step=10mm] (0,0)grid(#1,#2);
}

\newcommand\Tableau[2][\relax]{
  \begin{tikzpicture}[scale=0.5,draw/.append style={thick,black}]
    \ifx\relax#1\relax%
    \else 
      \foreach\box in {#1} { \filldraw[blue!30]\box+(-.5,-.5)rectangle++(.5,.5); }
    \fi
    \newcount\row\newcount\col
    \row=0
    \foreach \Row in {#2} {
       \col=1
       \foreach\k in \Row {
          \draw(\the\col,\the\row)+(-.5,-.5)rectangle++(.5,.5);
          \draw(\the\col,\the\row)node{\k};
          \global\advance\col by 1
       }
       \global\advance\row by -1
    }
  \end{tikzpicture}
}

\newcommand\YoungDiagram[2][\relax]{
  \begin{tikzpicture}[scale=0.5,draw/.append style={thick,black}]
    \ifx\relax#1\relax%
    \else 
    \foreach\box in {#1} {
      \filldraw[blue!30]\box rectangle ++(1,1);
    }
    \fi
    \newcount\row
    \row=0
    \foreach \col in {#2} {
       \draw(1,\the\row)grid ++(\col,1);
       \global\advance\row by -1
    }
  \end{tikzpicture}
}

\begin{document}

\title[Standard resolutions]{Resolutions of standard modules over  KLR algebras of type $A$}

\author{\sc Doeke Buursma}
\address{Department of Mathematics\\ University of Oregon\\Eugene\\ OR 97403, USA}\email{dbuursma@uoregon.edu}

\author{\sc Alexander Kleshchev}
\address{Department of Mathematics\\ University of Oregon\\Eugene\\ OR 97403, USA}\email{klesh@uoregon.edu}

\author{\sc David J. Steinberg}
\address{Department of Mathematics\\ University of Oregon\\
Eugene\\ OR~97403, USA}
\email{dsteinbe@uoregon.edu}


\thanks{The second author was supported by the NSF grants DMS-1161094, DMS-1700905, the Max-Planck-Institut, the Fulbright Foundation, and the DFG Mercator program through the University of Stuttgart.}

\begin{abstract}
Khovanov-Lauda-Rouquier algebras $R_\theta$ of finite Lie type are affine quasihereditary with standard modules $\De(\pi)$ labeled by Kostant partitions $\pi$ of $\theta$. In type $A$, we construct explicit projective resolutions of standard modules $\De(\pi)$.
\end{abstract}

\maketitle

\section{Introduction}\label{SIntro}
Let $R_{\theta,\F}$ be a Khovanov-Lauda-Rouquier (KLR) algebra of finite Lie type over a field $\F$ corresponding to $\theta\in Q_+$.
It is known that $R_{\theta,\F}$ is affine quasihereditary \cite{Kato,BKM,Kdonkin,KlL}. In particular, it has finite global dimension and comes with a family of {\em standard modules} $\{\De(\pi)_\F\mid \pi\in\KP(\theta)\}$ and {\em proper standard modules} $\{\bar\De(\pi)_\F\mid \pi\in\KP(\theta)\}$, where $\KP(\theta)$ denotes the set of the Kostant partitions of $\theta$. The (proper) standard modules have well-understood formal  characters.
Moreover, $L(\pi)_\F:=\head \De(\pi)_\F\cong \head \bar\De(\pi)_\F$ is irreducible, and $\{L(\pi)_\F\mid \pi\in \KP(\theta)\}$ is a complete set of irreducible $R_{\theta,\F}$-modules up to isomorphism and degree shift. Finally, the projective cover $P(\pi)_\F$ of $L(\pi)_\F$ has a finite $\De$-filtration and the (graded) decomposition number $d_{\si,\pi}^\F:=[\bar\De(\si)_\F:L(\pi)_\F]_q$ equals the (well-defined) multiplicity $(P(\pi)_\F:\De(\si)_\F)_q$, see \cite[Corollary 3.14]{BKM}.

The KLR algebra is defined over the integers, so we have a $\Z$-algebra  $R_\theta=R_{\theta,\Z}$ with $R_{\theta,\F}=R_\theta\otimes_\Z\F$. The standard modules have natural integral forms $\De(\pi)$ with $\De(\pi)_\F\cong \De(\pi)\otimes_\Z\F$, see \cite[\S4.2]{KS}. Moreover, as illustrated in \cite[\S4]{KS}, understanding $p$-torsion in the $\Z$-module $\Ext_{R_\theta}^m(\De(\pi),\De(\si))$ is relevant for comparing decomposition numbers $d_{\si,\pi}^\C$ and $d_{\si,\pi}^\F$ for a field $\F$ of characteristic $p$. This motivates our interest in
$\Ext_{R_\theta}^m(\De(\pi),\De(\si))$ and projective resolutions of (integral forms of) standard modules.

The problem of constructing a projective resolution of $\De(\pi)$ reduces easily to the semicuspidal case. To be more precise, let us fix a convex total order $\preceq$ on the set $\Phi_+$ of positive roots. A {\em Kostant partition} of $\theta$ is a sequence $\pi = (\beta_1^{m_1}, \ldots, \beta_t^{m_t})$ such that $m_1,\dots, m_t\in \Z_{>0}$,
$\beta_1 \succ \cdots \succ \beta_t$ are positive roots, and
$m_1\beta_1 + \cdots + m_t\beta_t = \theta$. Then
$\De(\pi)\cong \De(\beta_1^{m_1})\circ\dots\circ \De(\beta_t^{m_t})$ where `$\circ$' stands for induction product. Since induction product preserves projective modules, it is enough to resolve the standard modules of the form $\De(\be_s^{m_s})$, which are exactly the {\em semicuspidal standard modules}.
Moreover, the case where $\be_s$ is a simple root is easy: the algebra $R_{m\al_i}$ is the nil-Hecke algebra of rank $m$ and the standard module $\De(\al_i^m)$ is the projective indecomposable module $P(\al_i^m)=R_{m\al_i}1_{i^{(m)}}$ for an explicit primitive idempotent $1_{i^{(m)}}\in R_{m\al_i}$, see \cite[\S2.2]{KL1}.

Let us now specialize to Lie type $A_\infty$. Then every non-simple positive root is of the form $\al=\al_a+\al_{a+1}+\dots+\al_{b+1}$ for integers $a\leq b$, and we will work with the lexicographic convex order on the positive roots. Let $\theta:=m\al$. We consider the set of compositions
$$\La:=\{\la=(\la_a,\dots,\la_b)\mid 0\leq  \la_a,\dots,\la_b\leq m\}.$$
Let $\La(n):=\{\la\in\La\mid \la_a+\dots+\la_b=n\}$.
In Section~\ref{SDPR}, for each $\la\in \La$, we define explicit idempotents $e_\la\in R_{\theta}$ as concatenations of `divided power idempotents' of the form $1_{i^{(m)}}$ mentioned above. We define projective modules $P_\la :=q^{s_\la}R_\theta e_\la$, where $q^{s_\la}$ stands for the grading shift by an explicitly defined integer $s_\la$. For every $\mu\in\La(n+1)$ and $\mu\in \La(n)$ we then define explicit elements $d_n^{\mu,\la}\in e_\mu R_\theta e_\la$ so that
$$
P_\mu\mapsto P_\la,\ xe_\mu\mapsto xe_\mu d_n^{\mu,\la}
$$
is an $R_\theta$-module homomorphism. Taking direct sum over all such gives us a homomorhism
$$
d_n:P_{n+1}:=\bigoplus_{\mu\in \La(n+1)}P_\mu\longrightarrow P_{n}:=\bigoplus_{\la \in\La(n)}P_\la.
$$
We note that $P_n=0$ for $n>m(b-a+1)$ and that $d_n^{\mu,\la}=0$ unless the compositions $\mu$ and $\la$ differ in just one part. We also have a natural map $${\sf p}:P_0=q^{s_0}R_\theta e_0 \longrightarrow \De(\al^m),\ xe_0\mapsto xv_{\al^m},$$
where $v_{\al^m}$ is the standard generator of $\De(\al^m)$ of weight $a^m(a+1)^m\cdots (b+1)^m$, see (\ref{ESG}).

\begin{MainTheorem}\label{TA}
We have that
$$
0\longrightarrow P_{m(b-a+1)}\longrightarrow \dots \longrightarrow P_{n+1}\stackrel{d_n}{\longrightarrow} P_n\longrightarrow\dots \stackrel{d_0}{\longrightarrow} P_0\stackrel{\sf p}{\longrightarrow} \De(\al^m)\longrightarrow 0
$$
is a projective resolution of the standard module $\De(\al^m)$.
\end{MainTheorem}

When $m=1$ this is a version of the  resolution constructed in \cite[Theorem 4.12]{BKM}, and our resolution can be thought of as the `thick calculus generalization' (cf. \cite{KLMS}) of that resolution.

For an arbitrary $\theta$, we denote $\De:=\bigoplus_{\pi\in\KP(\theta)}\De(\pi)$.
In \cite{BKSTwo}, we will use the resolution from Theorem~\ref{TA} to describe the algebra $\Ext_{R_\theta}^*(\De,\De)$ in two special cases: (1) $\theta$ is a positive root in type $A$ and (2) Lie type is $A_2$, i.e. $\theta$ is of the form $r\al_1+s\al_2$.

We now describe the contents of the paper and the main idea of the proof. After reviewing KLR algebras and standard modules in Section~\ref{SPrelim}, we introduce the necessary combinatorial notation and define the resolution $P_\bullet$ in Section~\ref{SDPR}.

In order to prove that $P_\bullet$ is a resolution of $\De(\al^m)$, we want to show that it is a direct summand of a resolution $Q_\bullet$ of $q^{m(m-1)/2}\De(\al)^{\circ m}$ introduced in \S\ref{SQdef}. The resolution $Q_\bullet$ is obtained by taking the $m$th induced power of the resolution constructed in \cite[Theorem 4.12]{BKM}.

To check that
$P_\bullet$ is a direct summand of $Q_\bullet$, in \S\ref{SSComp}
we construct what will end up being a pair of chain maps
$f:P_\bullet\to Q_\bullet$ and $g:Q_\bullet\to P_\bullet$ with $g\circ f=\id$. The main difficulty is to verify that $f$ and $g$ are indeed chain maps. This verification occupies Sections~\ref{SF} and \ref{SG}. Modulo the fact that $f$ and $g$ are chain maps,  Theorem~\ref{TA} is proved in \S\ref{SSProof}.

\section{Preliminaries}\label{SPrelim}
\subsection{Basic notation}
Throughout, we work over an arbitrary commutative unital ground~$\k$ (since everything is defined over $\Z$, one could just consider the case $\k=\Z$).

For $r,s\in\Z$, we use the segment notation
$[r,s]:=\{t\in\Z\mid r\leq t\leq s\},\ [r,s):=\{t\in\Z\mid r\leq t< s\},\ \text{etc.}
$

Let $q$ be a variable, and $\Z((q))$ be the ring of Laurent series. For a non-negative integer $n$ we define
$
[n]:=(q^n-q^{-n})/(q-q^{-1})$ and
$[n]^!:=[1][2]\cdots [n].
$

We denote by $\Si_d$ the symmetric group on $d$ letters.
It is a Coxeter group with generators $\{s_r:=(r,r+1)\mid 1\leq r<d\}$ and the corresponding length function $\ell$. The longest element of $\Si_d$ is denoted $w_0$ or $w_{0,d}$. An element $w\in \Si_d$ is called fully commutative if it is possible to go between any two reduced decompositions of $w$ using only the relations of the form $s_r s_t=s_ts_r$ for $|r-t|>1$.
By definition, $\Si_d$ acts on $[1,d]$ on the left.
For a set $I$ the $d$-tuples from $I^d$ are written as words $\bi=i_1\cdots i_d$. The group $\Si_d$ acts on $I^d$ via place permutations as $w\cdot \bi=i_{w^{-1}(1)}\cdots i_{w^{-1}(d)}$.

Given a composition $\la=(\la_1,\dots,\la_k)$ of $d$, we have the corresponding standard parabolic subgroup
$\Si_\la:=\Si_{\la_1}\times \dots\times \Si_{\la_k}\leq \Si_d$.
For compositions $\la$ and $\mu$ of $d$, we denote by $\D^\la$ (resp. ${}^\mu\D$, resp. ${}^\mu\D^\la$) the set of the shortest coset representatives for
$\Si_d/\Si_\la$ (resp.
$\Si_\mu\backslash\Si_d$, resp.
$\Si_\mu\backslash\Si_d/\Si_\la$). The following is well-known and can be deduced for example from \cite[Lemma 1.6]{DJ}:

\begin{Lemma} \label{LDJ} 
Let $\la,\mu$ be compositions of $d$ and
 $w\in {}^\mu\D$. Then there exist unique elements $x\in {}^\mu\D^\la$ and $y\in\Si_\la$ such that $w=xy$ and $\ell(w)=\ell(x)+\ell(y)$.
\end{Lemma}

\subsection{KLR Algebras}\label{SSKLR}

From now on, we set $I:=\Z$. If $i,j \in I$ with $|i-j|=1$ we set $\eps_{i,j}:=j-i\in\{1,-1\}$.
For $i,j\in I$, we set
\begin{equation*}\label{ECN}
\cc_{i, j}:=
\left\{
\begin{array}{ll}
2 &\hbox{if $i=j$,}\\
-1 &\hbox{if $|i-j|=1$,}\\
0 &\hbox{otherwise.}
\end{array}
\right.
\end{equation*}

We identify $I$ with the set of vertices of the Dynkin diagram of type $A_\infty$ so that the numbers $\cc_{i,j}$ are the entries of the corresponding Cartan matrix.
The corresponding simple roots are denoted $\{\al_i\mid i\in I\}$ and set of the  positive roots is $\Phi_+:= \{\al_r+\dots+\al_s\mid r\leq s\}$. The root lattice is $Q:=\bigoplus_{i\in I}\Z\cdot\al_i$, and we set $Q_+:=\{\sum_i m_i\al_i\in Q\mid m_i\in\Z_{\geq 0}\ \text{for all $i$}\}$. For $\theta=\sum_i m_i\al_i\in Q_+$, we define its height $\height(\theta):=\sum_i m_i$. For $\theta\in Q_+$ of height $d$, we define
$I^\theta:=\{\bi=i_1\cdots i_d\in I^d\mid \al_{i_1}+\dots+\al_{i_d}=\theta\}$. If $\bi\in I^\theta$ and $\bj\in I^\eta$ then the concatenation of words $\bi\bj$ is an element of $ I^{\theta+\eta}$.

Let $\theta\in Q_+$ be of height $d$. The {\em KLR algebra}  \cite{KL1,Ro} is the unital $\k$-algebra $R_\theta$ (with identity denoted $1_\theta$)
with generators
$$\{1_\bi \mid  \bi \in \words^\theta \} \cup \{ y_1, \dots, y_d\} \cup \{\psi_1, \dots, \psi_{d-1}\}$$ and defining relations
\begin{align*}
	&\y_r \y_s = \y_s \y_r; 
	\\
	&1_\bi1_\bj=\de_{\bi,\bj}1_\bi \quad \text{and} \quad \sum_{\bi\in\words^\theta}1_\bi=1_\theta;  
	\\
	&\y_r 1_\bi = 1_\bi\y_r \quad \text{and} \quad \psi_r 1_\bi = 1_{s_r\cdot \bi}\psi_r; 
	\\
	&(\psi_r y_t - y_{s_r(t)} \psi_r)1_\bi = \de_{i_r,i_{r+1}}(\de_{t,r+1}-\de_{t,r})1_\bi; 
	\\
	&\psi_r^2 1_\bi =
	\begin{cases}
		0 & \text{if } i_r=i_{r+1}, \\
		\eps_{i_r,i_{r+1}}({y_r}-{y_{r+1}})1_\bi & \text{if } |i_r-i_{r+1}|=1, \\
		1_\bi & \text{otherwise};
	\end{cases} 
	\\
	&\psi_r \psi_s = \psi_s \psi_r \text{ if } |r-s|>1; 
	\\
	&(\psi_{r+1} \psi_{r} \psi_{r+1} - \psi_{r}\psi_{r+1}\psi_{r}) 1_\bi =
	\begin{cases}
		\eps_{i_r,i_{r+1}}1_\bi & \text{if } |i_r-i_{r+1}|=1 \text{ and } i_r = i_{r+2}, \\
		0 & \text{otherwise}.
	\end{cases} 
\end{align*}
The algebra $R_\theta$ is graded with $\deg 1_{\bi}=0$; $\deg (y_s)= 2$; $\deg (\psi_{r}1_{\bi})=-\cc_{i_r,i_{r+1}}$.

We will  use the Khovanov-Lauda \cite{KL1} diagrammatic notation for elements of $R_\theta$. In particular, for $\bi=i_1\cdots i_d\in I^\theta$, $1\leq r<d$ and $1\leq s\leq d$, we denote
$$
1_\bi=
\begin{braid}\tikzset{baseline=0mm}

\draw(0,1) node[above]{$i_1$}--(0,-1);
\draw(1,1) node[above]{$i_2$}--(1,-1);
\draw(2,1.1) node[above]{$\cdots$};
\draw(3,1) node[above]{$i_d$}--(3,-1);

\end{braid},\quad
1_\bi\psi_r=
\begin{braid}\tikzset{baseline=0mm}

\draw(0,1) node[above]{$i_1$}--(0,-1);
\draw(1,1.1) node[above]{$\cdots$};
\draw(2.5,1) node[above]{$i_{r-1}$}--(2.5,-1);
\draw(4,1.05) node[above]{$i_{r}$}--(5.3,-1);
\draw(5.3,1) node[above]{$i_{r+1}$}--(4,-1);
\draw(7,1) node[above]{$i_{r+2}$}--(7,-1);
\draw(8.5,1.1) node[above]{$\cdots$};
\draw(9.5,1) node[above]{$i_{d}$}--(9.5,-1);

\end{braid}
,\quad
1_\bi y_s=
\begin{braid}\tikzset{baseline=0mm}

\draw(0,1) node[above]{$i_1$}--(0,-1);
\draw(1,1.1) node[above]{$\cdots$};
\draw(2.5,1) node[above]{$i_{s-1}$}--(2.5,-1);
\draw(4,1.05) node[above]{$i_{s}$}--(4,-1);
\draw(5.3,1) node[above]{$i_{s+1}$}--(5.3,-1);
\draw(7.5,1.1) node[above]{$\cdots$};
\draw(8.5,1) node[above]{$i_{d}$}--(8.5,-1);
\blackdot(4,0);

\end{braid}
$$

For each element $w \in \Si_n$, fix a reduced expression $w=s_{r_1}\cdots s_{r_l}$ which determines an element $\psi_w = \psi_{r_1} \cdots \psi_{r_l}$ (depending on the choice of a reduced expression).

\begin{Theorem}\label{TBasis}{\cite[Theorem 2.5]{KL1}}, \cite[Theorem 3.7]{Ro}
Let $\th\in Q_+$ and $d=\height(\th)$. Then  the following sets are $\k$-bases of  $R_\th$:
\begin{align*} \{\psi_w y_1^{k_1}\dots y_d^{k_d}1_\bi\mid w\in \Si_d,\ k_1,\dots,k_d\in\Z_{\geq 0}, \ \bi\in I^\th\},
\\
 \{y_1^{k_1}\dots y_d^{k_d}\psi_w 1_\bi\mid w\in \Si_d,\ k_1,\dots,k_d\in\Z_{\geq 0}, \ \bi\in I^\th\}.
\end{align*}
\end{Theorem}

\subsection{Parabolic subalgebras and divided power idempotents}
Let $\theta_1,\dots,\theta_t\in Q_+$ and set $\theta:=\theta_1+\dots+\theta_t$. Set
$$1_{\theta_1,\dots,\theta_t}:=\sum_{\bi^1\in I^{\theta_1},\dots,\bi^t\in I^{\theta_t}}1_{\bi^1\cdots\bi^t}\in R_\theta.
$$
Then we have an algebra embedding
\begin{equation}\label{EEmb}
\iota_{\theta_1,\dots,\theta_t}: R_{\theta_1}\otimes\dots\otimes R_{\theta_t}\into  1_{\theta_1,\dots,\theta_t}R_{\theta_1+\dots+\theta_t}1_{\theta_1,\dots,\theta_t}
\end{equation}
obtained by horizontal concatenation of the Khovanov-Lauda diagrams. For  $r_1\in R_{\theta_1},\dots, r_t\in R_{\theta_t}$, when there is no confusion, we often write
$$
r_1\circ\dots\circ r_t:=\iota_{\theta_1,\dots,\theta_t}(r_1\otimes\dots\otimes r_t).
$$
For example,
\begin{equation}\label{E240918}
1_{\bi^1}\circ\dots\circ  1_{\bi^t}=1_{\bi^1\cdots \bi^t}\qquad(\bi^1\in I^{\theta_1},\dots,\bi^t\in I^{\theta_t}).
\end{equation}

We fix for the moment $i\in I$, $d\in\Z_{>0}$ and take $\theta=d\al_i$, in which case $R_{d\al_i}$ is isomorphic to the rank $d$ nil-Hecke algebra $\NH_d$.
Following \cite{KL1}, we consider the following elements of $R_{d\al_i}$:
$$
y_{0,d}:=\prod_{r=1}^d y_r^{r-1},\quad
1_{i^{(d)}}:=\psi_{w_{0,d}}y_{0,d},\quad\text{and}\quad
1_{i^{(d)}}':=y_{0,d}\psi_{w_{0,d}}.
$$

The following is well-known, see for example  \cite[\S2.2]{KL1}:

\begin{Lemma} \label{LW0Y0} 
In $R_{d\al_i}$, the elements $1_{i^{(d)}}$ and $1_{i^{(d)}}'$ are  idempotents. Moreover,
\begin{enumerate}
\item \label{LW0Y0i} $\psi_{w_{0,d}}f\psi_{w_{0,d}}=0$ for any polynomial $f$ in $y_1,\dots,y_d$ of degree less than $d(d-1)/2$.
\item \label{LW0Y0ii} $\psi_{w_{0,d}}y_{0,d}\psi_{w_{0,d}}=\psi_{w_{0,d}}$.
\end{enumerate}
\end{Lemma}

Now, let $\theta\in Q_+$ be arbitrary. We define $I^{\theta}_{\di}$ to be the set of all expressions of the form
$i_1^{(d_1)} \cdots i_r^{(d_r)}$ with
$d_1,\dots,d_r\in \Z_{\ge 0}$, $i_1,\dots,i_r\in I$
 and $d_1 \al_{i_1} + \cdots + d_r \al_{i_r} = \theta$. We refer to such expressions as {\em divided power words}.  We identify $I^\theta$ with the subset of $I^\theta_\di$ which consists of all divided power words as above with all $d_k=1$.
 We use the same notation for concatenation of divided power words as for concatenation of words.
For $\bi = i_1^{(d_1)}\cdots i_r^{(d_r)} \in I^{\theta}_{\di}$, we define the {\em divided power idempotent}
\[
1_{\bi}=1_{i_1^{(d_1)}\cdots i_r^{(d_r)}}:=1_{i_1^{(d_1)}}\circ\dots\circ 1_{i_r^{(d_r)}} \in R_{\theta}.
\]
Then we have the following analogue of (\ref{E240918}):
\begin{equation}\label{E240918_2}
1_{\bi^1}\circ\dots\circ  1_{\bi^t}=1_{\bi^1\cdots \bi^t}\qquad(\bi^1\in I^{\theta_1}_{\di},\dots,\bi^t\in I^{\theta_t}_{\di}).
\end{equation}
To be used as part of the Khovanov-Lauda diagrammatics, we denote
$$
\psi_{w_{0,d}}=:
\begin{braid}\tikzset{baseline=0mm}
\draw(1,-0.7) node[above]{\small $w_0$};

\draw [rounded corners,color=gray] (-0.2,0.8)--(2.2,0.8)--(2.2,-0.4)--(-0.2,-0.4)--cycle;

\end{braid},\quad
y_{0,d}=:
\begin{braid}\tikzset{baseline=0mm}
\draw(1,-0.7) node[above]{\small $y_0$};

\draw [rounded corners,color=gray] (-0.2,0.8)--(2.2,0.8)--(2.2,-0.4)--(-0.2,-0.4)--cycle;

\end{braid},\quad
1_{i^{(d)}}=
\begin{braid}\tikzset{baseline=0mm}
\draw(0,2.5)node[above]{$i$}--(0,1.5);
\draw(1,2.5) node[above]{$\cdots$};
\draw(2,2.5)node[above]{$i$}--(2,1.6);

\draw(1,-2) node[above]{\small $y_0$};

\draw(1,0.1) node[above]{\small $w_0$};

\draw [rounded corners,color=gray] (-0.2,1.6)--(2.2,1.6)--(2.2,0.4)--(-0.2,0.4)--cycle;

\draw(0,0.5)--(0,-0.5);
\draw(2,0.5)--(2,-0.5);
\draw(1,-0.7) node[above]{$\cdots$};

\draw [rounded corners,color=gray] (-0.2,-0.5)--(2.2,-0.5)--(2.2,-1.7)--(-0.2,-1.7)--cycle;

\end{braid}=:
\begin{braid}\tikzset{baseline=0mm}
\draw(0,-0.7) node[above]{\small $i$};
\draw(1,-0.7) node[above]{\small $\cdots$};
\draw(2,-0.7) node[above]{\small $i$};
\draw [rounded corners,color=gray] (-0.4,0.8)--(2.4,0.8)--(2.4,-0.4)--(-0.4,-0.4)--cycle;
\end{braid}
=
\begin{braid}\tikzset{baseline=0mm}
\draw(1,-0.7) node[above]{\small $i^{d}$};
\draw [rounded corners,color=gray] (-0.4,0.8)--(2.4,0.8)--(2.4,-0.4)--(-0.4,-0.4)--cycle;
\end{braid}
$$

For example, if $d=3$, we have
$$
1_{i^3}\psi_{w_0}=
\begin{braid}\tikzset{baseline=0mm}
\draw(0,1) node[above]{$i$}--(0,0);
\draw(1,1) node[above]{$i$}--(1,0);
\draw(2,1) node[above]{$i$}--(2,0);
\draw(1,-1.3) node[above]{\small $w_0$};
\draw [rounded corners,color=gray] (-0.2,0)--(2.2,0)--(2.2,-1)--(-0.2,-1)--cycle;
\end{braid}
=
\begin{braid}\tikzset{baseline=0mm}
\draw(0,1) node[above]{$i$}--(2,-2);
\draw(1,1) node[above]{$i$}--(0,-0.5)--(1,-2);
\draw(2,1) node[above]{$i$}--(0,-2);
\end{braid},\quad
1_{i^3}y_0=
\begin{braid}\tikzset{baseline=0mm}
\draw(0,1) node[above]{$i$}--(0,0);
\draw(1,1) node[above]{$i$}--(1,0);
\draw(2,1) node[above]{$i$}--(2,0);
\draw(1,-1.3) node[above]{\small $y_0$};

\draw [rounded corners,color=gray] (-0.2,0)--(2.2,0)--(2.2,-1)--(-0.2,-1)--cycle;

\end{braid}
=
\begin{braid}\tikzset{baseline=0mm}
\draw(0,1) node[above]{$i$}--(0,-2);
\draw(1,1) node[above]{$i$}--(1,-2);
\draw(2,1) node[above]{$i$}--(2,-2);
\blackdot(1,-0.5);
\blackdot(2,0);
\blackdot(2,-1);
\end{braid},
\quad
1_{i^{(3)}}=
\begin{braid}\tikzset{baseline=0mm}
\draw(0,-0.7) node[above]{\small $i$};
\draw(1,-0.7) node[above]{\small $i$};
\draw(2,-0.7) node[above]{\small $i$};
\draw [rounded corners,color=gray] (-0.4,0.7)--(2.4,0.7)--(2.4,-0.5)--(-0.4,-0.5)--cycle;
\end{braid}=
\begin{braid}\tikzset{baseline=0mm}
\draw(0,1) node[above]{$i$}--(2,-2);
\draw(1,1) node[above]{$i$}--(0,-0.5)--(1,-2);
\draw(2,1) node[above]{$i$}--(0,-2);
\blackdot(0.8,-1.7);
\blackdot(1.4,-1.1);
\blackdot(1.8,-1.7);
\end{braid}
$$
More generally, we denote
$$
1_{i_1^{(d_1)}\cdots\, i_r^{(d_r)}}=:
\begin{braid}\tikzset{baseline=0mm}
\draw(1,-0.7) node[above]{\small $i_1^{d_1}$};
\draw [rounded corners,color=gray] (-0.4,1.2)--(2.4,1.2)--(2.4,-0.5)--(-0.4,-0.5)--cycle;
\end{braid}\cdots
\begin{braid}\tikzset{baseline=0mm}
\draw(1,-0.7) node[above]{\small $i_r^{d_r}$};
\draw [rounded corners,color=gray] (-0.4,1.2)--(2.4,1.2)--(2.4,-0.5)--(-0.4,-0.5)--cycle;
\end{braid}
$$

For any $1\leq r< t\leq d$, we denote the cycle $(t,t-1,\dots,r)\in\Si_d$ by $(t\to r)$. Note that $(t\to r)=s_{t-1}s_{t-2}\cdots s_{r}$ is a unique reduced decomposition. So
\begin{equation}\label{EPsiCycle}
\psi_{t\to r}:=\psi_{(t\to r)}=\psi_{t-1}\psi_{t-2}\dots\psi_{r}\in R_\theta,
\end{equation}
In terms of the diagrammatic notation we have
$$
1_\bi\psi_{t\to r}=\begin{braid}\tikzset{baseline=0mm}

\draw(0,1) node[above]{$i_1$}--(0,-1);
\draw(1,1.1) node[above]{$\cdots$};
\draw(2.5,0.9) node[above]{$i_{r-1}$}--(2.5,-1);
\draw(4,1) node[above]{$i_{r}$}--(5.2,-1);

\draw(5.1,1.1) node[above]{$\cdots$};
\draw(6.5,0.9) node[above]{$i_{t-1}$}--(7.7,-1);

\draw(7.7,1) node[above]{$i_{t}$}--(4,-1);
\draw(8.9,0.9) node[above]{$i_{t+1}$}--(8.9,-1);

\draw(10.4,1.1) node[above]{$\cdots$};
\draw(11.4,1) node[above]{$i_{d}$}--(11.4,-1);

\end{braid}.
$$

\begin{Lemma} \label{LDivPower} 
In the algebra $R_{d\al_i}$, we have
\begin{enumerate}
\item \label{LDivPoweri} If $r_1+\dots+r_t=d$ then
$1_{i^{(r_1)} \cdots i^{(r_t)}}\psi_{w_{0,d}}=\psi_{w_{0,d}}$ and
$1_{i^{(r_1)} \cdots i^{(r_t)}}1_{i^{(d)}}=1_{i^{(d)}}$.
\item \label{LDivPowerii} $1_{i^{(d)}}\psi_{d\to 1}1_{i\,i^{(d-1)}}
=\psi_{d\to 1}1_{i\,i^{(d-1)}}$.
\end{enumerate}
\end{Lemma}
\begin{proof}
(i) Write $\psi_{w_{0,d}}=(\psi_{w_{0,r_1}}\circ\dots\circ \psi_{w_{0,r_t}})\psi_u$
for some $u\in\Si_d$ and use Lemma~\ref{LW0Y0}.

(ii) We have
\begin{align*}
1_{i^{(d)}}\psi_{d-1}\psi_{d-2}\dots\psi_11_{i\,i^{(d-1)}}
&=
1_{i^{(d)}}\psi_{d-1}\psi_{d-2}\dots\psi_1(1_{\al_i}\circ \psi_{w_{0,d-1}}y_{0,d-1})
\\
&=
1_{i^{(d)}}\psi_{w_{0,d}}(1_{\al_i}\circ y_{0,d-1})
\\
&=
\psi_{w_{0,d}}(1_{\al_i}\circ y_{0,d-1})
\\
&=
\psi_{d-1}\psi_{d-2}\dots\psi_1(1_{\al_i}\circ \psi_{w_{0,d-1}}y_{0,d-1})
\\
&=
\psi_{d\to 1}1_{i\,i^{(d-1)}},
\end{align*}
where we have used Lemma~\ref{LW0Y0}(\ref{LW0Y0ii}) for the third equality.
\end{proof}

The following lemma easily follows from the defining relations of $R_\theta$:

\begin{Lemma} \label{LPull} 
Let $1\leq r< s\leq d$, $t\in(r,s)$, $u\in [r,s)$, and $\bi\in I^\theta$. Then in $R_\theta$ we have:
\begin{enumerate}
\item[{\rm (i)}] $1_\bi\psi_{s\to r} \psi_t=1_\bi\psi_{t-1}\psi_{s\to r}$ unless $i_s=i_{t-1}=i_t\pm 1$.
\item[{\rm (ii)}] $1_\bi\psi_{s\to r} y_{u+1}=1_\bi y_{u}\psi_{s\to r}$ unless $i_s=i_{u}$.
\end{enumerate}
\end{Lemma}

\subsection{Modules over $R_\theta$}
Let $\theta\in Q_+$. We denote by $\Mod{R_\theta}$ the category of 
graded left $R_\theta$-modules. The morphisms in this category are all homogeneous degree zero  $R_\theta$-homomorphisms, which we denote $\hom_{R_\theta}(-,-)$.
For  $V\in\Mod{R_\theta}$, let $q^d V$ denote its grading shift by $d$,  so if $V_m$ is the degree $m$ component of $V$, then $(q^dV)_m= V_{m-d}.$ More generally, for a Laurent series  $a=a(q)=\sum_{n}a_nq^n\in\Z[q,q^{-1}]$ with non-negative coefficients, we set $a V:=\bigoplus_n(q^n V)^{\oplus a_n}$.
For $U,V\in \Mod{R_\theta}$,
we set
$\HOM_{R_\theta}(U, V):=\bigoplus_{d \in \Z} \HOM_{R_\theta}(U, V)_d$,
where
$$
\HOM_{R_\theta}(U, V)_d := \hom_{R_\theta}(q^d U, V) = \hom_{R_\theta}(U, q^{-d}V).
$$
We define
$\EXT^m_{R_\theta}(U,V)$ similarly in terms of $\operatorname{ext}^m_{R_\theta}(U,V)$.



For $\theta_1,\dots,\theta_t\in Q_+$ and $\theta:=\theta_1+\dots+\theta_t$, recalling (\ref{EEmb}), we have a 
functor
$$\Ind_{\theta_1,\dots,\theta_t}=R_\theta 1_{\theta_1,\dots,\theta_t}\otimes_{R_{\theta_1}\otimes\dots\otimes R_{\theta_t}}-:\Mod{(R_{\theta_1}\otimes\dots\otimes R_{\theta_t})}\to\Mod{R_{\theta}}.
$$
For $V_1\in\Mod{R_{\theta_1}}, \dots, V_t\in\Mod{R_{\theta_t}}$, we denote by $V_1\boxtimes \dots\boxtimes V_t$ the $\k$-module $V_1\otimes \dots\otimes V_t$, considered naturally as an $(R_{\theta_1}\otimes\dots\otimes R_{\theta_t})$-module, and set
$$V_1\circ\dots\circ V_t:=\Ind_{\theta_1,\dots,\theta_t} V_1\boxtimes \dots\boxtimes V_t.$$
For $v_1\in V_1,\dots, v_t\in V_t$, we denote
$$
v_1\circ \dots\circ v_t:= 1_{\theta_1,\dots,\theta_t}\otimes v_1\otimes\dots\otimes v_t\in V_1\circ \dots\circ V_t.
$$
Since $R_\theta 1_{\theta_1,\dots,\theta_t}$ is a free $R_{\theta_1}\otimes\dots\otimes R_{\theta_t}$-module of finite rank by Theorem~\ref{TBasis}, we get the following well-known properties:

\begin{Lemma}\label{circexact}
The functor $\Ind_{\theta_1,\dots,\theta_t}$	 is exact and sends finitely generated projectives to finitely generated projectives.
\end{Lemma}



The following lemma is easy to check using the fact that $1_\bk$ is an idempotent in $R_\chi$ for $\bk\in I^\chi_{\di}$:

\begin{Lemma} \label{LIndIdemp} 
For $\bi\in I^{\theta}_{\di}$ and $\bj\in I^{\eta}_{\di}$, we have
$R_\theta1_{\bi}\circ R_\eta 1_\bj\cong R_{\theta+\eta}1_{\bi\bj}$.
\end{Lemma}

\subsection{Standard modules}\label{SSStdMod}

The algebra $R_\theta$ is affine quasihereditary in the sense of \cite{Kdonkin}. In particular, it comes with an important class of {\em standard modules}, which we now describe explicitly, referring to \cite[\S3]{BKM}. We are going to work with the lexicographic convex order $\preceq$ on $\Phi_+$, i.e. for $\al = \alpha_r + \cdots + \alpha_s\in \Phi_+$ and $\al' = \alpha_{r'} + \cdots + \alpha_{s'}\in \Phi_+$, we have $\al \prec \al'$ if and only if either $r < r'$ or $r = r'$ and $s < s'$.

Fix $\al = \alpha_r + \cdots + \alpha_{s} \in \Phi_+$ of height $l:=s-r+1$, and set
$$\bi_\al := r\,(r+1)\cdots s\in I^\al.
$$
We define the $R_\al$-module $\Delta(\al)$ to be the cyclic $R_\al$-module generated by a vector $v_\al$ of degree $0$ with defining relations
\begin{itemize}
	\item $1_\bi v_\al = \delta_{\bi,\bi_\al}v_\al$ for all  $\bi\in I^\al$;
	\item $\psi_t v_\al = 0$ for all $1\leq t < l$;
	\item $y_t v_\al = y_u v_\al$ for all $1\leq t, u \leq l$.
\end{itemize}
By \cite[Corollary 3.5]{BKM}, $\De(\al)$ is indeed the standard module corresponding to $\al$. In fact, if $\k=\Z_p$, by \cite[(4.11)]{KS}, $\De(\al)$ is a universal $\k$-form of the standard module in the sense of \cite[\S4.2]{KS}.

The module $\Delta(\al)$ can be considered as an $(R_\al, \k[x])$-bimodule with the right action given by $v_\al x := y_1 v_\al.$ Then it is easy to check that there is an isomorphism of graded $\k$-modules $\k[x] \to \Delta(\al), \ f\mapsto v_\al f.$

Fix an integer $m\in \Z_{>0}$.
Let $\NH_m$ be the rank $m$ nil-Hecke algebra with standard generators $x_1,\dots,x_{m}$, $\tau_1,\dots,\tau_{m-1}$. For any $i\in I$, there is an isomorphism $\phi: R_{m\alpha_i}\iso \NH_m,\ y_t\mapsto x_t,\ \psi_u\mapsto \tau_u$, and we have the idempotent
\begin{equation} \label{DeltaIdem}
e_m := \phi(1_{i^{(d)}}') 
\in \NH_m.
\end{equation}

%

The $R_{m\al}$-module $\Delta(\al)^{\circ m}$ is cyclicly generated by $v_\al^{\circ m}$.
As explained in \cite[\S3.2]{BKM}, $\NH_m$ acts on $\Delta(\al)^{\circ m}$ on the right so that
\begin{align*}
v_\al^{\circ m}x_t&=v_\al^{\circ (t-1)}\circ (v_\al x)\circ v_\al^{\circ (m-t)}\qquad(1\leq t\leq m),
\\
v_\al^{\circ m}\tau_u&=v_\al^{\circ (u-1)}\circ (\psi_z (v_\al \circ v_\al))\circ v_\al^{\circ (m-u-1)}\qquad(1\leq u< m),
\end{align*}
where $z$ is the longest element of $\D^{(l,l)}$.

Define
\begin{equation}\label{DeltaDef}
\Delta(\al^m) := q^{\binom{m}{2}}\Delta(\al)^{\circ m}e_m.
\end{equation}
As in \cite[Lemma 3.10]{BKM}, we have an isomorphism of $R_{m\al}$-modules
\begin{equation}\label{deltaiso}
	\Delta(\al)^{\circ m} \cong [m]^!\Delta(\al^m).
\end{equation}

Given $\theta\in Q_+$ and a  Kostant partition $\pi = (\beta_1^{m_1}, \ldots, \beta_t^{m_t})\in\KP(\theta)$ as in the Introduction, we define the corresponding {\em standard module}
$$\Delta(\pi) := \Delta(\beta_1^{m_1}) \circ \cdots \circ \Delta(\beta_t^{m_t}).$$
If $\k$ is a field, the modules $\{\De(\pi)\mid \pi\in \KP(\theta)\}$ are the standard modules for an affine quasihereditary structure on the algebra $R_\theta$, see \cite{BKM,Kdonkin}. If $\k=\Z$ or $\Z_p$, they can be thought of as integral forms of the standard modules, see \cite[\S4]{KS}.

\section{Semicuspidal resolution}\label{SDPR}
Throughout the subsection, we fix $m\in\Z_{>0}$, $a,b\in\Z$ with $a\leq b$, and set $$l:=b+2-a,\quad d:=lm.$$
We denote
$$\al:=\al_a+\dots+\al_{b+1}\in \Phi_+$$
and $\theta:=m\al$. Note that $l=\height(\al)$ and $d=\height(\theta)$.
Our goal is to construct a resolution $P_\bullet=P^{\al^m}_{\bullet}$ of the semicuspidal standard module $\De(\al^m)$.

\subsection{Combinatorics}\label{SSComb}
We consider the set of compositions
$$\La=\La^{\al^m}:=\{\la=(\la_a,\dots,\la_b)\mid  \la_a,\dots,\la_b\in [0,m]\}.$$
For $\la\in\La$, we denote $|\la|:=\la_a+\dots+\la_b$, and for $n\in\Z_{\geq 0}$, we set
$$\La(n)=\La^{\al^m}(n):=\{\la\in\La \mid |\la|=n\}.$$
Let  $a\leq i\leq b$. We set
$$
\e_i:=(0,\dots,0,1,0,\dots,0)\in\La(1),
$$
with $1$ in the $i$th position.

Let $\la\in\La$. Set
\begin{align*}
s_\la&:=-\frac{lm(m-1)}{2}+(m+1)n-\sum_{i=a}^b\la_i^2\in\Z,\\
\bj^\la&:=a^{m-\la_a}(a+1)^{m-\la_{a+1}}\cdots b^{m-\la_b}(b+1)^{m}b^{\la_b} (b-1)^{\la_{b-1}}\cdots a^{\la_a}\in I^\theta,
\\
\bi^\la&:=a^{(m-\la_a)}(a+1)^{(m-\la_{a+1})}\dots b^{(m-\la_b)}(b+1)^{(m)}b^{(\la_b)} (b-1)^{(\la_{b-1})}\dots a^{(\la_a)}\in I^\theta_\di.
\end{align*}
We also associate to $\la$ a composition $\om_\la$ of $d$ with $2n+1$ non-negative parts: 
\begin{equation*}
\om_\la
:=(m-\la_a,m-\la_{a+1},\dots,m-\la_b,m,\la_b,\la_{b-1},\dots,\la_a).
\end{equation*}
Let  $i\in[a,b]$. We denote
\begin{align*}
r_i^-(\la)&:=\sum_{s=a}^{i}(m-\la_s),\quad
 r_i^+(\la):=d-\sum_{s=a}^{i-1}\la_s, \quad
l_i^\pm(\la):=r_{i\pm1}^\pm(\la)+1,
\end{align*}
where $r_{a-1}^-(\la)$ is interpreted as $0$, and $r_{b+1}^+(\la)$ is interpreted as $d-\sum_{s=a}^{b}\la_s$.
Moreover, denote
\begin{align*}
r_{b+1}(\la)&:=d-\sum_{s=a}^{b}\la_s,\quad
l_{b+1}(\la):=r_{b+1}(\la)-m+1.
\end{align*}
Define
\begin{align*}
U_i^\pm(\la)&:=[l_i^\pm(\la),r_i^\pm(\la)],\   U_i(\la):=U_i^-(\la)\sqcup U_i^+(\la), \   U_{b+1}(\la)&:=[l_{b+1}(\la),r_{b+1}(\la)].
\end{align*}
Observe that for all $j\in[a,b+1]$, we have
$$
U_j(\la)=\{s\in [1,d]\mid \bj^\la_s=j\}.
$$

We also consider the sets of multicompositions
\begin{align*}
\bLa&:=(\La^\al)^m=\{\bde=(\de^{(1)},\dots,\de^{(m)})\mid \de^{(1)},\dots,\de^{(m)}\in\La^\al\},
\\
\bLa(n)&:=\{\bde=(\de^{(1)},\dots,\de^{(m)})\in\bLa\mid |\de^{(1)}|+\dots+|\de^{(m)}|=n\}.
\end{align*}
Note that by definition all $\de^{(r)}_i\in\{0,1\}$.
For $1\leq r\leq m$ and $a\leq i\leq b$, we define ${\bolde}^{r}_i\in\bLa(1)$ to be the multicomposition whose $r$th component is $\e_i$ and whose other components are zero.
For $\bde\in\bLa$, denote
$$\la^\bde:=\de^{(1)}+\dots+\de^{(m)}\in\La.$$

Fix $\bde=(\de^{(1)},\dots,\de^{(m)})\in\bLa$. Define
$$\bj^\bde:=\bj^{\de^{(1)}}\dots\bj^{\de^{(m)}}\in I^\theta.$$
For $i\in[a,b]$  we define
\begin{align*}
U_i^{(\pm)}(\bde)&:=\{l(r-1)+u\mid \text{$r\in[1,m]$, $u\in U_i^{(\pm)}(\de^{(r)})$}\},\quad U_i(\bde):=U_i^+(\bde)\sqcup U_i^-(\bde),\\ U_{b+1}(\bde)&:=\{l(r-1)+u\mid \text{$r\in[1,m]$, $u\in U_{b+1}(\de^{(r)})$}\}.
\end{align*}
Observe that for any $j\in[a,b+1]$, we have
$$
U_j^{(\bde)}=\{s\in[1,d]\mid \bj^\bde_s=j\} \quad \text{and}\quad
|U_j^{(\pm)}(\la^\bde)|=|U_j^{(\pm)}(\bde)|.
$$

For $\la\in\La$, $\bde\in\bLa$, and $i\in[a,b]$, we define some signs:
\begin{align*}
\sgn_{\la;i}&:=(-1)^{\sum_{j=a}^{i-1}\la_{j}},
\qquad
\sgn_{\bde;r,i}:=(-1)^{\sum_{s=1}^{r-1}|\de^{(s)}|+\sum_{j=a}^{i-1}\de^{(r)}_j},
\\
t_\bde&:= \sum_{\substack{1\leq r<s\leq m \\
a\leq j<i\leq b }}\de^{(r)}_i\de^{(s)}_j,\qquad \si_\bde:=(-1)^{t_\bde},\\
\tau_\la&:=(-1)^{\sum_{i=a}^b\la_i(\la_i-1)/2},\qquad \tau_\bde:=\si_\bde\tau_{\la^\bde}.
\end{align*}
($\tau_\la,\tau_\bde$ are not to be confused with $\tau_w\in\NH_d$ which will not be used again).

\begin{Lemma} \label{LSigns} 
Let $\bde=(\de^{(1)},\dots,\de^{(m)})\in\bLa$, $i\in[a,b]$ and $r\in[1,m]$. If $\de_i^{(r)}=0$, then
$$
\sgn_{\la^\bde;i}\si_\bde=\si_{\bde+\bolde_i^{r}} \sgn_{\bde;r,i}(-1)^{\sum_{s=1}^{r-1}\de_i^{(s)}}.
$$
\end{Lemma}
\begin{proof}
Let $\la:=\la^\bde$ and $\bga:=\bde+\bolde_i^{r}$.
Writing `$\equiv$' for `$\equiv\pmod{2}$', we have to prove
\begin{align*}
\sum_{j=a}^{i-1}\la_j+\sum_{t<s,\ k>j}\de_k^{(t)}\de_j^{(s)}\equiv
\sum_{t<s,\ k>j}\ga_k^{(t)}\ga_j^{(s)}+\sum_{s=1}^{r-1}|\de^{(s)}|+\sum_{j=a}^{i-1}\de_j^{(r)}+\sum_{s=1}^{r-1}\de_i^{(s)}.
\end{align*}
Note that
\begin{equation}\label{E170916}
\sum_{t<s,\ k>j}\ga_k^{(t)}\ga_j^{(s)}=\sum_{t<s,\ k>j}\de_k^{(t)}\de_j^{(s)}+\sum_{s>r,\ j<i}\de_j^{(s)}+\sum_{s<r,\ j>i}\de_j^{(s)},
\end{equation}
so the required comparison boils down to
\begin{align*}
\sum_{j=a}^{i-1}\la_j\equiv
\sum_{s>r,\ j<i}\de_j^{(s)}+\sum_{s<r,\ j>i}\de_j^{(s)}+\sum_{s=1}^{r-1}|\de^{(s)}|+\sum_{j=a}^{i-1}\de_j^{(r)}+\sum_{s=1}^{r-1}\de_i^{(s)},
\end{align*}
which is easy to see.
\end{proof}


\begin{Lemma} \label{LSignsG} 
Let $\bga=(\ga^{(1)},\dots,\ga^{(m)})\in\bLa$, $i\in[a,b]$ and $r\in[1,m]$. If $\ga_i^{(r)}=1$, then
$$
\tau_\bga \sgn_{\la^\bga;i}=\sgn_{\bga-\bolde_i^{r};r,i} \tau_{\bga-\bolde_i^{r}}(-1)^{\sum_{s=r+1}^m \ga_i^{(s)}\}}.
$$
\end{Lemma}
\begin{proof}
Let $\mu:=\la^\bga$, $\bde:=\bga-\bolde_i^{r}$, and $\la:=\la^\bde$. Writing `$\equiv$' for `$\equiv\pmod{2}$', we have to prove the comparison
\begin{align*}
&\sum_{t<s,\ k>j}\ga_k^{(t)}\ga_j^{(s)}+\sum_{j=a}^b\mu_j(\mu_j-1)/2+\sum_{j=a}^{i-1}\mu_j
\\
\equiv
&
\sum_{s=1}^{r-1}|\de^{(s)}|+\sum_{j=a}^{i-1}\de^{(r)}_j
+\sum_{t<s,\ k>j}\de_k^{(t)}\de_j^{(s)}+\sum_{j=a}^b\la_j(\la_j-1)/2
+\sum_{s=r+1}^m \ga_i^{(s)}.
\end{align*}
Note that
$$
\sum_{j=a}^b\mu_j(\mu_j-1)/2-\sum_{j=a}^b\la_j(\la_j-1)/2=\la_i.
$$
So, using also (\ref{E170916}), the required comparison boils down to
\begin{align*}
\sum_{s>r,\ j<i}\de_j^{(s)}+\sum_{s<r,\ j>i}\de_j^{(s)}
+\sum_{j=a}^{i}\la_j\equiv
\sum_{s=1}^{r-1}|\de^{(s)}|+\sum_{j=a}^{i-1}\de_j^{(r)}+\sum_{s=r+1}^{m}\de_i^{(s)},
\end{align*}
which is easy to see.
\end{proof}


\subsection{The resolutions $P_\bullet^{\al^m}$ and $P_\bullet^{\pi}$}\label{SSRes}
Let $\la\in\La$. Recalling the divided power word $\bi^\la\in I^\theta_\di$ and the integer $s_\la$ from $\S\ref{SSComb}$, we set
\begin{align*}
e_\la:=1_{\bi^\la}\in R_\theta \qquad\text{and}\qquad
P_\la=P^{\al^m}_{\la}:=q^{s_\la}R_\theta e_\la.
\end{align*}
In particular, $P_\la$ is a projective left $R_\theta$-module.
Note that $1_{\bj^\la}e_\la=e_\la1_{\bj^\la}=e_\la$.
Further, set for any $n\in\Z_{\geq 0}$:
$$
P_n=P^{\al^m}_{n}:=\bigoplus_{\la\in\La(n)} P_\la.
$$
Note that $P_n=0$ for $n>d-m$. The projective resolution $P_\bullet=P^{\al^m}_\bullet$ of $\De(\al^m)$ will be of the form
$$
\dots \longrightarrow P_{n+1}\stackrel{d_n}{\longrightarrow} P_n\longrightarrow\dots \stackrel{d_0}{\longrightarrow} P_0.
$$

To describe the boundary maps $d_n$, we first consider a more general situation. Suppose we are given two sets of idempotents $\{e_a\mid a\in A\}$ and $\{f_b\mid b\in B\}$ in an algebra $R$. An $A\times B$ matrix $D:=(d^{a,b})_{a\in A, b\in B}$ with every $d^{a,b}\in e_aR f_b$ then yields the homomorphism between the projective $R$-modules
\begin{equation}\label{rightmult}
\rho_D:\bigoplus_{a\in A}Re_a\to \bigoplus_{b\in B}Rf_b,\ (r_ae_a)_{a\in A}\mapsto \big(\sum_{a\in A}r_ad^{a,b}\big)_{b\in B},
\end{equation}
which we refer to as the {\em right multiplication with $D$.}

We now define a $\La(n+1)\times\La(n)$ matrix $D_n$ with entries $d_n^{\mu,\la}\in e_\mu R_\theta e_\la$. Let $\la\in\La(n)$ and $a\leq i\leq b$ be such that $\la_i<m$. Recalling (\ref{EPsiCycle}), define
\begin{align*}
\psi_{\la;i}:=\psi_{r_i^+(\la+\e_i)\to r_i^-(\la)}.
\end{align*}
Note that $\psi_{\la;i}1_{\bj^\la}=1_{\bj^{\la+\e_i}}\psi_{\la;i}$.
Recalling the sign $\sgn_{\la;i}$ from \S\ref{SSComb}, we now set
$$
d_n^{\mu,\la}:=
\left\{
\begin{array}{ll}
\sgn_{\la;i}e_\mu\psi_{\la;i} e_\la &\hbox{if $\mu=\la+\e_i$ for some $a\leq i\leq b$,}\\
0 &\hbox{otherwise.}
\end{array}
\right.
$$
Diagrammatically, for $\mu=\la+\e_i$ as above, we have
$$
d^{\mu,\la}_n=\pm
\begin{braid}\tikzset{baseline=0mm}
\draw(1.05,2) node[above]{\small $a^{m-\la_a}$};
\draw [rounded corners,color=gray] (-0.6,3.6)--(2.6,3.6)--(2.6,2.2)--(-0.6,2.2)--cycle;
\draw(4,2) node[above]{\small $\cdots$};
\draw(7,2) node[above]{\small $i^{m-\la_i-1}$};
\draw [rounded corners,color=gray] (4.8,3.6)--(9,3.6)--(9,2.2)--(4.8,2.2)--cycle;
\draw(12,2) node[above]{\small $(i+1)^{\la_{i+1}}$};
\draw [rounded corners,color=gray] (9.3,3.6)--(14.5,3.6)--(14.5,2.2)--(9.3,2.2)--cycle;
\draw(16.6,2) node[above]{\small $\cdots$};
\draw(20,2) node[above]{\small $i^{\la_i+1}$};
\draw [rounded corners,color=gray] (18.2,3.6)--(21.9,3.6)--(21.9,2.2)--(18.2,2.2)--cycle;
\draw(25,2) node[above]{\small $(i-1)^{\la_{i-1}}$};
\draw [rounded corners,color=gray] (22.3,3.6)--(27.5,3.6)--(27.5,2.2)--(22.3,2.2)--cycle;
\draw(29,2) node[above]{\small $\cdots$};
\draw(31.4,2) node[above]{\small $a^{\la_a}$};
\draw [rounded corners,color=gray] (30,3.6)--(33,3.6)--(33,2.2)--(30,2.2)--cycle;

\draw(1.05,-3) node[above]{\small $a^{m-\la_a}$};
\draw [rounded corners,color=gray] (-0.6,-1.4)--(2.6,-1.4)--(2.6,-2.8)--(-0.6,-2.8)--cycle;
\draw(4,-3) node[above]{\small $\cdots$};
\draw(7.4,-3) node[above]{\small $i^{m-\la_i}$};
\draw [rounded corners,color=gray] (4.8,-1.4)--(9.4,-1.4)--(9.4,-2.8)--(4.8,-2.8)--cycle;
\draw(12.4,-3) node[above]{\small $(i+1)^{\la_{i+1}}$};
\draw [rounded corners,color=gray] (9.7,-1.4)--(14.9,-1.4)--(14.9,-2.8)--(9.7,-2.8)--cycle;
\draw(17,-3) node[above]{\small $\cdots$};
\draw(20.4,-3) node[above]{\small $i^{\la_i}$};
\draw [rounded corners,color=gray] (18.6,-1.4)--(21.8,-1.4)--(21.8,-2.8)--(18.6,-2.8)--cycle;
\draw(25,-3) node[above]{\small $(i-1)^{\la_{i-1}}$};
\draw [rounded corners,color=gray] (22.3,-1.4)--(27.5,-1.4)--(27.5,-2.8)--(22.3,-2.8)--cycle;
\draw(29,-3) node[above]{\small $\cdots$};
\draw(31.4,-3) node[above]{\small $a^{\la_a}$};
\draw [rounded corners,color=gray] (30,-1.4)--(33,-1.4)--(33,-2.8)--(30,-2.8)--cycle;

\draw(0,2.2)--(0,-1.4);
\draw(1,-0.3) node[above]{\small $\cdots$};
\draw(2,2.2)--(2,-1.4);
\draw(3.9,-0.3) node[above]{\small $\cdots$};
\draw(5.2,2.2)--(5.2,-1.4);
\draw(7,-0.3) node[above]{\small $\cdots$};
\draw(8.5,2.2)--(8.5,-1.4);
\draw(21.6,2.2)--(9,-1.4);
\draw(9.6,2.2)--(10,-1.4);
\draw(12.2,-0.3) node[above]{\small $\cdots$};
\draw(14,2.2)--(14.4,-1.4);
\draw(18.6,2.2)--(19,-1.4);
\draw(19.9,-0.3) node[above]{\small $\cdots$};
\draw(20.8,2.2)--(21.2,-1.4);
\draw(22.8,2.2)--(22.8,-1.4);
\draw(25,-0.3) node[above]{\small $\cdots$};
\draw(27,2.2)--(27,-1.4);
\draw(29,-0.3) node[above]{\small $\cdots$};
\draw(30.4,2.2)--(30.4,-1.4);
\draw(31.7,-0.3) node[above]{\small $\cdots$};
\draw(32.7,2.2)--(32.7,-1.4);
\end{braid}
$$

We now set the boundary map $d_n$ to be the right multiplication with $D_n$:
$$d_n:=\rho_{D_n}.$$

\begin{Example} \label{ExP} 
{\rm
Let $a=b=1$ and $m=2$. Then the resolution $P_\bullet$ is
$$
0\to R_\theta1_{2^{(2)}1^{(2)}}\stackrel{d_1}{\longrightarrow}
R_\theta1_{12^{(2)}1}\stackrel{d_0}{\longrightarrow}
q^{-2} R_\theta1_{1^{(2)}2^{(2)}}\longrightarrow\De((\al_1+\al_2)^2)\longrightarrow 0,
$$
where $d_1$ is a right multiplication with
$$
1_{2^{(2)}1^{(2)}} \psi_3\psi_2\psi_1 1_{12^{(2)}1}=\begin{braid}\tikzset{baseline=0mm}
\draw(1,2) node[above]{\small $2\ 2$};
\draw [rounded corners,color=gray] (-0.2,3.5)--(2.2,3.5)--(2.2,2.2)--(-0.2,2.2)--cycle;
\draw(3.55,2) node[above]{\small $1\ 1$};
\draw [rounded corners,color=gray] (2.4,3.5)--(4.8,3.5)--(4.8,2.2)--(2.4,2.2)--cycle;

\draw(0.4,-3) node[above]{\small $1$};
\draw [rounded corners,color=gray] (1.2,-1.5)--(3.6,-1.5)--(3.6,-2.8)--(1.2,-2.8)--cycle;
\draw(1.8,-3) node[above]{\small $2$};
\draw(2.9,-3) node[above]{\small $2$};
\draw(4.2,-3) node[above]{\small $1$};

\draw(0.5,2.2)--(1.7,-1.5);
\draw(1.5,2.2)--(2.7,-1.5);
\draw(3,2.2)--(4.2,-1.5);
\draw(4,2.2)--(0.5,-1.5);
\end{braid}
$$
and $d_0$ is a right multiplication with
$$
1_{12^{(2)}1} \psi_3\psi_2 1_{1^{(2)}2^{(2)}}=\begin{braid}\tikzset{baseline=0mm}
\draw(0.4,2) node[above]{\small $1$};
\draw [rounded corners,color=gray] (1.2,3.5)--(3.6,3.5)--(3.6,2.2)--(1.2,2.2)--cycle;
\draw(1.8,2) node[above]{\small $2$};
\draw(2.9,2) node[above]{\small $2$};
\draw(4.2,2) node[above]{\small $1$};
\draw(0.4,2) node[above]{\small $1$};

\draw(1,-3) node[above]{\small $1\ 1$};
\draw [rounded corners,color=gray] (-0.2,-1.5)--(2.2,-1.5)--(2.2,-2.8)--(-0.2,-2.8)--cycle;
\draw(3.55,-3) node[above]{\small $2\ 2$};
\draw [rounded corners,color=gray] (2.4,-1.5)--(4.8,-1.5)--(4.8,-2.8)--(2.4,-2.8)--cycle;

\draw(0.5,2.2)--(0.5,-1.5);
\draw(1.7,2.2)--(2.9,-1.5);
\draw(2.9,2.2)--(4.1,-1.5);
\draw(4,2.2)--(1.5,-1.5);
\end{braid}
$$
}
\end{Example}

It is far from clear that $\ker d_n=\im d_{n+1}$ but at least the following is easy to see:

\begin{Lemma} 
The homomorphisms $d_n$ are homogeneous of degree $0$ for all $n$.
\end{Lemma}
\begin{proof}
Let $\la\in\La(n)$ be such that $\la_i<m$ for some $a\leq i\leq b$, so that $\mu:=\la+\e_i\in \La(n+1)$.
The homomorphism $d_n=\rho_{D_n}$ is a right multiplication with the matrix $D_n$. Its $(\mu,\la)$-component is a homomorphism
$
P_\mu\to P_\la
$ obtained by the right multiplication with $\pm e_\mu\psi_{\la;i}e_\la$. Recall that $P_\mu=q^{s_\mu}R_\theta e_\mu$ and $P_\la=q^{s_\la}R_\theta e_\la$. So we just need to show that $s_\mu=s_\la+\deg(e_\mu\psi_{\la;i}e_\la)$. This is an easy computation using the fact that by definition we have
$
\deg(e_\mu\psi_{\la;i}e_\la)=m-2\la_i
$.
\end{proof}


Let $\eta\in Q_+$ be arbitrary and $\pi = (\beta_1^{m_1}, \dots, \beta_t^{m_t})\in \KP(\eta)$. We will define the resolution $P_\bullet^\pi$ of $\Delta(\pi) = \Delta(\beta_1^{m_1})\circ\cdots \circ\Delta(\beta_t^{m_t})$ using the general notion of the induced product of chain complexes.


Let $(C_\bullet, d_C)$ and $(D_\bullet, d_D)$ be chain complexes of left $R_{\eta'}$-modules and $R_{\eta''}$-modules, respectively. We define a complex of $R_{\eta'+\eta''}$-modules
\begin{equation}\label{ECircProdRes}
(C_\bullet\circ D_\bullet)_n := \bigoplus_{p+q = n}C_p\circ D_q
\end{equation}
with differential given by
$$d_{C\circ D}: C_\bullet\circ D_\bullet \to C_\bullet\circ D_\bullet, \ x\circ y \mapsto x\circ d_D(y) + (-1)^q d_C(x)\circ y.$$

\begin{Lemma}\label{circresolution}
	If $C_\bullet$ and $D_\bullet$ are projective resolutions of modules $M$ and $N$, respectively, then $C_\bullet\circ D_\bullet$ is a projective resolution of $M\circ N$.
\end{Lemma}
\begin{proof}
This follows from Lemma \ref{circexact} and \cite[Lemma 2.7.3]{Weibel}.
\end{proof}

We now define
\begin{equation}\label{fullresolution}
	P_\bullet^\pi := P_\bullet^{\beta_1^{m_1}}\circ\cdots\circ P_\bullet^{\beta_t^{m_t}}
\end{equation}
which, by Lemma~\ref{circresolution}, will turn out to be a projective resolution of $\Delta(\pi)$ in view of Theorem~\ref{main theorem}.

\subsection{The resolution $Q_\bullet$}\label{SQdef}
In order to check that $P_\bullet$ is a resolution of $\De(\al^m)$, we show that it is a direct summand of a known resolution $Q_\bullet$ of $q^{m(m-1)/2}\De(\al)^{\circ m}$. To describe the latter resolution, let us first consider the special case $m=1$.

\begin{Lemma} \label{L070417} 
We have that $P^\al_{\bullet}$ is a resolution of $\De(\al)$.
\end{Lemma}
\begin{proof}
This is a special case of \cite[Theorem 4.12]{BKM}, corresponding to the standard choice of $(\al_{i+1}+\dots+\al_j,\al_i)$ as the minimal pair for an arbitrary positive root $\al_i+\dots+\al_j$ in the definition of $\bi_{\al,\si}$, see \cite[\S4.5]{BKM}.
\end{proof}

Let $Q_\bullet$ be the resolution $q^{m(m-1)/2} (P^\al_{\bullet})^{\circ m}$, cf. (\ref{ECircProdRes}). To describe $Q_\bullet$ more explicitly,
let $n\in\Z_{\geq 0}$ and $\bde=(\de^{(1)},\dots,\de^{(m)})\in\bLa(n)$.
Recalling the definitions of \S\ref{SSComb}, we set
\begin{align*}
e_\bde:=1_{\bj^\bde}\in R_\theta,\quad
Q_\bde:=q^{n+m(m-1)/2}R_\theta e_\bde.
\end{align*}
For example taking each $\de^{(r)}$ to be $0\in\La(0)$, we get
\begin{equation}\label{EBZero}
\bde={\bzero}:=(0,\dots,0)\in\bLa(0)
\end{equation}
 and
$e_{\bzero}=1_{(a\,(a+1)\cdots (b+1))^m}.$
Further for any $n\in\Z_{\geq 0}$, we set
$$
Q_n:=\bigoplus_{\bde\in\bLa(n)}Q_\bde.
$$

The projective resolution $Q_\bullet$ is
$$
\dots \longrightarrow Q_{n+1}\stackrel{c_n}{\longrightarrow} Q_n\longrightarrow\dots \stackrel{c_0}{\longrightarrow} Q_0\stackrel{\sf q}{\longrightarrow} q^{m(m-1)/2}\De(\al)^{\circ m}\longrightarrow 0,
$$
with the augmentation map
\begin{equation}\label{EAugMap}
{\sf q}:Q_0\to q^{m(m-1)/2}\Delta(\alpha)^{\circ m},\ xe_{\bzero}\mapsto xv_\alpha^{\circ m},
\end{equation}
see \S\ref{SSStdMod}, and $c_n$ being right multiplication with the $\bLa(n+1)\times\bLa(n)$ matrix $C=(c_n^{\bga,\bde})$ defined as follows.
If $\bde+\bolde^{r}_i\in\bLa$ for some $r\in [1, m]$ and $i\in[a, b]$, i.e. $\de^{(r)}_i=0$, we set
\begin{equation}\label{EQ123}
\psi_{\bde;r,i}:=\iota_{\al,\dots,\al}(1^{\otimes (r-1)}\otimes \psi_{\de^{(r)};i}\otimes 1^{\otimes (m-r)})=1^{\circ (r-1)}\circ \psi_{\de^{(r)};i}\circ 1^{\circ (m-r)}.
\end{equation}
Recalling the signs defined in \S\ref{SSComb}, for $\bde\in\bLa(n)$ and $\bga\in\bLa(n+1)$, we now define
$$
c_n^{\bga,\bde}=
\left\{
\begin{array}{ll}
\sgn_{\bde;r,i}e_\bga\psi_{\bde;r,i}e_\bde &\hbox{if $\bga=\bde+\bolde^{r}_i$ for some $1\leq r\leq m$ and $a\leq i\leq b$,}\\
0 &\hbox{otherwise.}
\end{array}
\right.
$$

The fact that $Q_\bullet$ is indeed isomorphic to the resolution $q^{m(m-1)/2} (P^\al_{\bullet})^{\circ m}$ is easily checked using the isomorphism $R_\theta e_\bde\cong R_\theta 1_{\bj^{\de^{(1)}}}\circ\dots\circ R_\theta1_{\bj^{\de^{(m)}}}$, which comes from Lemma~\ref{LIndIdemp}.

\begin{Example} \label{ExQ} 
{\rm
Let $a=b=1$ and $m=2$. Then the resolution $Q_\bullet$ is
$$
0\to q^3R_\theta1_{2121}\stackrel{c_1}{\longrightarrow}
q^2R_\theta1_{2112}\oplus q^2R_\theta1_{1221}\stackrel{c_0}{\longrightarrow}
q R_\theta1_{1212}\stackrel{\sf q}{\longrightarrow} q\De(\al_1+\al_2)^{\circ 2}\longrightarrow 0,
$$
where $c_1$ is a right multiplication with the matrix
$
(-1_{2121} \psi_3\ \  1_{2121} \psi_1)
$,
and $c_0$ is a right multiplication with the matrix
$
\left(
\begin{matrix}
 1_{2112} \psi_1   \\
 1_{1221} \psi_3
\end{matrix}
\right).
$
}
\end{Example}

\subsection{Comparison maps} \label{SSComp}
We now construct what will end up being a pair of chain maps
$f:P_\bullet\to Q_\bullet$ and $g:Q_\bullet\to P_\bullet$ with $g\circ f=\id$. As usual, $f_n$ and $g_n$ will be given as right multiplications with certain matrices $F_n$ and $G_n$, respectively.

Let $\la\in\La$. Recall the definitions of \S\ref{SSComb}. We denote by $w^\la_0$ the longest element of the parabolic subgroup $\Si_{\om_\la}\leq \Si_d$. We also denote
$$
y^\la:=1_{a^{m-\la_a}\cdots b^{m-\la_b}}\circ y_{0,m}\circ  1_{b^{\la_b}\cdots a^{\la_a}}.
$$

Let $\bde=(\de^{(1)},\dots,\de^{(m)})\in\bLa$.
We define $u(\bde)\in\Si_d$ as follows: for all $i=a,\dots,b$, the permutation $u(\bde)$ maps:
\begin{enumerate}
\item[$\bullet$] the elements of $U^\pm_i(\la^\bde)$ increasingly to the elements of $U^\pm_i(\bde)$;
\item[$\bullet$] the elements of $U_{b+1}(\la^\bde)$ increasingly to the elements of $U_{b+1}(\bde)$.
\end{enumerate}
Set $w(\bde):=u(\bde)^{-1}$.
Then $w(\bde)$ can also be characterized as the element of $\Si_d$ which for all $i=a,\dots,b$, maps
the elements of $U^{(\pm)}_i(\bde)$ increasingly to the elements of $U^{(\pm)}_i(\la^\bde)$ and the elements of $U_{b+1}(\bde)$ increasingly to the elements of $U_{b+1}(\la^\bde)$.

Recall the signs $\si_\bde$ and $\tau_{\bde}$ defined in \S\ref{SSComb}.
We now define $F_n$ as the $\La(n)\times \bLa(n)$-matrix with the entries $f_n^{\la,\bde}$ defined for any $\la\in\La(n),\bde\in\bLa(n)$ as follows:
$$
f_n^{\la,\bde}:=
\left\{
\begin{array}{ll}
\si_{\bde}e_\la\psi_{w_0^\la}\psi_{w(\bde)}e_\bde &\hbox{if $\la=\la^\bde$,}\\
0 &\hbox{otherwise.}
\end{array}
\right.
$$
We define $G_n$ as the $\bLa(n)\times \La(n)$-matrix with the entries $g_n^{\bde,\la}$ defined for any $\bde\in\bLa(n), \la\in\La(n)$ as follows:
$$
g_n^{\bde,\la}:=
\left\{
\begin{array}{ll}
\tau_\bde e_\bde\psi_{u(\bde)}y^\la e_\la &\hbox{if $\la=\la^\bde$,}\\
0 &\hbox{otherwise.}
\end{array}
\right.
$$

\begin{Example} Let $m=d=2$ as in Examples~\ref{ExP} and \ref{ExQ}.  Then:
\begin{align*}
F_0
&=
\left(\begin{braid}\tikzset{baseline=0mm}
\draw(1,2) node[above]{\small $1\ 1$};
\draw [rounded corners,color=gray] (-0.2,3.5)--(2.2,3.5)--(2.2,2.2)--(-0.2,2.2)--cycle;
\draw(1,0) node[above]{\small $w_0$};
\draw [rounded corners,color=gray] (-0.2,1.5)--(2.2,1.5)--(2.2,0.2)--(-0.2,0.2)--cycle;
\draw(0.5,2.2)--(0.5,1.5);
\draw(1.5,2.2)--(1.5,1.5);
\draw(3.55,2) node[above]{\small $2\ 2$};
\draw [rounded corners,color=gray] (2.4,3.5)--(4.8,3.5)--(4.8,2.2)--(2.4,2.2)--cycle;
\draw(3.9,0) node[above]{\small $w_0$};
\draw [rounded corners,color=gray] (2.4,1.5)--(4.8,1.5)--(4.8,0.2)--(2.4,0.2)--cycle;
\draw(3.05,2.2)--(3.05,1.5);
\draw(4.05,2.2)--(4.05,1.5);
\draw(0.4,-3) node[above]{\small $1$};
\draw(1.8,-3) node[above]{\small $2$};
\draw(2.9,-3) node[above]{\small $1$};
\draw(4.2,-3) node[above]{\small $2$};
\draw(0.5,0.2)--(0.5,-1.5);
\draw(1.5,0.2)--(2.8,-1.5);
\draw(3,0.2)--(1.8,-1.5);
\draw(4,0.2)--(4,-1.5);
\end{braid}\right)
=
\left(\begin{braid}\tikzset{baseline=0mm}
\draw(1,2) node[above]{\small $1\ 1$};
\draw [rounded corners,color=gray] (-0.2,3.5)--(2.2,3.5)--(2.2,2.2)--(-0.2,2.2)--cycle;
\draw(3.55,2) node[above]{\small $2\ 2$};
\draw [rounded corners,color=gray] (2.4,3.5)--(4.8,3.5)--(4.8,2.2)--(2.4,2.2)--cycle;
\draw(0.4,-3) node[above]{\small $1$};
\draw(1.8,-3) node[above]{\small $2$};
\draw(2.9,-3) node[above]{\small $1$};
\draw(4.2,-3) node[above]{\small $2$};
\draw(0.5,2.2)--(2.9,-1.5);
\draw(1.5,2.2)--(0.5,-1.5);
\draw(3,2.2)--(4.2,-1.5);
\draw(4,2.2)--(1.7,-1.5);
\end{braid}\right),
\\
F_1&
=\left(\begin{braid}\tikzset{baseline=0mm}
\draw(0.4,2) node[above]{\small $1$};
\draw [rounded corners,color=gray] (1.2,3.5)--(3.6,3.5)--(3.6,2.2)--(1.2,2.2)--cycle;
\draw(1.8,2) node[above]{\small $2$};
\draw(2.9,2) node[above]{\small $2$};
\draw(2.4,0) node[above]{\small $w_0$};
\draw [rounded corners,color=gray] (1.2,1.5)--(3.6,1.5)--(3.6,0.2)--(1.2,0.2)--cycle;
\draw(4.2,2) node[above]{\small $1$};
\draw(0.4,-3) node[above]{\small $1$};
\draw(1.8,-3) node[above]{\small $2$};
\draw(2.9,-3) node[above]{\small $2$};
\draw(4.2,-3) node[above]{\small $1$};
\draw(0.5,2.2)--(0.5,-1.5);
\draw(1.75,2.2)--(1.75,1.5);
\draw(1.75,0.2)--(1.75,-1.5);
\draw(2.9,2.2)--(2.9,1.5);
\draw(2.9,0.2)--(2.9,-1.5);
\draw(4.2,2.2)--(4.2,-1.5);
\end{braid}\quad
\begin{braid}\tikzset{baseline=0mm}
\draw(0.4,2) node[above]{\small $1$};
\draw [rounded corners,color=gray] (1.2,3.5)--(3.6,3.5)--(3.6,2.2)--(1.2,2.2)--cycle;
\draw(1.8,2) node[above]{\small $2$};
\draw(2.9,2) node[above]{\small $2$};
\draw(2.4,0) node[above]{\small $w_0$};
\draw [rounded corners,color=gray] (1.2,1.5)--(3.6,1.5)--(3.6,0.2)--(1.2,0.2)--cycle;
\draw(4.2,2) node[above]{\small $1$};
\draw(0.4,-3) node[above]{\small $2$};
\draw(1.8,-3) node[above]{\small $1$};
\draw(2.9,-3) node[above]{\small $1$};
\draw(4.2,-3) node[above]{\small $2$};
\draw(0.5,2.2)--(1,0)--(2.9,-1.5);
\draw(1.75,2.2)--(1.75,1.5);
\draw(1.75,0.2)--(0.5,-1.5);
\draw(3,2.2)--(3,1.5);
\draw(3,0.2)--(4.2,-1.5);
\draw(4.2,2.2)--(3.8,0)--(1.7,-1.5);
\end{braid}
\right)
=
\left(\begin{braid}\tikzset{baseline=0mm}
\draw(0.4,2) node[above]{\small $1$};
\draw [rounded corners,color=gray] (1.2,3.5)--(3.6,3.5)--(3.6,2.2)--(1.2,2.2)--cycle;
\draw(1.8,2) node[above]{\small $2$};
\draw(2.9,2) node[above]{\small $2$};
\draw(4.2,2) node[above]{\small $1$};
\draw(0.4,-3) node[above]{\small $1$};
\draw(1.8,-3) node[above]{\small $2$};
\draw(2.9,-3) node[above]{\small $2$};
\draw(4.2,-3) node[above]{\small $1$};
\draw(0.5,2.2)--(0.5,-1.5);
\draw(1.75,2.2)--(2.9,-1.5);
\draw(2.95,2.2)--(1.7,-1.5);
\draw(4.2,2.2)--(4.2,-1.5);
\end{braid}\quad
\begin{braid}\tikzset{baseline=0mm}
\draw(0.4,2) node[above]{\small $1$};
\draw [rounded corners,color=gray] (1.2,3.5)--(3.6,3.5)--(3.6,2.2)--(1.2,2.2)--cycle;
\draw(1.8,2) node[above]{\small $2$};
\draw(2.9,2) node[above]{\small $2$};
\draw(4.2,2) node[above]{\small $1$};
\draw(0.4,-3) node[above]{\small $2$};
\draw(1.8,-3) node[above]{\small $1$};
\draw(2.9,-3) node[above]{\small $1$};
\draw(4.2,-3) node[above]{\small $2$};
\draw(0.5,2.2)--(2.9,-1.5);
\draw(1.75,2.2)--(4.2,-1.5);
\draw(3,2.2)--(0.5,-1.5);
\draw(4.2,2.2)--(1.7,-1.5);
\end{braid}
\right),
\\
F_2&=
\left(\begin{braid}\tikzset{baseline=0mm}
\draw(1,2) node[above]{\small $2\ 2$};
\draw [rounded corners,color=gray] (-0.2,3.5)--(2.2,3.5)--(2.2,2.2)--(-0.2,2.2)--cycle;
\draw(3.55,2) node[above]{\small $1\ 1$};
\draw [rounded corners,color=gray] (2.4,3.5)--(4.8,3.5)--(4.8,2.2)--(2.4,2.2)--cycle;
\draw(0.4,-3) node[above]{\small $2$};
\draw(1.8,-3) node[above]{\small $1$};
\draw(2.9,-3) node[above]{\small $2$};
\draw(4.1,-3) node[above]{\small $1$};
\draw(1,0) node[above]{\small $w_0$};
\draw [rounded corners,color=gray] (-0.2,1.5)--(2.2,1.5)--(2.2,0.2)--(-0.2,0.2)--cycle;
\draw(3.9,0) node[above]{\small $w_0$};
\draw [rounded corners,color=gray] (2.4,1.5)--(4.8,1.5)--(4.8,0.2)--(2.4,0.2)--cycle;
\draw(3.05,2.2)--(3.05,1.5);
\draw(4.05,2.2)--(4.05,1.5);
\draw(0.5,2.2)--(0.5,1.5);
\draw(1.5,2.2)--(1.5,1.5);
\draw(0.5,0.2)--(0.5,-1.5);
\draw(1.5,0.2)--(2.9,-1.5);
\draw(3,0.2)--(1.7,-1.5);
\draw(4,0.2)--(4,-1.5);
\end{braid}\right)
=
\left(\begin{braid}\tikzset{baseline=0mm}
\draw(1,2) node[above]{\small $2\ 2$};
\draw [rounded corners,color=gray] (-0.2,3.5)--(2.2,3.5)--(2.2,2.2)--(-0.2,2.2)--cycle;
\draw(3.55,2) node[above]{\small $1\ 1$};
\draw [rounded corners,color=gray] (2.4,3.5)--(4.8,3.5)--(4.8,2.2)--(2.4,2.2)--cycle;
\draw(0.4,-3) node[above]{\small $2$};
\draw(1.8,-3) node[above]{\small $1$};
\draw(2.9,-3) node[above]{\small $2$};
\draw(4.2,-3) node[above]{\small $1$};
\draw(0.5,2.2)--(2.9,-1.5);
\draw(1.5,2.2)--(0.5,-1.5);
\draw(3,2.2)--(4.2,-1.5);
\draw(4,2.2)--(1.7,-1.5);
\end{braid}\right),
\\
G_0&
=\left(
\begin{braid}\tikzset{baseline=0mm}
\draw(0.4,2) node[above]{\small $1$};
\draw(1.6,2) node[above]{\small $2$};
\draw(2.9,2) node[above]{\small $1$};
\draw(4.2,2) node[above]{\small $2$};
\draw(1,-3) node[above]{\small $1\ 1$};
\draw [rounded corners,color=gray] (-0.2,-2.8)--(2.2,-2.8)--(2.2,-1.5)--(-0.2,-1.5)--cycle;
\draw(3.55,-3) node[above]{\small $2\ 2$};
\draw [rounded corners,color=gray] (2.4,-2.8)--(4.8,-2.8)--(4.8,-1.5)--(2.4,-1.5)--cycle;
\draw(3.7,-1.1) node[above]{\small $y_0$};
\draw [rounded corners,color=gray] (2.4,0.4)--(4.8,0.4)--(4.8,-1)--(2.4,-1)--cycle;
\draw(0.5,2.2)--(0.5,-1.5);
\draw(1.5,2.2)--(2.9,0.4);
\draw(2.9,-1)--(2.9,-1.5);
\draw(3,2.2)--(2.3,1.3)--(1.5,-1.5);
\draw(4.2,2.2)--(4.2,0.4);
\draw(4.2,-1)--(4.2,-1.5);
\end{braid}\right)
=\left(
\begin{braid}\tikzset{baseline=0mm}
\draw(0.4,2) node[above]{\small $1$};
\draw(1.6,2) node[above]{\small $2$};
\draw(2.9,2) node[above]{\small $1$};
\draw(4,2) node[above]{\small $2$};
\draw(1,-3) node[above]{\small $1\ 1$};
\draw [rounded corners,color=gray] (-0.2,-2.8)--(2.2,-2.8)--(2.2,-1.5)--(-0.2,-1.5)--cycle;
\draw(3.55,-3) node[above]{\small $2\ 2$};
\draw [rounded corners,color=gray] (2.4,-2.8)--(4.8,-2.8)--(4.8,-1.5)--(2.4,-1.5)--cycle;
\draw(0.5,2.2)--(0.5,-1.5);
\draw(1.5,2.2)--(3,-1.5);
\draw(2.95,2.2)--(1.5,-1.5);
\draw(4,2.2)--(4,-1.5);
\blackdot(4,-0.5);
\end{braid}\right),
\\
G_1&
=\left(\begin{array}c
\begin{braid}\tikzset{baseline=0mm}
\draw(0.4,2) node[above]{\small $1$};
\draw(1.7,2) node[above]{\small $2$};
\draw(2.9,2) node[above]{\small $2$};
\draw(4.2,2) node[above]{\small $1$};
\draw(0.4,-3) node[above]{\small $1$};
\draw [rounded corners,color=gray] (1.2,-2.8)--(3.6,-2.8)--(3.6,-1.5)--(1.2,-1.5)--cycle;
\draw(1.8,-3) node[above]{\small $2$};
\draw(2.9,-3) node[above]{\small $2$};
\draw(4.2,-3) node[above]{\small $1$};
\draw(2.4,-1.1) node[above]{\small $y_0$};
\draw [rounded corners,color=gray] (1.2,0.4)--(3.6,0.4)--(3.6,-1)--(1.2,-1)--cycle;
\draw(0.5,2.2)--(0.5,-1.5);
\draw(1.7,2.2)--(1.7,0.4);
\draw(2.95,2.2)--(2.95,0.4);
\draw(1.7,-1)--(1.7,-1.5);
\draw(2.95,-1)--(2.95,-1.5);
\draw(4.2,2.2)--(4.2,-1.5);
\end{braid}\\
\begin{braid}\tikzset{baseline=0mm}
\draw(0.4,2) node[above]{\small $2$};
\draw(1.6,2) node[above]{\small $1$};
\draw(2.9,2) node[above]{\small $1$};
\draw(4.2,2) node[above]{\small $2$};
\draw(0.4,-3) node[above]{\small $1$};
\draw [rounded corners,color=gray] (1.2,-2.8)--(3.6,-2.8)--(3.6,-1.5)--(1.2,-1.5)--cycle;
\draw(1.8,-3) node[above]{\small $2$};
\draw(2.9,-3) node[above]{\small $2$};
\draw(4.2,-3) node[above]{\small $1$};
\draw(2.4,-1.1) node[above]{\small $y_0$};
\draw [rounded corners,color=gray] (1.2,0.4)--(3.6,0.4)--(3.6,-1)--(1.2,-1)--cycle;
\draw(1.7,-1)--(1.7,-1.5);
\draw(2.95,-1)--(2.95,-1.5);
\draw(0.5,2.2)--(1.7,0.4);
\draw(1.5,2.2)--(3.8,1.2)--(4.2,-1.5);
\draw(2.95,2.2)--(1.1,1.2)--(0.5,-1.5);
\draw(4.2,2.2)--(2.95,0.4);
\end{braid}
\end{array}\right)
=\left(\begin{array}c
\begin{braid}\tikzset{baseline=0mm}
\draw(0.4,2) node[above]{\small $1$};
\draw(1.7,2) node[above]{\small $2$};
\draw(2.9,2) node[above]{\small $2$};
\draw(4.2,2) node[above]{\small $1$};
\draw(0.4,-3) node[above]{\small $1$};
\draw [rounded corners,color=gray] (1.2,-2.8)--(3.6,-2.8)--(3.6,-1.5)--(1.2,-1.5)--cycle;
\draw(1.8,-3) node[above]{\small $2$};
\draw(2.9,-3) node[above]{\small $2$};
\draw(4.2,-3) node[above]{\small $1$};
\draw(0.5,2.2)--(0.5,-1.5);
\draw(1.7,2.2)--(1.7,-1.5);
\draw(2.95,2.2)--(2.95,-1.5);
\draw(4.2,2.2)--(4.2,-1.5);
\blackdot(2.95,-0.5);
\end{braid}\\
\begin{braid}\tikzset{baseline=0mm}
\draw(0.4,2) node[above]{\small $2$};
\draw(1.6,2) node[above]{\small $1$};
\draw(2.9,2) node[above]{\small $1$};
\draw(4.2,2) node[above]{\small $2$};
\draw(0.4,-3) node[above]{\small $1$};
\draw [rounded corners,color=gray] (1.2,-2.8)--(3.6,-2.8)--(3.6,-1.5)--(1.2,-1.5)--cycle;
\draw(1.8,-3) node[above]{\small $2$};
\draw(2.9,-3) node[above]{\small $2$};
\draw(4.2,-3) node[above]{\small $1$};
\draw(0.5,2.2)--(1.7,-1.5);
\draw(1.5,2.2)--(4.2,-1.5);
\draw(3,2.2)--(0.5,-1.5);
\draw(4.2,2.2)--(2.95,-1.5);
\blackdot(3.15,-1);
\end{braid}
\end{array}\right),
\\
G_2&
=\left(-\begin{braid}\tikzset{baseline=0mm}
\draw(0.4,2) node[above]{\small $2$};
\draw(1.6,2) node[above]{\small $1$};
\draw(2.9,2) node[above]{\small $2$};
\draw(4,2) node[above]{\small $1$};
\draw(1,-3) node[above]{\small $2\ 2$};
\draw [rounded corners,color=gray] (-0.2,-2.8)--(2.2,-2.8)--(2.2,-1.5)--(-0.2,-1.5)--cycle;
\draw(3.55,-3) node[above]{\small $1\ 1$};
\draw [rounded corners,color=gray] (2.4,-2.8)--(4.8,-2.8)--(4.8,-1.5)--(2.4,-1.5)--cycle;
\draw(1,-1.1) node[above]{\small $y_0$};
\draw [rounded corners,color=gray] (-0.2,0.4)--(2.2,0.4)--(2.2,-1)--(-0.2,-1)--cycle;
\draw(1.55,-1)--(1.55,-1.5);
\draw(0.5,-1)--(0.5,-1.5);
\draw(0.5,2.2)--(0.5,0.4);
\draw(1.65,2.2)--(3.1,-1.5);
\draw(3,2.2)--(1.5,0.4);
\draw(4,2.2)--(4,-1.5);
\end{braid}\right)
=\left(-\begin{braid}\tikzset{baseline=0mm}
\draw(0.4,2) node[above]{\small $2$};
\draw(1.6,2) node[above]{\small $1$};
\draw(2.9,2) node[above]{\small $2$};
\draw(4,2) node[above]{\small $1$};
\draw(1,-3) node[above]{\small $2\ 2$};
\draw [rounded corners,color=gray] (-0.2,-2.8)--(2.2,-2.8)--(2.2,-1.5)--(-0.2,-1.5)--cycle;
\draw(3.55,-3) node[above]{\small $1\ 1$};
\draw [rounded corners,color=gray] (2.4,-2.8)--(4.8,-2.8)--(4.8,-1.5)--(2.4,-1.5)--cycle;
\draw(0.5,2.2)--(0.5,-1.5);
\draw(1.6,2.2)--(3.05,-1.5);
\draw(3,2.2)--(1.5,-1.5);
\draw(4,2.2)--(4,-1.5);
\blackdot(1.73,-1);
\end{braid}\right).
\end{align*}
\end{Example}

\begin{Lemma} \label{LDeg} 
Let $\bde\in\bLa(n)$ and $\la=\la^\bde$.  Then:
\begin{enumerate}
\item[{\rm (i)}] $\deg(\psi_{w(\bde)}e_\bde)=\frac{m(m-1)(l-1)}{2}-mn+\sum_{i=a}^b\la_i^2$,
\item[{\rm (ii)}] $\deg(f_{n}^{\la,\bde})=-\frac{m(m-1)(l+1)}{2}+mn-\sum_{i=a}^b\la_i^2$,
\item[{\rm (iii)}] $\deg(g_{n}^{\bde,\la})=\frac{m(m-1)(l+1)}{2}-mn+\sum_{i=a}^b\la_i^2$.
\end{enumerate}
\end{Lemma}
\begin{proof}
(i) We prove this by induction on $m$. Denote the right hand side by $R(m)$ and the left hand side by $L(m)$. If $m=1$ then $w(\bde)=1$, so $L(1)=0$. Moreover,
$$R(1)=-n+ \sum_{i=a}^b\la_i^2=-n+\sum_{i=a}^b(\de^{(1)}_i)^2=-n+\sum_{i=a}^b\de^{(1)}_i=0.$$

Let $m>1$. It suffices to prove that $R(m)-R(m-1)=L(m)-L(m-1)$.
Let $\de^{(m)}=(\eps_a,\dots,\eps_b)$.
Then, since all $\eps_i$ are $0$ or $1$, we have
\begin{align*}
R(m)-R(m-1)&=\frac{m(m-1)(l-1)}{2}-mn+\sum_{i=a}^b\la_i^2
 -
\frac{(m-1)(m-2)(l-1)}{2}
\\
&\ \ \
+(m-1)(n-\sum_{i=a}^b\eps_i)-\sum_{i=a}^b(\la_i-\eps_i)^2
\\&=(m-1)(l-1) -n -m\sum_{i=a}^b\eps_i+2\sum_{i=a}^b\la_i\eps_i.
\end{align*}

On the other hand, consider the Khovanov-Lauda diagram of $\psi_{w(\bde)}e_\bde$. The bottom positions of the diagram correspond to to the letters if the word $\bj^\bde$, and so the rightmost $l$ bottom positions of this diagram correspond to the letters of $\bj^{\de^{(m)}}$. In other words, counting from the right, the sequence of colors of these positions is $a^{\eps_a},\dots,b^{\eps_b},b+1,b^{1-\eps_b},\dots,a^{1-\eps_a}$.
Note that the strings which originate in these positions do not intersect each other, so $L(m)-L(m-1)$ equals the sum of the degrees of the intersections of these strings with the other strings of the diagram, i.e.
\begin{align*}
L(m)-L(m-1)&=\sum_{i=a+1}^b\eps_i(\la_{i-1}-\eps_{i-1})+\la_b-\eps_b
\\
&\ \ \ +\sum_{i=a+1}^b(1-\eps_i)(m-1-2(\la_i-\eps_i)+\la_{i-1}-\eps_{i-1})
\\
&\ \ \ + (1-\eps_a)(m-1-2(\la_a-\eps_a)),
\end{align*}
which is easily seen to equal the expression for $R(m)-R(m-1)$ obtained above.

(ii) This follows from (i) since $\deg(f_{n}^{\la,\bde})=\deg(\psi_{w(\bde)}e_\bd)+\deg(e_\la\psi_{w_0^\la})$ and
\begin{align*}\deg(e_\la\psi_{w_0^\la})=-m(m-1)-\sum_{i=a}^b(\la_i(\la_i-1)+(m-\la_i)(m-\la_i-1)).
\end{align*}

(iii) This follows from (i) since
$$
\deg(g_{n}^{\bde,\la})=\deg(e_\bde\psi_{u(\bde)})+\deg(y^\la)=\deg(\psi_{w(\bde)}e_\bde)+m(m-1).
$$
\end{proof}

\begin{Corollary} 
The homomorphisms $f_n$ and $g_n$ are homogeneous of degree $0$ for all $n$.
\end{Corollary}
\begin{proof}
Let $\bde\in\bLa(n)$ and $\la=\la^\bde$.
The homomorphism $f_n$ is a right multiplication with the matrix $F_n$. Its $(\la,\bde)$-component is a homomorphism
$
P_\la\to Q_\bde
$ obtained by the right multiplication with $f_n^{\la,\bde}$. Recall that $P_\la=q^{s_\la}R_\theta e_\la$ and $Q_\bde=q^{n+m(m-1)/2}R_\theta e_\bde$. So we just need to show that
$s_\la=n+m(m-1)/2+\deg(f_n^{\la,\bde})$, which easily follows from Lemma~\ref{LDeg}(ii).

The homomorphism $g_n$ is a right multiplication with the matrix $G_n$. Its $(\bde,\la)$-component is a homomorphism
$
Q_\bde\to P_\la
$ obtained by the right multiplication with $g_n^{\bde,\la}$. So we just need to show that
$n+m(m-1)/2=s_\la+\deg(g_n^{\bde,\la})$, which easily follows from Lemma~\ref{LDeg}(iii).
\end{proof}

\begin{Corollary} \label{Csamedegree}
Suppose $\bde, \beps \in \bLa(n)$ are such that $\la^{\bde}=\la^{\beps}$.  Then $\deg(\psi_{w(\bde)}e_\bde) = \deg(\psi_{w(\beps)}e_{\beps})$.
\end{Corollary}

\subsection{Independence of reduced decompositions}
Throughout this subsection we fix $\bde\in\bLa(n)$ and set $\la:=\la^\bde$.

Recall that in general the element $\psi_w\in R_\theta$ depends on a choice of a reduced decomposition of $w\in\Si_d$.
While it is clear from the form of the braid relations in the KLR algebra that $e_\la\psi_{w_0^\la}$ does not depend on a  choice of a reduced decomposition of $w_0^\la$, it is not obvious that  a similar statement is true for $\psi_{w(\bde)}e_\bde$ and $e_\bde\psi_{u(\bde)}$. So a priori the elements $f_n^{\la,\bde}=\pm e_\la\psi_{w_0^\la}\psi_{w(\bde)}e_\bde$ and $g_n^{\bde,\la}=\pm e_\bde\psi_{u(\bde)}y^\la e_\la$ might depend on choices of reduced decompositions of $w(\bde)$ and $u(\bde)$.  In this subsection we will prove that this is not the case,   and so in a sense the maps  $f_n$ and $g_n$ are canonical.

Recall the composition $\om_\la$ and the words $\bj^\la,\bj^\bde$ from \S\ref{SSComb}.


\begin{Lemma}\label{Lunique}
The element $w(\bde)$ is the unique element of\,\, ${}^{\om_\la}\D^{(l^m)}$ with $w(\bde) \cdot \bj^\bde=\bj^\la$.
\end{Lemma}
\begin{proof}
That $w(\bde)\in {}^{\om_\la}\D^{(l^m)}$ and $w(\bde) \cdot \bj^\bde=\bj^\la$ follows from the definitions. To prove the uniqueness statement, let $w\in {}^{\om_\la}\D^{(l^m)}$ and $w \cdot \bj^\bde=\bj^\la$.
Since $w\in {}^{\om_\la}\D$, it maps the elements of $U_{b+1}(\bde)$ increasingly to the elements of $U_{b+1}(\la)$.
By definition, $\bj^{\bde}=\bj^{\de^{(1)}}\dots\bj^{\de^{(m)}}$. For every $r\in [1,m]$, the entries of  $\bj^{\de^{(r)}}$ have the following properties: (1) each $i\in[a,b+1]$ appears among them exactly once; (2) the entries that precede $b+1$ appear in the increasing order; (3) the entries that succeed $b+1$ appear in the decreasing order.
Since $w\in \D^{(l^m)}$, it maps the positions corresponding to the entries in (2) to the positions which are to the left of the positions occupied with $b+1$ in $\bj^\la$, and it maps the positions corresponding to the entries in (3) to the positions which are to the right of the positions occupied with $b+1$ in $\bj^\la$.
In other words, for all $i\in[a,b]$, the permutation $w$ maps the elements of $U_i^\pm(\bde)$ to the elements of $U_i^\pm(\la)$.
As $w\in {}^{\om_\la}\D$, it now follows that for every $i\in[a,b]$, the permutation $w$ maps the elements of $U_i^\pm(\bde)$ to the elements of $U_i^\pm(\la)$  increasingly. We have shown that $w=w(\bde)$.
\end{proof}

\begin{Lemma} \label{Lposdeg}
Let $\eps,\de\in\La^\al$, and $\bj^\eps=w\cdot\bj^\de$ for some $w\in\Si_l$. Then either $\eps=\de$ and $w=1$, or $\deg(\psi_w 1_{\bj^\de})>0$.
\end{Lemma}
\begin{proof}
Since every $i\in[a,b+1]$ appears in $\bj^\de$ exactly once, $\eps=\de$ implies $w=1$. On the other hand, if $\eps\neq\de$, let $i$ be maximal with $\eps_i\neq \de_i$.
Then the strings colored $i$ and $i+1$ in the Khovanov-Lauda diagram $D$ for $\psi_w 1_{\bj^\de}$ intersect (for any choice of a reduced decomposition of $w$), which contributes a degree $1$ crossing into $D$. On the other hand, since every $j\in[a,b+1]$ appears in $\bj^\de$ exactly once, $D$ has no same color crossings, which are the only possible crossings of negative degree. The lemma follows.
\end{proof}

\begin{Lemma} \label{L160316} 
Suppose that $w\in {}^{(l^m)}\D^{\om_\la}$ and $w\cdot \bj^\la$ is of the form $\bi^{(1)}\dots\bi^{(m)}$ with $\bi^{(1)},\dots,\bi^{(m)}\in I^\al$. Then $w\cdot \bj^\la=\bj^\beps$ for some $\beps\in\bLa$ with $\la^\beps=\la$.
\end{Lemma}
\begin{proof}
Ler $r\in[1,m]$. By assumption, the entries $i^{(r)}_1,\dots,i^{(r)}_l$ of $\bi^{(r)}$ have the following properties: (1) each $i\in[a,b+1]$ appears among them exactly once; (2) the entries that precede $b+1$ appear in the increasing order; (3) the entries that succeed $b+1$ appear in the decreasing order. The result follows.
\end{proof}

\begin{Lemma} \label{deg1}
Let $P=\{w\in  {}^{\om_\la}\D\mid w\cdot\bj^{\bde}=\bj^\la\}$. Then $w(\bde)\in P$ and $\deg(\psi_{w(\bde)}e_\bde)<\deg  (\psi_{w}e_\bde)$ for any $w\in P\setminus\{w(\bde)\}$.
\end{Lemma}
\begin{proof}
It is clear that $w(\bde)\in P$. On the other hand, by Lemma~\ref{LDJ}, an arbitrary $w\in P$ can be written uniquely in the form $w=xy$ with $x\in {}^{\om_\la}\D^{(l^m)}$, $y\in\Si_{(l^m)}$ and $\ell(xy)=\ell(x)+\ell(y)$.

Since $xy\cdot\bj^\bde=\bj^\la$, we have $y\cdot\bj^\bde=x^{-1}\cdot\bj^\la$.
As $x^{-1}\in {}^{(l^m)}\D^{\om_\la}$, it follows from Lemma~\ref{L160316}, that $y\cdot\bj^\bde=x^{-1}\cdot\bj^\la$ is of the form $\bj^\beps$ for some $\beps\in\bLa$ with $\la^{\beps}=\la$.  By Lemma~\ref{Lunique}, $x=w(\beps)$.  If $w \neq w(\bde)$, then  $y \neq 1$ and we have
\begin{align*}
\deg(\psi_{w}e_\bde)
=\deg(\psi_{w(\beps)} e_\beps)+\deg(\psi_{y}e_\bde)
=\deg(\psi_{w(\bde)}e_\bde)+\deg(\psi_y e_\bde)>\deg(\psi_{w(\bde)}e_\bde),
\end{align*}
where we have used Corollary~\ref{Csamedegree} for the second equality and Lemma~\ref{Lposdeg} for the inequality.
\end{proof}

\begin{Lemma}\label{LFIndependent}
The element $f_n^{\la,\bde}=\sigma_\bde e_\la\psi_{w_0^\la}\psi_{w(\bde)}e_\bde$ is independent of the choice of reduced expressions for $w_0^\la$ and $w(\bde)$.
\end{Lemma}

\begin{proof}
It is clear from the form of the braid relations in the KLR algebra that $e_\la\psi_{w_0^\la}$ is independent of the choice of a reduced expression for $w_0^\la$. On the other hand, if $\psi_{w(\bde)}e_\bde$ and ${\psi_{w(\bde)}'}e_\bde$ correspond to different reduced expressions of $w(\bde)$, it follows from the defining relations of the KLR algebra and Theorem~\ref{TBasis} that $\psi_{w(\bde)}e_\bde-{\psi_{w(\bde)}'}e_\bde$ is a linear combination of elements of the form $\psi_u ye_\bde$ with $u\in\Si_d$, $y\in \k[y_1,\dots,y_d]$ such that, $\deg(\psi_u ye_\bde)=\deg(\psi_{w(\bde)}e_\bde)$ and $u\bj^\bde=\bj^\la$. We have to prove $e_\la\psi_{w_0^\la}\psi_uye_\bde=0$. Suppose otherwise.

Since we are using any preferred reduced decompositions for $u$, we may assume in addition that $u\in{}^{\om_\la}\D$, since otherwise $e_\la\psi_{w_0^\la}\psi_u=0$. Now by Lemma~\ref{deg1}, $\deg(\psi_ue_\bde)>\deg(\psi_{w(\bde)}e_\bde)$, whence $\deg(\psi_u ye_\bde)>\deg(\psi_{w(\bde)}e_\bde)$, giving a contradiction.
\end{proof}

\begin{Lemma}\label{LGIndependent}
The element $g_n^{\bde,\la}=\tau_\bde e_\bde\psi_{u(\bde)}y^\la e_\la$ is independent of the choice of a reduced expression for $u(\bde)$.
\end{Lemma}

\begin{proof}
The argument is similar to that of the previous lemma.
If $e_\bde\psi_{u(\bde)}$ and $e_\bde\psi_{u(\bde)}'$ correspond to different reduced expressions of $u(\bde)$, then $e_\bde\psi_{u(\bde)}-e_\bde\psi_{u(\bde)}'$ is a linear combination of elements of the form $e_\bde y\psi_w$ with $w\in\Si_d$, $y\in \k[y_1,\dots,y_d]$ such that $\deg(e_\bde y\psi_w)=\deg(e_\bde\psi_{u(\bde)})$ and $w^{-1}\bj^\bde=\bj^\la$. Moreover, in the KL dialgram of  $\psi_{w}1_{\bj^{\la}}$ the strings colored $b+1$ do not cross each other, since this was the case for the KL diagram of $\psi_{u(\bde)}1_{\bj^\la}$.
Hence, if $e_\bde y\psi_w y^\la e_\la\neq 0$, we may assume in addition that $w^{-1}\in{}^{\om_\la}\D$.  Now, using Lemma~\ref{deg1}, we conclude that $\deg(e_\bde y\psi_w)>\deg(e_\bde\psi_{u(\bde)})$, getting a contradiction.
\end{proof}

\subsection{Splitting}
In this subsection, we aim to show that $g\circ f = \id$.
We fix $\la \in \La(n)$ throughout the subsection. We need to prove
$\sum_{\bde\in\bLa(n)}f_n^{\la,\bde}g_n^{\bde,\la}=e_\la$. Since $f_n^{\la,\bde}=0$ unless $\la^\bde=\la$, this is equivalent to
$$\sum_{\bde\in\bLa(n),\ \la^\bde=\la}f_n^{\la,\bde}g_n^{\bde,\la}=e_\la.$$

Let $\bde\in\bLa(n)$ with $\la^\bde=\la$.  We say that $\bde$ is {\em initial} if $a$ preceeds $a+1$ in $\bj^{\de^{(r)}}$ for $r\in[1, m-\la_a]$ and $a$ succeeds $a+1$ in $\bj^{\de^{(r)}}$ for $r\in(m-\la_a, m]$. In other words, $\bde$ is initial if $\de^{(r)}_a=0$ for $r\in[1,m-\la_a]$ and $\de^{(r)}_a=1$ for $r\in(m-\la_a,m]$.

Let $w\in\Si_d$ and $1\leq r,s\leq d$. We say that $(r,s)$ is an {\em inversion pair for $w$} if $r<s$, $w(r)>w(s)$, and $\bj^\la_s-\bj^\la_r=\pm1$.

\begin{Lemma} \label{LInd} 
Let $\bde\in \bLa(n)$  be initial with $\la^\bde=\la$.  Set $\bar \al=\al_{a+1}+\dots+\al_b$, $\bar\theta=m\bar\al$, $\bar\la=(\la_{a+1},\dots,\la_b)$, $\bar n:=\la_{a+1}+\dots+\la_b$, and $\bar \bde=(\bar\de^{(1)},\dots,\bar\de^{(m)})$, where $\bar \de^{(r)}=(\de^{(r)}_{a+1},\dots,\de^{(r)}_b)$ for all $r\in[1,m]$. Then
\begin{equation}\label{E210916}
f_n^{\la,\bde} g_n^{\bde,\la}=1_{a^{(m-\la_a)}}\circ f_{\bar n}^{\bar\la,\bar\bde} g_{\bar n}^{\bar \bde,\bar \la}\circ 1_{a^{(\la_a)}}.
\end{equation}
\end{Lemma}
\begin{proof}
By definition,
$$
f_n^{\la,\bde} g_n^{\bde,\la}=(-1)^{\sum_{i=a}^b\la_i(\la_i-1)/2}e_\la\psi_{w_0^\la}\psi_{w(\bde)}\psi_{u(\bde)}y^\la e_\la.
$$
Throughout the proof, `inversion pair' means `inversion pair for $w(\bde)$'. Recall that $w(\bde)=u(\bde)^{-1}$.
Since $\bde$ is initial, in the Khovanov-Lauda diagram for  $\psi_{w(\bde)}$ (for any choice of reduced expression) no strings of color $a$ cross each other.
We want to apply quadratic relations on pairs of strings one of which has color $a$ and the other has color $a+1$.  These correspond to inversion pairs $(r,s)$ with $r\in U^-_a(\la), s\not \in U^-_a(\la)$ or $s\in U^+_a(\la), r\not\in U^+_a(\la)$.

Note that there are exactly $r-1$ inversion pairs of the form $(r,s)$ when $r\in U^-_a(\la)$ and $d-s$ inversion pairs of the form $(r,s)$ when $s\in U^+_a(\la)$. Applying the corresponding quadratic relations, we see that $f_n^{\la,\bde} g_n^{\bde,\la}$ equals
\begin{equation}\label{E190316}
(-1)^{\la_a(\la_a-1)/2}
e_\la\psi_{w_0^\la}
(y_{0,m-\la_a}\circ f_{\bar n}^{\bar\la,\bar\bde} g_{\bar n}^{\bar \bde,\bar \la}\circ y_{0,\la_a}')
y^\la e_\la+(*),
\end{equation}
where $(*)$ a sum of elements of the form
$$
e_\la\psi_{w_0^\la}
\iota_{(m-\la_a)\al_a,\bar \theta,\la_a\al_a}(Y^-\otimes X\otimes Y'^+)
y^\la e_\la,
$$
with $X\in R_{\bar\theta}$, $Y^\pm$ a polynomial in the variables $y_r$ with $r\in U_a^\pm(\la)$, and $\deg Y^-+\deg Y^+<\deg y_{0,m-\la_a}+\deg y_{0,\la_a}'$.
By Lemma~\ref{LW0Y0}(\ref{LW0Y0i}), we have $(*)=0$. So by Lemma~\ref{LW0Y0}(\ref{LW0Y0ii}), the expression (\ref{E190316}) equals the right hand side of (\ref{E210916}).
\end{proof}

Define $\bde_\la =(\de_\la^{(1)}, \cdots, \de_\la^{(m)})$ to be the unique element of $\bLa(n)$ such that for each $a \leq i \leq b$ we have:
\begin{itemize}
\item $i$ precedes $b+1$ in $\bj^{\de_\la^{(r)}}$ for $1\leq r\leq m-\la_i$;
\item $i$ succeeds $b+1$ in $\bj^{\de_\la^{(r)}}$ for $m-\la_i<r\leq  m$.
\end{itemize}
Note that $\la^{\bde_\la}=\la$ but $\bde_{\la^\bde}$ in general differs from $\bde$.

\begin{Lemma} 
Let $\bde\in\bLa$ satisfy $\la^\bde=\la$. Then
$$f_n^{\la,\bde} g_n^{\bde,\la}=
\left\{
\begin{array}{ll}
e_\la &\hbox{if $\bde=\bde_\la$,}\\
0 &\hbox{otherwise.}
\end{array}
\right.
$$
\end{Lemma}
\begin{proof}
If $\bde=\bde_\la$, the result follows by induction on $\height(\al)$ using from Lemma~\ref{LInd}. If $\bde\neq \bde_\la$, we may assume using Lemma~\ref{LInd} that $\bde$ is not initial. This implies that for some $r\in[1,m)$, we have $\de^{(r)}_a=1$ and $\de^{(r+1)}_a=0$, i.e. the last entry of the word $\bj^{\de^{(r)}}$ and the first entry of the word $\bj^{\de^{(r+1)}}$ are both equal to $a$. It follows that  $
\begin{braid}\tikzset{baseline=0mm}
\draw(0,0.5) node[above]{$a$}--(1,0)--(0,-0.5);
\draw(1,0.5) node[above]{$a$}--(0,0)--(1,-0.5);
\end{braid}
$
is a sub-diagram of a Khovanov-Lauda diagram for $\psi_{w(\bde)}\psi_{u(\bde)}y^\la e_\la$, so
$
f_n^{\la,\bde} g_n^{\bde,\la}=\pm e_\la\psi_{w_0^\la}\psi_{w(\bde)}\psi_{u(\bde)}y^\la e_\la=0.
$
\end{proof}

\begin{Corollary} \label{CSplit} 
For any $n\in\Z_{\geq 0}$, we have $g_n\circ f_n=\id$.
\end{Corollary}

\subsection{Proof of Theorem~\ref{TA}, assuming $f$ and $g$ are chain maps}
\label{SSProof}
In Sections~\ref{SF} and \ref{SG}, we will prove that $f$ and $g$ constructed above are chain maps. The goal of this subsection is to demonstrate that this is sufficient to establish our main result.

Let
$$Y_0 
:= \prod_{k=1}^m y_{kl}^{k-1} \in \k[y_1, \ldots, y_d].$$
We also define $W_0$
to be the longest element of $\D^{(l^m)}$.
For $1\leq r<m$, we have a fully commutative element
$$z_r:=\prod_{k=1}^l((r-1)l+k,rl+k) \in \Si_d.$$
In fact, if $w_{0,m}=s_{r_1}\dots s_{r_N}$ is a reduced decomposition in $\Si_m$, then  $W_0=z_{r_1}\dots z_{r_N}$ in $\Si_d$ with $\ell(W_0)= \ell(z_{r_1})+\dots+\ell( z_{r_N})$.
Recalling (\ref{EBZero}) we have
$$\bj^\bzero=(a\,(a+1)\cdots (b+1))^m\in I^d.$$ Then the following is easy to check:

\begin{Lemma} \label{thickcrossdegree}
	Let $P = \{ w\in \D^{(l^m)} \mid w\cdot \bj^\bzero = \bj^\bzero\}$.
	Then $W_0\in P$ and $\deg(\psi_{W_0}e_\bzero) < \deg(\psi_w e_\bzero)$ for any $w\in P\setminus \{ W_0\}$
\end{Lemma}

\begin{Lemma} \label{thickcrossindep}
Let $\psi_{W_0}$ and $\psi_{W_0}'$ in $R_\theta$ correspond to different reduced expressions of $W_0\in\Si_d$. Then
 $\psi_{W_0}v_\alpha^{\circ m}=\psi_{W_0}'v_\alpha^{\circ m}$ in
 $\Delta(\alpha)^{\circ m}$.
\end{Lemma}
\begin{proof}
	It follows from the defining relations of the KLR algebra and Theorem~\ref{TBasis} that $\psi_{W_0}e_\bzero - \psi_{W_0}'e_\bzero$ is a linear combination of elements of the form $\psi_w e_\bzero$ where $w\cdot \bj^\bzero = \bj^\bzero$, $w \neq W_0$, and $\deg(\psi_w e_\bzero) = \deg(\psi_{W_0}e_\bzero)$.
	Moreover, for each such $w$, by Lemma~\ref{LDJ}, there exist (unique) elements $x\in \D^{(l^m)}$ and $y\in \Si_{(l^m)}$ such that $w = xy$ and $\ell(w) = \ell(x)+\ell(y)$.
	Since $\psi_y v_\alpha^{\circ m} = 0$ if $y\neq 1$, we may assume $w \in \D^{(l^m)}$.
	We now have $w\in \D^{(l^m)}$, $w\cdot \bj^\bzero = \bj^\bzero$, and $w\neq W_0$.
	By Lemma~\ref{thickcrossdegree}, no such element $w$ can exist, so $(\psi_{W_0}e_\bzero - \psi_{W_0}'e_\bzero)v_\alpha^{\circ m} = 0$.
\end{proof}



\begin{Lemma} \label{FafterG}
 	There is a choice of reduced expression of $W_0$ for which $g_0^{\bzero,0}f_0^{0,\bzero} = e_{\bzero} Y_0 \psi_{W_0}e_{\bzero}$.
\end{Lemma}
\begin{proof}
	Note that $\ell(u(\bzero)w_0^0 w(\bzero)) = \ell(u(\bzero))+\ell(w_0^0)+\ell(w(\bzero))$ and $u(\bzero)w_0^0 w(\bzero) = W_0$. So for an appropriate choice of reduced expression for $W_0$, we have $\psi_{u(\bzero)}\psi_{w_0^0}\psi_{w(\bzero)} = \psi_{W_0}$. We now compute:
	\begin{align*}
		g_0^{\bzero,0}f_0^{0,\bzero} &= e_{\bzero}\psi_{u(\bzero)}y^0 e_0 \psi_{w_0^0}\psi_{w(\bzero)} e_\bzero \\
		&= e_{\bzero}\psi_{u(\bzero)}y^0 \psi_{w_0^0}\psi_{w(\bzero)} e_\bzero \\
		&= e_{\bzero}Y_0\psi_{u(\bzero)}\psi_{w_0^0}\psi_{w(\bzero)} e_\bzero \\
		&= e_\bzero Y_0 \psi_{W_0}e_\bzero,
	\end{align*}
	where the first equality is by definition of $g_0^{\bzero,0}$ and $f_0^{0,\bzero}$, the second follows from Lemma \ref{LW0Y0}(\ref{LW0Y0ii}), the third by the defining relations of $R_\theta$ and the observation that $Y_0 = u(\bzero)\cdot y^0$, and the fourth from the remark at the beginning of the proof.
\end{proof}

Note that
$$f_0^{0,\bzero}=\si_{\bzero}e_0\psi_{w_0^0}\psi_{w(\bzero)}e_\bzero=\psi_g1_{\bi_\al^m},$$
where $g$ is the longest element of $\Si_d$ with $g\cdot(\bi_\al^m)=a^m(a+1)^m\cdots (b+1)^m$. Let
\begin{equation}\label{ESG}
v_{\al^m}:=f_0^{0,\bzero}v_\al^{\circ m}e_m= \psi_gv_\al^{\circ m}e_m\in\De(\al^m).
\end{equation}
We refer to $v_{\al^m}$ as the {\em standard generator} of $\De(\al^m)$. It can be checked directly but also follows from Theorem~\ref{main theorem} that it does generate $\De(\al^m)$. Note that $e_0 v_{\al^m}=v_{\al^m}$, so there is a homomorphism
$${\sf p}:P_0=q^{-lm(m-1)/2}R_\theta e_0 \longrightarrow \De(\al^m),\ xe_0\mapsto xv_{\al^m},$$
where $v_{\al^m}$ is the standard generator of $\De(\al^m)$, see (\ref{ESG}).

\begin{Theorem} \label{main theorem}
If $f$ and $g$ are chain maps then $P_\bullet=P_\bullet^{\al^m}$ is a projective resolution of $\De(\al^m)$, with the augmentation map ${\sf p}$.
\end{Theorem}
\begin{proof}
The modules $P_n$ are projective by construction. By Corollary~\ref{CSplit}, $P_\bullet$ is a complex, isomorphic to a direct summand of the complex $Q_\bullet$. Since $Q_\bullet$ is a resolution of $q^{m(m-1)/2}\De(\al)^{\circ m}$,
it follows from the assumptions that $P_\bullet$ is exact in degrees $>0$ and its $0$th cohomology is 
given by ${\sf q} f_0 g_0(Q_0)$ where ${\sf q}:Q_0\to q^{m(m-1)/2}\De(\al)^{\circ m}$ is the augmentation map (\ref{EAugMap}).

Recalling the $(R_\theta, \NH_m)$-bimodule structure of $q^{m(m-1)/2}\De(\al)^{\circ m}$ from \S\ref{SSStdMod} and the idempotent $e_m\in \NH_m$ defined in (\ref{DeltaIdem}), using Lemma \ref{thickcrossindep}, we have
\begin{equation}\label{E221018}
	Y_0 \psi_{W_0}v_\alpha^{\circ m} = Y_0 \psi_{z_{r_1}}\cdots \psi_{z_{r_N}} v_{\alpha}^{\circ m} = v_\alpha^{\circ m}e_m,
\end{equation}
where $w_{0,m}=s_{r_1}\cdots s_{r_N}$ is a reduced decomposition in $\Si_m$.
We now have:
\begin{align*}
	{\sf q} f_0 g_0(Q_0) &= {\sf q}(q^{m(m-1)/2}R_\theta e_\bzero Y_0\psi_{W_0} e_\bzero) \\
	&= q^{m(m-1)/2}R_\theta Y_0\psi_{W_0} v_\alpha^{\circ m} \\
	&= q^{m(m-1)/2}R_\theta v_\alpha^{\circ m} e_m \\
	&= q^{m(m-1)/2}\Delta(\alpha)^{\circ m}e_m \\
	&= \Delta(\alpha^m),
\end{align*}
where the first equality follows from Lemma \ref{FafterG}, the second from the definition of ${\sf q}$, the third from (\ref{E221018}), and the last two from the definitions.

It remains to note that ${\sf q}(f_0(e_0))=v_{\al^m}$. For,
note that
${\sf q}(f_0(e_0))=f_0v_\al^{\circ m}$, and ${\sf q}(f_0(e_0))\in
{\sf q}(f_0(P_0))={\sf q}(f_0(g_0(Q_0)))=\Delta(\alpha^m)$. So
$f_0v_\al^{\circ m}=f_0v_\al^{\circ m}e_m=
v_{\al^m}$.
\end{proof}

\section{Verification that $f$ is a chain map}
\label{SF}
We continue with the running assumptions of the previous section.
In addition, throughout the section we fix
$$
n\in\Z_{\geq 0},\quad \mu=(\mu_a,\dots,\mu_b)\in\La(n+1)\quad \text{and}\quad \bde=(\de^{(1)},\dots,\de^{(m)})\in \bLa(n).
$$

\subsection{Special reduced expressions}
Recall the notation of \S\ref{SSComb}.
Let $\la\in\La^{\al^m}$ and $\de\in\La^\al$ be such that $\bar\la:=\la-\de\in\La^{\al^{m-1}}$. For $i\in[a,b+1]$, we denote by $p_i=p_i(\de)\in[1,l]$ the position occupied by $i$ in $\bj^{\de}$,  and set
\begin{equation}\label{EQI}
q_i=q_i(\la,\de):=
\left\{
\begin{array}{ll}
l_i^-(\la) &\hbox{if $i\neq b+1$ and $\de_i=0$,}\\
l_i^+(\la) &\hbox{if $i\neq b+1$ and $\de_i=1$,}\\
l_{b+1}(\la) &\hbox{if $i= b+1$.}
\end{array}
\right.
\end{equation}
Let $$Q:=\{q_a,\dots,q_{b+1}\}\subseteq [1,d].$$
Note that $\{p_a,\dots,p_{b+1}\}=[1,l]$.
Define $x_\de^\la\in\Si_d$ to be the permutation which maps $p_i$ to $q_i$ for all $i\in[a,b+1]$, and maps the elements of $[l+1,d]$ increasingly to the elements of $[1,d]\setminus Q$. It is easy to see that $x_\de^\la$ is fully commutative, so $\psi_{x_\de^\la}$ is well defined, and $1_{\bj^\la}\psi_{x_\de^\la}=\psi_{x_\de^\la}1_{\bj^\de\bj^{\bar\la}}$.

Now let $\la=\la^\bde$, and $r\in[1,m]$. Define
$$
\bde^{\geq r}:=(\de^{(r)},\dots,\de^{(m)})\in(\La^\al)^{m-r+1},\quad
\la^{\geq r}:=\la^{\bde^{\geq r}}\in\La^{\al^{m-r+1}}.
$$
Define $x(\bde,1):=x_{\de^{(1)}}^\la\in\Si_d$. More generally, define for all $r=1,\dots,m-1$, the permutations $x(\bde,r)\in\Si_d$ so that $x(\bde,r)$ is the image of
$$(1_{\Si_{(r-1)l}},x_{\de^{(r)}}^{\la^{\geq r}})\in \Si_{(r-1)l}\times\Si_{(m-r+1)l}$$
under the natural embedding $\Si_{(r-1)l}\times\Si_{(m-r+1)l}\into \Si_d$.

Recall the element $w(\bde)\in\Si_d$ defined in \S\ref{SSComp}. The following lemma follows from definitions:

\begin{Lemma} \label{LWX} 
We have $w(\bde)=x(\bde,1)\cdots x(\bde,m-1)$ and $\ell(w(\bde))=\ell(x(\bde,1))+\dots +\ell(x(\bde,m-1))$.
\end{Lemma}

In view of the lemma, when convenient, we will always choose reduced decompositions so that
\begin{equation}\label{WPsiW}
\psi_{w(\bde)}=\psi_{x(\bde,1)}\dots \psi_{x(\bde,m-1)}.
\end{equation}
By Lemma~\ref{LFIndependent}, we then have
\begin{equation}\label{CWX}
f_n^{\la,\bde}=\si_\bde e_\la\psi_{w_0^\la}\psi_{x(\bde,1)}\dots \psi_{x(\bde,m-1)}.
\end{equation}

\begin{Example} 
{\rm
Let $a=1$, $b=2$, $\la=(2,2)$, and $\bde=\big((1,0),(1,1),(0,1),(0,0)\big)$, so that $m=4$ and $n=4$. Then $f_n^{\la,\bde}$ is
$$
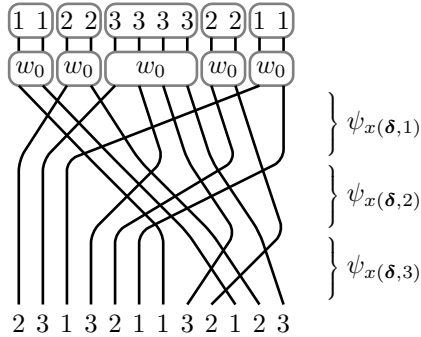
\begin{braid}\tikzset{baseline=0mm}
\draw(1,9) node[above]{\small $1$};
\draw(2,9) node[above]{\small $1$};
\draw(3,9) node[above]{\small $2$};
\draw(4,9) node[above]{\small $2$};
\draw(5,9) node[above]{\small $3$};
\draw(6,9) node[above]{\small $3$};
\draw(7,9) node[above]{\small $3$};
\draw(8,9) node[above]{\small $3$};
\draw(9,9) node[above]{\small $2$};
\draw(10,9) node[above]{\small $2$};
\draw(11,9) node[above]{\small $1$};
\draw(12,9) node[above]{\small $1$};
\draw [rounded corners,color=gray] (0.6,10.5)--(2.4,10.5)--(2.4,9.1)--(0.6,9.1)--cycle;
\draw [rounded corners,color=gray] (2.6,10.5)--(4.4,10.5)--(4.4,9.1)--(2.6,9.1)--cycle;
\draw [rounded corners,color=gray] (4.6,10.5)--(8.4,10.5)--(8.4,9.1)--(4.6,9.1)--cycle;
\draw [rounded corners,color=gray] (8.6,10.5)--(10.4,10.5)--(10.4,9.1)--(8.6,9.1)--cycle;
\draw [rounded corners,color=gray] (10.6,10.5)--(12.4,10.5)--(12.4,9.1)--(10.6,9.1)--cycle;
\draw(1,9.1)--(1,8.5);
\draw(2,9.1)--(2,8.5);
\draw(3,9.1)--(3,8.5);
\draw(4,9.1)--(4,8.5);
\draw(5,9.1)--(5,8.5);
\draw(6,9.1)--(6,8.5);
\draw(7,9.1)--(7,8.5);
\draw(8,9.1)--(8,8.5);
\draw(9,9.1)--(9,8.5);
\draw(10,9.1)--(10,8.5);
\draw(11,9.1)--(11,8.5);
\draw(12,9.1)--(12,8.5);
\draw [rounded corners,color=gray] (0.6,8.5)--(2.4,8.5)--(2.4,7.1)--(0.6,7.1)--cycle;
\draw [rounded corners,color=gray] (2.6,8.5)--(4.4,8.5)--(4.4,7.1)--(2.6,7.1)--cycle;
\draw [rounded corners,color=gray] (4.6,8.5)--(8.4,8.5)--(8.4,7.1)--(4.6,7.1)--cycle;
\draw [rounded corners,color=gray] (8.6,8.5)--(10.4,8.5)--(10.4,7.1)--(8.6,7.1)--cycle;
\draw [rounded corners,color=gray] (10.6,8.5)--(12.4,8.5)--(12.4,7.1)--(10.6,7.1)--cycle;
\draw(1.5,7) node[above]{\small $w_0$};
\draw(3.5,7) node[above]{\small $w_0$};
\draw(6.5,7) node[above]{\small $w_0$};
\draw(9.5,7) node[above]{\small $w_0$};
\draw(11.5,7) node[above]{\small $w_0$};
\draw(1,7.1)--(4,4)--(7,1)--(7,-2);
\draw(2,7.1)--(5,4)--(8,1)--(10,-2);
\draw(3,7.1)--(1,4)--(1,1)--(1,-2);
\draw(4,7.1)--(6,4)--(9,1)--(11,-2);
\draw(5,7.1)--(2,4)--(2,1)--(2,-2);
\draw(6,7.1)--(7,4)--(4,1)--(4,-2);
\draw(7,7.1)--(8,4)--(10,1)--(8,-2);
\draw(8,7.1)--(9,4)--(11,1)--(12,-2);
\draw(9,7.1)--(10,4)--(5,1)--(5,-2);
\draw(10,7.1)--(11,4)--(12,1)--(9,-2);
\draw(11,7.1)--(3,4)--(3,1)--(3,-2);
\draw(12,7.1)--(12,4)--(6,1)--(6,-2);
\draw(14,5.5) node{$\left.{
\begin{matrix}
    \\
 \\
 \\
 \\
 \\
\end{matrix}
}\right\}$};
\draw(14,2.5) node{$\left.{
\begin{matrix}
    \\
 \\
 \\
 \\
 \\
\end{matrix}
}\right\}$};
\draw(14,-0.5) node{$\left.{
\begin{matrix}
    \\
 \\
 \\
 \\
 \\
\end{matrix}
}\right\}$};
\draw(16.2,5.5) node{\small $\psi_{x(\bde,1)}$};
\draw(16.2,2.5) node{\small $\psi_{x(\bde,2)}$};
\draw(16.2,-0.5) node{\small $\psi_{x(\bde,3)}$};
\draw(1,-2.8) node{\small $2$};
\draw(2,-2.8) node{\small $3$};
\draw(3,-2.8) node{\small $1$};
\draw(4,-2.8) node{\small $3$};
\draw(5,-2.8) node{\small $2$};
\draw(6,-2.8) node{\small $1$};
\draw(7,-2.8) node{\small $1$};
\draw(8,-2.8) node{\small $3$};
\draw(9,-2.8) node{\small $2$};
\draw(10,-2.8) node{\small $1$};
\draw(11,-2.8) node{\small $2$};
\draw(12,-2.8) node{\small $3$};
\end{braid}
$$
}
\end{Example}

\subsection{A commutation lemma} \label{SSCommL}


It will be convenient to use the following notation. Let $\la\in\La(n)$. Consider the parabolic (non-unital) subalgebra
$$
R^\la:=R_{(m-\la_a)\al_a,\dots,(m-\la_b)\al_b,m\al_{b+1},\la_b\al_b,\dots,\la_a\al_a}\subseteq R_\theta.
$$
We have
$$
R^\la\cong R_{(m-\la_a)\al_a}\otimes \dots\otimes R_{(m-\la_b)\al_b}\otimes R_{m\al_{b+1}}\otimes R_{\la_b\al_b}\otimes \dots\otimes R_{\la_a\al_a}.
$$
The natural (unital) embeddings of the algebras $$R_{(m-\la_a)\al_a},\dots,R_{(m-\la_b)\al_b},R_{m\al_{b+1}},R_{\la_b\al_b},\dots,R_{\la_a\al_a}$$ into $R^\la$, allow us to consider them as (non-unital) subalgebras of $R_\theta$. We denote the corresponding (non-unital) algebra embeddings by
$$
\iota^\la_{a;-},\dots,\iota^\la_{b;-},\iota^\la_{b+1},\iota^\la_{b;+},\dots,\iota^\la_{a;+}.
$$
For example, setting
$$
\psi_{w_0}^\la(i;-):=\iota^\la_{i;-}(\psi_{w_{0,m-\la_i}}),\  \psi_{w_0}^\la(i;+):=\iota^\la_{i;+}(\psi_{w_{0,\la_i}}),\
\psi_{w_0}^\la(b+1):=\iota^\la_{b+1}(\psi_{w_{0,m}}),
$$
for all $i\in [a,b]$, can write $\psi_{w_0^\la}\in R_\theta$ as a commuting product
\begin{equation}\label{E130416}
\psi_{w_0^\la}=\psi_{w_0}^\la(a;-) \dots \psi_{w_0}^\la(b;-) \psi_{w_0}^\la(b+1) \psi_{w_0}^\la(b;+)\dots \psi_{w_0}^\la(a;+).
\end{equation}

\begin{Lemma} \label{L290316} 
Let $i\in[a,b]$ with $\mu_i>0$, and $\la:=\mu-\e_i$. Then
$$1_{\bj^\mu}\psi_{\la;i} \psi_{w_0^\la}=1_{\bj^\mu} \psi_{w_0^\mu}\psi_{l_i^+(\mu)\to l_i^-(\la)}.
$$
\end{Lemma}
\begin{proof}
We have
\begin{align*}
1_{\bj^\mu}\psi_{\la;i} \psi_{w_0^\la}&=
1_{\bj^\mu}\psi_{\la;i}1_{\bj^\la} \psi_{w_0^\la}
\\
&=1_{\bj^\mu}\psi_{\la;i} \psi_{w_0}^\la(b+1) \prod_{j\in [a,b]}\psi_{w_0}^\la(j;-)\psi_{w_0}^\la(j;+)
\\
&=1_{\bj^\mu}\psi_{w_0}^\mu(b+1) \Big[\prod_{j\in [a,b]\setminus\{i\}}\psi_{w_0}^\mu(j;-)\psi_{w_0}^\mu(j;+)\Big]
\psi_{\la;i} \psi_{w_0}^\la(i;-)\psi_{w_0}^\la(i;+)
\\
&=1_{\bj^\mu} \psi_{w_0}^\mu(b+1) \Big[\prod_{j\in [a,b]}\psi_{w_0}^\mu(j;-)\psi_{w_0}^\mu(j;+)\Big]\psi_{l_i^+(\mu)\to l_i^-(\la)}
\\
&=1_{\bj^\mu} \psi_{w_0^\mu}\psi_{l_i^+(\mu)\to l_i^-(\la)},
\end{align*}
where we have used (\ref{E130416}) for the second equality,  Lemma~\ref{LPull}(i) for the third equality, and the definition of $\psi_{\la,i}$ as an explicit cycle element for the fourth equality.
\end{proof}

\begin{Example}
We illustrate Lemma~\ref{L290316} diagramatically.  Let $a=1$, $b=2$, $m=4$ and $\mu = (2,3)$.  Then, Lemma~\ref{L290316} claims the following equality:
$$
\begin{braid}\tikzset{baseline=15mm}
\draw(-0.5,7.7) node[above]{\small $1_{\bj^\mu}\{$};
\draw(0.2,6.3) node{$\left.{
\begin{matrix}
\\
\\
\\
\\
\\
\\
\\
\end{matrix}
}\right\{$};
\draw(0.2,3) node{$\left.{
\begin{matrix}
    \\
 \\
 \\
\\
\\
\end{matrix}
}\right\{$};
\draw(-0.9,6.4) node{\small $\psi_{\la;1}$};
\draw(-0.9,3) node{\small $\psi_{w_0^\la}$};
\draw(1,8) node[above]{\small $1$};
\draw(2,8) node[above]{\small $1$};
\draw(3,8) node[above]{\small $2$};
\draw(4,8) node[above]{\small $3$};
\draw(5,8) node[above]{\small $3$};
\draw(6,8) node[above]{\small $3$};
\draw(7,8) node[above]{\small $3$};
\draw(8,8) node[above]{\small $2$};
\draw(9,8) node[above]{\small $2$};
\draw(10,8) node[above]{\small $2$};
\draw(11,8) node[above]{\small $1$};
\draw(12,8) node[above]{\small $1$};
\draw [rounded corners,color=gray] (0.6,4.5)--(3.4,4.5)--(3.4,3.1)--(0.6,3.1)--cycle;
\draw [rounded corners,color=gray] (4.6,4.5)--(8.4,4.5)--(8.4,3.1)--(4.6,3.1)--cycle;
\draw [rounded corners,color=gray] (8.6,4.5)--(11.4,4.5)--(11.4,3.1)--(8.6,3.1)--cycle;
\draw(1,8.1)--(1,4.5);
\draw(2,8.1)--(2,4.5);
\draw(3,8.1)--(4,4.5)--(4,2.5);
\draw(4,8.1)--(5,4.5);
\draw(5,8.1)--(6,4.5);
\draw(6,8.1)--(7,4.5);
\draw(7,8.1)--(8,4.5);
\draw(8,8.1)--(9,4.5);
\draw(9,8.1)--(10,4.5);
\draw(10,8.1)--(11,4.5);
\draw(11,8.1)--(12,4.5)--(12,2.5);
\draw(12,8.1)--(3,4.5);
\draw(1,3.1)--(1,2.5);
\draw(2,3.1)--(2,2.5);
\draw(3,3.1)--(3,2.5);
\draw(5,3.1)--(5,2.5);
\draw(6,3.1)--(6,2.5);
\draw(7,3.1)--(7,2.5);
\draw(8,3.1)--(8,2.5);
\draw(9,3.1)--(9,2.5);
\draw(10,3.1)--(10,2.5);
\draw(11,3.1)--(11,2.5);
\draw(2,3) node[above]{\small $w_0$};
\draw(6.5,3) node[above]{\small $w_0$};
\draw(6.5,3) node[above]{\small $w_0$};
\draw(10,3) node[above]{\small $w_0$};
\draw(1,1) node[above]{\small $1$};
\draw(2,1) node[above]{\small $1$};
\draw(3,1) node[above]{\small $1$};
\draw(4,1) node[above]{\small $2$};
\draw(5,1) node[above]{\small $3$};
\draw(6,1) node[above]{\small $3$};
\draw(7,1) node[above]{\small $3$};
\draw(8,1) node[above]{\small $3$};
\draw(9,1) node[above]{\small $2$};
\draw(10,1) node[above]{\small $2$};
\draw(11,1) node[above]{\small $2$};
\draw(12,1) node[above]{\small $1$};
\end{braid}
=
\begin{braid}\tikzset{baseline=15mm}
\draw(13.9,7.7) node[above]{\small $\} 1_{\bj^\mu}$};
\draw(13,7) node{$\left.{
\begin{matrix}
\\
\\
\\
\\
\end{matrix}
}\right\}$};
\draw(13,3.7) node{$\left.{
\begin{matrix}
\\
\\
\\
\\
\\
\\
\\
\\
\end{matrix}
}\right\}$};
\draw(14.5,7) node{\small $\psi_{w_0^\mu}$};
\draw(16.5,3.7) node{\small $\psi_{l_1^+(\mu) \to l_1^-(\la)}$};
\draw(1,8) node[above]{\small $1$};
\draw(2,8) node[above]{\small $1$};
\draw(3,8) node[above]{\small $2$};
\draw(4,8) node[above]{\small $3$};
\draw(5,8) node[above]{\small $3$};
\draw(6,8) node[above]{\small $3$};
\draw(7,8) node[above]{\small $3$};
\draw(8,8) node[above]{\small $2$};
\draw(9,8) node[above]{\small $2$};
\draw(10,8) node[above]{\small $2$};
\draw(11,8) node[above]{\small $1$};
\draw(12,8) node[above]{\small $1$};
\draw [rounded corners,color=gray] (0.6,7.5)--(2.4,7.5)--(2.4,6.1)--(0.6,6.1)--cycle;
\draw [rounded corners,color=gray] (3.6,7.5)--(7.4,7.5)--(7.4,6.1)--(3.6,6.1)--cycle;
\draw [rounded corners,color=gray] (7.6,7.5)--(10.4,7.5)--(10.4,6.1)--(7.6,6.1)--cycle;
\draw [rounded corners,color=gray] (10.6,7.5)--(12.4,7.5)--(12.4,6.1)--(10.6,6.1)--cycle;
\draw(1,8.1)--(1,7.5);
\draw(2,8.1)--(2,7.5);
\draw(3,8.1)--(3,6.1)--(4,2.5);
\draw(4,8.1)--(4,7.5);
\draw(5,8.1)--(5,7.5);
\draw(6,8.1)--(6,7.5);
\draw(7,8.1)--(7,7.5);
\draw(8,8.1)--(8,7.5);
\draw(9,8.1)--(9,7.5);
\draw(10,8.1)--(10,7.5);
\draw(11,8.1)--(11,7.5);
\draw(12,8.1)--(12,7.5);
\draw(1,6.1)--(2,2.5);
\draw(2,6.1)--(3,2.5);
\draw(4,6.1)--(5,2.5);
\draw(5,6.1)--(6,2.5);
\draw(6,6.1)--(7,2.5);
\draw(7,6.1)--(8,2.5);
\draw(8,6.1)--(9,2.5);
\draw(9,6.1)--(10,2.5);
\draw(10,6.1)--(11,2.5);
\draw(11,6.1)--(1,2.5);
\draw(12,6.1)--(12,2.5);
\draw(1.5,6) node[above]{\small $w_0$};
\draw(5.5,6) node[above]{\small $w_0$};
\draw(9,6) node[above]{\small $w_0$};
\draw(11.5,6) node[above]{\small $w_0$};
\draw(1,1) node[above]{\small $1$};
\draw(2,1) node[above]{\small $1$};
\draw(3,1) node[above]{\small $1$};
\draw(4,1) node[above]{\small $2$};
\draw(5,1) node[above]{\small $3$};
\draw(6,1) node[above]{\small $3$};
\draw(7,1) node[above]{\small $3$};
\draw(8,1) node[above]{\small $3$};
\draw(9,1) node[above]{\small $2$};
\draw(10,1) node[above]{\small $2$};
\draw(11,1) node[above]{\small $2$};
\draw(12,1) node[above]{\small $1$};
\end{braid}
$$
\end{Example}

\subsection{$f$ is a chain map}
Let $\la=\la^\bde$. For  $i\in[a,b+1]$, let
$q_i=q_i(\la,\de^{(1)})$ be defined as in (\ref{EQI}), and consider the cycle
$$
c_i=\left\{
\begin{array}{ll}
r_i^-(\mu)\to l_i^-(\mu) &\hbox{if $i\neq b+1$ and $\de^{(1)}_i=0$,}\\
r_i^+(\mu)\to l_i^+(\mu) &\hbox{if $i\neq b+1$ and $\de^{(1)}_i=1$,}\\
r_{b+1}(\mu)\to l_{b+1}(\mu) &\hbox{if $i= b+1$.}
\end{array}
\right.
$$
Let $c$ be the commuting product of cycles:
$$
c:=c_ac_{a+1}\dots c_{b+1}.
$$
Recall also the elements defined in (\ref{EQ123}).

\begin{Lemma} \label{L140416} 
Suppose $i\in [a,b]$ is such that $\mu_i>0$ and $\la:=\mu-\e_i=\la^\bde$. Set $\bar\theta=(m-1)\al$, $\bar\mu:=\mu-\de^{(1)}$, and $\bar\la:=\la-\de^{(1)}$. Then $1_{\bj^\mu}\psi_{\la;i}e_\la\psi_{w_0^\la}\psi_{x(\bde,1)}$ equals
$$
\left\{
\begin{array}{ll}
1_{\bj^\mu} \psi_{w_0^\mu}\psi_{x(\bde+\bolde_i^{1},1)}\psi_{\bde;1,i}
+1_{\bj^\mu}\psi_c \psi_{x_{\de^{(1)}}^\mu}
(1_\al\circ 1_{\bj^{\bar\mu}}\psi_{\bar\la;i}e_{\bar \la}\psi_{w_0^{\bar\la}})
 &\hbox{if $\de^{(1)}_i=0$,}\\
-1_{\bj^\mu}\psi_c \psi_{x_{\de^{(1)}}^\mu}
(1_\al\circ 1_{\bj^{\bar\mu}}\psi_{\bar\la;i}e_{\bar \la}\psi_{w_0^{\bar\la}}) &\hbox{if $\de^{(1)}_i=1$.}
\end{array}
\right.
$$
\end{Lemma}
\begin{proof}
By Lemma~\ref{LW0Y0}(\ref{LW0Y0ii}), we have $e_\la\psi_{w_0^\la}=1_{\bj^\la}\psi_{w_0^\la}$. So using also Lemma~\ref{L290316}, we get
$$
1_{\bj^\mu}\psi_{\la;i}e_\la\psi_{w_0^\la}\psi_{x(\bde,1)}=
 1_{\bj^\mu}\psi_{\la;i}1_{\bj^\la}\psi_{w_0^\la}\psi_{x(\bde,1)}
=1_{\bj^\mu} \psi_{w_0^\mu}\psi_{l_i^+(\mu)\to l_i^-(\la)}\psi_{x(\bde,1)}.
$$

As usual, we denote by $p_j$ the position occupied by $j$ in $\de^{(1)}$ and $q_j$ be defined as in (\ref{EQI}).
By Lemma~\ref{LWX}, the permutation $x(\bde,1)$ maps $p_j$ to $q_j$ for all $j\in[a,b+1]$, and the elements of $[l+1,d]$ increasingly to the elements of $[1,d]\setminus Q$.

{\em Case 1: $\de^{(1)}_i=0$}. 
In this case we have $q_i=l_i^-(\la)$. So the KLR diagram $D$ of $1_{\bj^\mu}\psi_{l_i^+(\mu)\to l_i^-(\la)}\psi_{x(\bde,1)}$ has an $i$-string $S$ from the position $p_i$ in the bottom to the position $l_i^+(\mu)$ in the top, and the only $(i,i)$-crossings in $D$ will be the crossings of the string $S$ with the $m-\mu_i$ strings which originate in the positions $L_i^-(\mu)$ in the top.
The $(i+1)$-string $T$ in $D$ originating in position $p_{i+1}$ in the bottom is to the right of all $(i,i)$ crossings. Pulling $T$ to the left produces error terms, which arise from opening  $(i,i)$-crossings, but all of them, except the last one, amount to zero, when multiplied on the left by $1_{\bj^\mu} \psi_{w_0^\mu}$. The last error term is equal to $\psi_{x_{\de^{(1)}}^\mu}
\iota_{\al,\bar\theta}(1_\al\otimes \psi_{l_i^+(\bar\mu)\to l_i^-(\bar\la)}))$, and the result of pulling $T$ past all $(i,i)$-crossings gives $\psi_{x(\bde+\bolde_i^{1},1)}\psi_{\bde;1,i}$. Multiplying on the left by $1_{\bj^\mu} \psi_{w_0^\mu}$, gives
$$
1_{\bj^\mu} \psi_{w_0^\mu}\psi_{x(\bde+\bolde_i^{1},1)}\psi_{\bde;1,i}+1_{\bj^\mu} \psi_{w_0^\mu}\psi_{x_{\de^{(1)}}^\mu}
\iota_{\al,\bar\theta}(1_\al\otimes \psi_{l_i^+(\bar\mu)\to l_i^i(\bar\la)}),
$$
and it remains to observe using Lemma~\ref{LPull} that
$$
\psi_{w_0^\mu}\psi_{x_{\de^{(1)}}^\mu}
\iota_{\al,\bar\theta}(1_\al\otimes \psi_{l_i^+(\bar\mu)\to l_i^i(\bar\la)})=1_{\bj^\mu}\psi_c \psi_{x_{\de^{(1)}}^\mu}
(1_\al\circ 1_{\bj^{\bar\mu}}\psi_{\bar\la;i}e_{\bar \la}\psi_{w_0^{\bar\la}}).
$$

{\em Case 2: $\de^{(1)}_i=1$.}
Let  $D$ be the KLR diagram of $1_{\bj^\mu}\psi_{l_i^+(\mu)\to l_i^-(\la)}\psi_{x(\bde,1)}$.
Let $S$ be the $i$-string originating in the position $l^+_{i}(\mu)$ in the top row of $D$, and $T$ be the $(i+1)$-string originating in the position $p_{i+1}$ in the bottom row of $D$. The quadratic relation on these strings produces a difference of two terms, one with a dot on $S$ and the other with a dot on $T$. The term with a dot on $T$ equals $0$ after multiplying on the left by $1_{\bj^\mu} \psi_{w_0^\mu}$. The term with a dot on $S$, when multiplied on the left by $1_{\bj^\mu} \psi_{w_0^\mu}$, yields
$$
-1_{\bj^\mu} \psi_{w_0^\mu}\psi_{x_{\de^{(1)}}^\mu}
\iota_{\al,\bar\theta}(1_\al\otimes \psi_{l_i^+(\bar\mu)\to l_i^-(\bar\la)})=-1_{\bj^\mu}\psi_c \psi_{x_{\de^{(1)}}^\mu}
(1_\al\circ 1_{\bj^{\bar\mu}}\psi_{\bar\la;i}e_{\bar \la}\psi_{w_0^{\bar\la}}),
$$
where we have used Lemma~\ref{LPull} for the last equality.
\end{proof}

\begin{Example}
We illustrate Lemma~\ref{L140416} diagramatically.
Let $a=1$, $b=2$, $m=4$, $\mu=(2,3)$ and $\bde=((1,0),(0,1),(1,1),(0,1))$.

First, take $i=2$. Then $\de_1^{(2)}=0$ and Lemma~\ref{L140416} claims the following:
$$
\begin{braid}\tikzset{baseline=8mm}
\draw(1,8) node[above]{\small $1$};
\draw(2,8) node[above]{\small $1$};
\draw(3,8) node[above]{\small $2$};
\draw(4,8) node[above]{\small $3$};
\draw(5,8) node[above]{\small $3$};
\draw(6,8) node[above]{\small $3$};
\draw(7,8) node[above]{\small $3$};
\draw(8,8) node[above]{\small $2$};
\draw(9,8) node[above]{\small $2$};
\draw(10,8) node[above]{\small $2$};
\draw(11,8) node[above]{\small $1$};
\draw(12,8) node[above]{\small $1$};
\draw [rounded corners,color=gray] (0.6,4.5)--(2.4,4.5)--(2.4,3.1)--(0.6,3.1)--cycle;
\draw [rounded corners,color=gray] (2.6,4.5)--(4.4,4.5)--(4.4,3.1)--(2.6,3.1)--cycle;
\draw [rounded corners,color=gray] (4.6,4.5)--(8.4,4.5)--(8.4,3.1)--(4.6,3.1)--cycle;
\draw [rounded corners,color=gray] (8.6,4.5)--(10.4,4.5)--(10.4,3.1)--(8.6,3.1)--cycle;
\draw [rounded corners,color=gray] (10.6,4.5)--(12.4,4.5)--(12.4,3.1)--(10.6,3.1)--cycle;
\draw(1,8.1)--(1,4.5);
\draw(2,8.1)--(2,4.5);
\draw(3,8.1)--(3,4.5);
\draw(4,8.1)--(5,4.5);
\draw(5,8.1)--(6,4.5);
\draw(6,8.1)--(7,4.5);
\draw(7,8.1)--(8,4.5);
\draw(8,8.1)--(9,4.5);
\draw(9,8.1)--(10,4.5);
\draw(10,8.1)--(4,4.5);
\draw(11,8.1)--(11,4.5);
\draw(12,8.1)--(12,4.5);
\draw(1,3.1)--(1,2.5);
\draw(2,3.1)--(2,2.5);
\draw(3,3.1)--(3,2.5);
\draw(4,3.1)--(4,2.5);
\draw(5,3.1)--(5,2.5);
\draw(6,3.1)--(6,2.5);
\draw(7,3.1)--(7,2.5);
\draw(8,3.1)--(8,2.5);
\draw(9,3.1)--(9,2.5);
\draw(10,3.1)--(10,2.5);
\draw(11,3.1)--(11,2.5);
\draw(12,3.1)--(12,2.5);
\draw [rounded corners,color=gray] (0.6,2.5)--(2.4,2.5)--(2.4,1.1)--(0.6,1.1)--cycle;
\draw [rounded corners,color=gray] (2.6,2.5)--(4.4,2.5)--(4.4,1.1)--(2.6,1.1)--cycle;
\draw [rounded corners,color=gray] (4.6,2.5)--(8.4,2.5)--(8.4,1.1)--(4.6,1.1)--cycle;
\draw [rounded corners,color=gray] (8.6,2.5)--(10.4,2.5)--(10.4,1.1)--(8.6,1.1)--cycle;
\draw [rounded corners,color=gray] (10.6,2.5)--(12.4,2.5)--(12.4,1.1)--(10.6,1.1)--cycle;
\draw(1.5,1) node[above]{\small $w_0$};
\draw(3.5,1) node[above]{\small $w_0$};
\draw(6.5,1) node[above]{\small $w_0$};
\draw(9.5,1) node[above]{\small $w_0$};
\draw(11.5,1) node[above]{\small $w_0$};
\draw(1,3) node[above]{\small $1$};
\draw(2,3) node[above]{\small $1$};
\draw(3,3) node[above]{\small $2$};
\draw(4,3) node[above]{\small $2$};
\draw(5,3) node[above]{\small $3$};
\draw(6,3) node[above]{\small $3$};
\draw(7,3) node[above]{\small $3$};
\draw(8,3) node[above]{\small $3$};
\draw(9,3) node[above]{\small $2$};
\draw(10,3) node[above]{\small $2$};
\draw(11,3) node[above]{\small $1$};
\draw(12,3) node[above]{\small $1$};
\draw(1,1.1)--(4,-2);
\draw(2,1.1)--(5,-2);
\draw(3,1.1)--(1,-2);
\draw(4,1.1)--(6,-2);
\draw(5,1.1)--(2,-2);
\draw(6,1.1)--(7,-2);
\draw(7,1.1)--(8,-2);
\draw(8,1.1)--(9,-2);
\draw(9,1.1)--(10,-2);
\draw(10,1.1)--(11,-2);
\draw(11,1.1)--(3,-2);
\draw(12,1.1)--(12,-2);

\draw(1,-2.8) node{\small $2$};
\draw(2,-2.8) node{\small $3$};
\draw(3,-2.8) node{\small $1$};
\draw(4,-2.8) node{\small $1$};
\draw(5,-2.8) node{\small $1$};
\draw(6,-2.8) node{\small $2$};
\draw(7,-2.8) node{\small $3$};
\draw(8,-2.8) node{\small $3$};
\draw(9,-2.8) node{\small $3$};
\draw(10,-2.8) node{\small $2$};
\draw(11,-2.8) node{\small $2$};
\draw(12,-2.8) node{\small $1$};
\end{braid}
=
\begin{braid}\tikzset{baseline=8mm}
\draw(1,8) node[above]{\small $1$};
\draw(2,8) node[above]{\small $1$};
\draw(3,8) node[above]{\small $2$};
\draw(4,8) node[above]{\small $3$};
\draw(5,8) node[above]{\small $3$};
\draw(6,8) node[above]{\small $3$};
\draw(7,8) node[above]{\small $3$};
\draw(8,8) node[above]{\small $2$};
\draw(9,8) node[above]{\small $2$};
\draw(10,8) node[above]{\small $2$};
\draw(11,8) node[above]{\small $1$};
\draw(12,8) node[above]{\small $1$};
\draw(1,8.1)--(1,7.5);
\draw(2,8.1)--(2,7.5);
\draw(3,8.1)--(3,6)--(6,2.5);
\draw(4,8.1)--(4,7.5);
\draw(5,8.1)--(5,7.5);
\draw(6,8.1)--(6,7.5);
\draw(7,8.1)--(7,7.5);
\draw(8,8.1)--(8,7.5);
\draw(9,8.1)--(9,7.5);
\draw(10,8.1)--(10,7.5);
\draw(11,8.1)--(11,7.5);
\draw(12,8.1)--(12,7.5);
\draw(1,1) node[above]{\small $3$};
\draw(2,1) node[above]{\small $2$};
\draw(3,1) node[above]{\small $1$};
\draw(4,1) node[above]{\small $1$};
\draw(5,1) node[above]{\small $1$};
\draw(6,1) node[above]{\small $2$};
\draw(7,1) node[above]{\small $3$};
\draw(8,1) node[above]{\small $3$};
\draw(9,1) node[above]{\small $3$};
\draw(10,1) node[above]{\small $2$};
\draw(11,1) node[above]{\small $2$};
\draw(12,1) node[above]{\small $1$};
\draw(1,1.1)--(2,-2.0);
\draw(2,1.1)--(1,-2.0);
\draw(3,1.1)--(3,-2.0);
\draw(4,1.1)--(4,-2.0);
\draw(5,1.1)--(5,-2.0);
\draw(6,1.1)--(6,-2.0);
\draw(7,1.1)--(7,-2.0);
\draw(8,1.1)--(8,-2.0);
\draw(9,1.1)--(9,-2.0);
\draw(10,1.1)--(10,-2.0);
\draw(11,1.1)--(11,-2.0);
\draw(12,1.1)--(12,-2.0);
\draw(1,6.1)--(4,2.5);
\draw(2,6.1)--(5,2.5);
\draw(4,6.1)--(1,2.5);
\draw(5,6.1)--(7,2.5);
\draw(6,6.1)--(8,2.5);
\draw(7,6.1)--(9,2.5);
\draw(8,6.1)--(2,2.5);
\draw(9,6.1)--(10,2.5);
\draw(10,6.1)--(11,2.5);
\draw(11,6.1)--(3,2.5);
\draw(12,6.1)--(12,2.5);
\draw [rounded corners,color=gray] (0.6,7.5)--(2.4,7.5)--(2.4,6.1)--(0.6,6.1)--cycle;
\draw [rounded corners,color=gray] (3.6,7.5)--(7.4,7.5)--(7.4,6.1)--(3.6,6.1)--cycle;
\draw [rounded corners,color=gray] (7.6,7.5)--(10.4,7.5)--(10.4,6.1)--(7.6,6.1)--cycle;
\draw [rounded corners,color=gray] (10.6,7.5)--(12.4,7.5)--(12.4,6.1)--(10.6,6.1)--cycle;
\draw(1.5,6) node[above]{\small $w_0$};
\draw(5.5,6) node[above]{\small $w_0$};
\draw(9,6) node[above]{\small $w_0$};
\draw(11.5,6) node[above]{\small $w_0$};
%
\draw(1,-2.8) node{\small $2$};
\draw(2,-2.8) node{\small $3$};
\draw(3,-2.8) node{\small $1$};
\draw(4,-2.8) node{\small $1$};
\draw(5,-2.8) node{\small $1$};
\draw(6,-2.8) node{\small $2$};
\draw(7,-2.8) node{\small $3$};
\draw(8,-2.8) node{\small $3$};
\draw(9,-2.8) node{\small $3$};
\draw(10,-2.8) node{\small $2$};
\draw(11,-2.8) node{\small $2$};
\draw(12,-2.8) node{\small $1$};
\end{braid}
+
\begin{braid}\tikzset{baseline=8mm}
\draw(1,8) node[above]{\small $1$};
\draw(2,8) node[above]{\small $1$};
\draw(3,8) node[above]{\small $2$};
\draw(4,8) node[above]{\small $3$};
\draw(5,8) node[above]{\small $3$};
\draw(6,8) node[above]{\small $3$};
\draw(7,8) node[above]{\small $3$};
\draw(8,8) node[above]{\small $2$};
\draw(9,8) node[above]{\small $2$};
\draw(10,8) node[above]{\small $2$};
\draw(11,8) node[above]{\small $1$};
\draw(12,8) node[above]{\small $1$};
\draw(1,8.1)--(4,4.5);
\draw(2,8.1)--(5,4.5);
\draw(3,8.1)--(1,4.5);
\draw(4,8.1)--(6,4.5);
\draw(5,8.1)--(7,4.5);
\draw(6,8.1)--(8,4.5);
\draw(7,8.1)--(2,4.5);
\draw(8,8.1)--(9,4.5);
\draw(9,8.1)--(10,4.5);
\draw(10,8.1)--(11,4.5);
\draw(11,8.1)--(12,4.5);
\draw(12,8.1)--(3,4.5);
\draw(1,3) node[above]{\small $2$};
\draw(2,3) node[above]{\small $3$};
\draw(3,3) node[above]{\small $1$};
\draw(4,3) node[above]{\small $1$};
\draw(5,3) node[above]{\small $1$};
\draw(6,3) node[above]{\small $3$};
\draw(7,3) node[above]{\small $3$};
\draw(8,3) node[above]{\small $3$};
\draw(9,3) node[above]{\small $2$};
\draw(10,3) node[above]{\small $2$};
\draw(11,3) node[above]{\small $2$};
\draw(12,3) node[above]{\small $1$};
\draw(1,3.1)--(1,-1.9);
\draw(2,3.1)--(2,-1.9);
\draw(3,3.1)--(3,-1.9);
\draw(4,3.1)--(4,0.7);
\draw(5,3.1)--(5,0.7);
\draw(6,3.1)--(7,0.7);
\draw(7,3.1)--(8,0.7);
\draw(8,3.1)--(9,0.7);
\draw(9,3.1)--(10,0.7);
\draw(10,3.1)--(11,0.7);
\draw(11,3.1)--(6,0.7);
\draw(12,3.1)--(12,0.7);
\draw(4,-0.7)--(4,-1.9);
\draw(5,-0.7)--(5,-1.9);
\draw(6,-0.7)--(6,-3.1);
\draw(7,-0.7)--(7,-1.9);
\draw(8,-0.7)--(8,-1.9);
\draw(9,-0.7)--(9,-1.9);
\draw(10,-0.7)--(10,-1.9);
\draw(11,-0.7)--(11,-1.9);
\draw(12,-0.7)--(12,-3.1);
\draw [rounded corners,color=gray] (3.6,0.7)--(5.4,0.7)--(5.4,-0.7)--(3.6,-0.7)--cycle;
\draw [rounded corners,color=gray] (6.6,0.7)--(9.4,0.7)--(9.4,-0.7)--(6.6,-0.7)--cycle;
\draw [rounded corners,color=gray] (9.6,0.7)--(11.4,0.7)--(11.4,-0.7)--(9.6,-0.7)--cycle;
\draw [rounded corners,color=gray] (3.6,-1.9)--(5.4,-1.9)--(5.4,-3.3)--(3.6,-3.3)--cycle;
\draw [rounded corners,color=gray] (6.6,-1.9)--(9.4,-1.9)--(9.4,-3.3)--(6.6,-3.3)--cycle;
\draw [rounded corners,color=gray] (9.6,-1.9)--(11.4,-1.9)--(11.4,-3.3)--(9.6,-3.3)--cycle;
\draw(4.5,-2.6) node{\small $w_0$};
\draw(8,-2.6) node{\small $w_0$};
\draw(10.5,-2.6) node{\small $w_0$};
\draw(1,-2.8) node{\small $2$};
\draw(2,-2.8) node{\small $3$};
\draw(3,-2.8) node{\small $1$};
\draw(4,-0.8) node[above]{\small $1$};
\draw(5,-0.8) node[above]{\small $1$};
\draw(6,-0.8) node[above]{\small $2$};
\draw(7,-0.8) node[above]{\small $3$};
\draw(8,-0.8) node[above]{\small $3$};
\draw(9,-0.8) node[above]{\small $3$};
\draw(10,-0.8) node[above]{\small $2$};
\draw(11,-0.8) node[above]{\small $2$};
\draw(12,-0.8) node[above]{\small $1$};
\end{braid}
$$

On the other hand, if $i=1$, we have $\de_1^{(1)}=1$, and Lemma~\ref{L140416} posits the following equality.
$$
\begin{braid}\tikzset{baseline=8mm}
\draw(1,8) node[above]{\small $1$};
\draw(2,8) node[above]{\small $1$};
\draw(3,8) node[above]{\small $2$};
\draw(4,8) node[above]{\small $3$};
\draw(5,8) node[above]{\small $3$};
\draw(6,8) node[above]{\small $3$};
\draw(7,8) node[above]{\small $3$};
\draw(8,8) node[above]{\small $2$};
\draw(9,8) node[above]{\small $2$};
\draw(10,8) node[above]{\small $2$};
\draw(11,8) node[above]{\small $1$};
\draw(12,8) node[above]{\small $1$};
\draw [rounded corners,color=gray] (0.6,4.5)--(3.4,4.5)--(3.4,3.1)--(0.6,3.1)--cycle;
\draw [rounded corners,color=gray] (4.6,4.5)--(8.4,4.5)--(8.4,3.1)--(4.6,3.1)--cycle;
\draw [rounded corners,color=gray] (8.6,4.5)--(11.4,4.5)--(11.4,3.1)--(8.6,3.1)--cycle;
\draw(1,8.1)--(1,4.5);
\draw(2,8.1)--(2,4.5);
\draw(3,8.1)--(4,4.5);
\draw(4,8.1)--(5,4.5);
\draw(5,8.1)--(6,4.5);
\draw(6,8.1)--(7,4.5);
\draw(7,8.1)--(8,4.5);
\draw(8,8.1)--(9,4.5);
\draw(9,8.1)--(10,4.5);
\draw(10,8.1)--(11,4.5);
\draw(11,8.1)--(12,4.5);
\draw(12,8.1)--(3,4.5);
\draw(1,3.1)--(1,2.5);
\draw(2,3.1)--(2,2.5);
\draw(3,3.1)--(3,2.5);
\draw(4,3.1)--(4,1.1)--(1,-2);
\draw(5,3.1)--(5,2.5);
\draw(6,3.1)--(6,2.5);
\draw(7,3.1)--(7,2.5);
\draw(8,3.1)--(8,2.5);
\draw(9,3.1)--(9,2.5);
\draw(10,3.1)--(10,2.5);
\draw(11,3.1)--(11,2.5);
\draw(12,3.1)--(12,1.1)--(3,-2);
\draw [rounded corners,color=gray] (0.6,2.5)--(3.4,2.5)--(3.4,1.1)--(0.6,1.1)--cycle;
\draw [rounded corners,color=gray] (4.6,2.5)--(8.4,2.5)--(8.4,1.1)--(4.6,1.1)--cycle;
\draw [rounded corners,color=gray] (8.6,2.5)--(11.4,2.5)--(11.4,1.1)--(8.6,1.1)--cycle;
\draw(2,1) node[above]{\small $w_0$};
\draw(6.5,1) node[above]{\small $w_0$};
\draw(6.5,1) node[above]{\small $w_0$};
\draw(10,1) node[above]{\small $w_0$};
\draw(1,3) node[above]{\small $1$};
\draw(2,3) node[above]{\small $1$};
\draw(3,3) node[above]{\small $1$};
\draw(4,3) node[above]{\small $2$};
\draw(5,3) node[above]{\small $3$};
\draw(6,3) node[above]{\small $3$};
\draw(7,3) node[above]{\small $3$};
\draw(8,3) node[above]{\small $3$};
\draw(9,3) node[above]{\small $2$};
\draw(10,3) node[above]{\small $2$};
\draw(11,3) node[above]{\small $2$};
\draw(12,3) node[above]{\small $1$};
\draw(1,1.1)--(4,-2);
\draw(2,1.1)--(5,-2);
\draw(3,1.1)--(6,-2);
\draw(5,1.1)--(2,-2);
\draw(6,1.1)--(7,-2);
\draw(7,1.1)--(8,-2);
\draw(8,1.1)--(9,-2);
\draw(9,1.1)--(10,-2);
\draw(10,1.1)--(11,-2);
\draw(11,1.1)--(12,-2);

\draw(1,-2.8) node{\small $2$};
\draw(2,-2.8) node{\small $3$};
\draw(3,-2.8) node{\small $1$};
\draw(4,-2.8) node{\small $1$};
\draw(5,-2.8) node{\small $1$};
\draw(6,-2.8) node{\small $1$};
\draw(7,-2.8) node{\small $3$};
\draw(8,-2.8) node{\small $3$};
\draw(9,-2.8) node{\small $3$};
\draw(10,-2.8) node{\small $2$};
\draw(11,-2.8) node{\small $2$};
\draw(12,-2.8) node{\small $2$};
\end{braid}
=
\begin{braid}\tikzset{baseline=8mm}
\draw(1,8) node[above]{\small $1$};
\draw(2,8) node[above]{\small $1$};
\draw(3,8) node[above]{\small $2$};
\draw(4,8) node[above]{\small $3$};
\draw(5,8) node[above]{\small $3$};
\draw(6,8) node[above]{\small $3$};
\draw(7,8) node[above]{\small $3$};
\draw(8,8) node[above]{\small $2$};
\draw(9,8) node[above]{\small $2$};
\draw(10,8) node[above]{\small $2$};
\draw(11,8) node[above]{\small $1$};
\draw(12,8) node[above]{\small $1$};
\draw(1,3) node[above]{\small $2$};
\draw(2,3) node[above]{\small $3$};
\draw(3,3) node[above]{\small $1$};
\draw(4,3) node[above]{\small $1$};
\draw(5,3) node[above]{\small $1$};
\draw(6,3) node[above]{\small $3$};
\draw(7,3) node[above]{\small $3$};
\draw(8,3) node[above]{\small $3$};
\draw(9,3) node[above]{\small $2$};
\draw(10,3) node[above]{\small $2$};
\draw(11,3) node[above]{\small $2$};
\draw(12,3) node[above]{\small $1$};
\draw(1,3.1)--(1,-2);
\draw(2,3.1)--(2,-2);
\draw(3,3.1)--(3,-2);
\draw(4,3.1)--(4,-0.1);
\draw(5,3.1)--(5,-0.1);
\draw(6,3.1)--(7,-0.1);
\draw(7,3.1)--(8,-0.1);
\draw(8,3.1)--(9,-0.1);
\draw(9,3.1)--(10,-0.1);
\draw(10,3.1)--(11,-0.1);
\draw(11,3.1)--(12,-0.1);
\draw(12,3.1)--(6,-0.1);
\draw(4,-1.5)--(4,-1.7);
\draw(5,-1.5)--(5,-1.7);
\draw(6,-1.5)--(6,-1.7);
\draw(7,-1.5)--(7,-1.7);
\draw(8,-1.5)--(8,-1.7);
\draw(9,-1.5)--(9,-1.7);
\draw(10,-1.5)--(10,-1.7);
\draw(11,-1.5)--(11,-1.7);
\draw(12,-1.5)--(12,-1.7);
\draw(1,8.1)--(1,6.3)--(4,4.5);
\draw(2,8.1)--(2,6.3)--(5,4.5);
\draw(3,8.1)--(3,6.3)--(1,4.5);
\draw(4,8.1)--(5,6.3)--(6,4.5);
\draw(5,8.1)--(6,6.3)--(7,4.5);
\draw(6,8.1)--(7,6.3)--(8,4.5);
\draw(7,8.1)--(4,6.3)--(2,4.5);
\draw(8,8.1)--(8,6.3)--(9,4.5);
\draw(9,8.1)--(9,6.3)--(10,4.5);
\draw(10,8.1)--(10,6.3)--(11,4.5);
\draw(11,8.1)--(12,6.3)--(12,4.5);
\draw(12,8.1)--(11,6.3)--(3,4.5);
\draw [rounded corners,color=gray] (3.6,-0.1)--(6.4,-0.1)--(6.4,-1.5)--(3.6,-1.5)--cycle;
\draw [rounded corners,color=gray] (6.6,-0.1)--(9.4,-0.1)--(9.4,-1.5)--(6.6,-1.5)--cycle;
\draw [rounded corners,color=gray] (9.6,-0.1)--(12.4,-0.1)--(12.4,-1.5)--(9.6,-1.5)--cycle;
\draw [rounded corners,color=gray] (3.6,-1.7)--(6.4,-1.7)--(6.4,-3.1)--(3.6,-3.1)--cycle;
\draw [rounded corners,color=gray] (6.6,-1.7)--(9.4,-1.7)--(9.4,-3.1)--(6.6,-3.1)--cycle;
\draw [rounded corners,color=gray] (9.6,-1.7)--(12.4,-1.7)--(12.4,-3.1)--(9.6,-3.1)--cycle;
\draw(5,-3.3) node[above]{\small $w_0$};
\draw(8,-3.3) node[above]{\small $w_0$};
\draw(11,-3.3) node[above]{\small $w_0$};
\draw(0,2) node[above]{\small $-$};
%
\draw(1,-2.8) node{\small $2$};
\draw(2,-2.8) node{\small $3$};
\draw(3,-2.8) node{\small $1$};
\draw(4,-0.8) node{\small $1$};
\draw(5,-0.8) node{\small $1$};
\draw(6,-0.8) node{\small $1$};
\draw(7,-0.8) node{\small $3$};
\draw(8,-0.8) node{\small $3$};
\draw(9,-0.8) node{\small $3$};
\draw(10,-0.8) node{\small $2$};
\draw(11,-0.8) node{\small $2$};
\draw(12,-0.8) node{\small $2$};
\end{braid}
$$
\end{Example}

\begin{Corollary} \label{C160416} 
If $\mu_i>0$ for some $i\in[a,b]$ and $\la:=\mu-\e_i=\la^\bde$, then
\begin{equation}\label{E150416}
1_{\bj^\mu}\psi_{\la;i}e_\la \psi_{w_0^\la}\psi_{w(\bde)}
=\sum_{r\in[1,m]:\ \de_i^{(r)}=0} (-1)^{\sum_{s=1}^{r-1} \de_i^{(s)}}1_{\bj^\mu} \psi_{w_0^\mu}\psi_{w(\bde+\bolde_i^{r})}\psi_{\bde;r,i}.
\end{equation}
\end{Corollary}
\begin{proof}
The proof is by induction on $m$, the induction base $m=1$ being obvious.  Let $\bar\bde:=(\de^{(2)},\dots,\de^{(m)})$, $\bar\theta=(m-1)\al$, $\bar\mu:=\mu-\de^{(1)}$, and $\bar\la:=\la-\de^{(1)}$.
By~(\ref{WPsiW}), we have
\begin{equation}\label{E200416}
1_{\bj^\mu}\psi_{\la;i}e_\la \psi_{w_0^\la}\psi_{w(\bde)}=1_{\bj^\mu}\psi_{\la;i}e_\la\psi_{w_0^\la}\psi_{x(\bde,1)}\psi_{x(\bde,2)}\dots \psi_{x(\bde,m-1)}.
\end{equation}
Now we apply Lemma~\ref{L140416}. We consider the case $\de_i^{(1)}=0$, the case $\de_i^{(1)}=1$ being similar. Then we get the following expression for the right hand side of (\ref{E200416}):
$$
\big(1_{\bj^\mu} \psi_{w_0^\mu}\psi_{x(\bde+\bolde_i^{1},1)}\psi_{\bde;1,i}
+1_{\bj^\mu}\psi_c \psi_{x_{\de^{(1)}}^\mu}
(1_\al\circ 1_{\bj^{\bar\mu}}\psi_{\bar\la;i}e_{\bar \la}\psi_{w_0^{\bar\la}})\big)
\psi_{x(\bde,2)}\dots \psi_{x(\bde,m-1)}.
$$
Opening parentheses,  we get two summands $S_1+S_2$.
Note that
\begin{align*}
S_1&=1_{\bj^\mu} \psi_{w_0^\mu}\psi_{x(\bde+\bolde_i^{1},1)}\psi_{\bde;1,i}\psi_{x(\bde,2)}\dots \psi_{x(\bde,m-1)}
\\
&=  1_{\bj^\mu} \psi_{w_0^\mu}\psi_{x(\bde+\bolde_i^{1},1)}\psi_{x(\bde,2)}\dots \psi_{x(\bde,m-1)}\psi_{\bde;1,i}
\\
&=  1_{\bj^\mu} \psi_{w_0^\mu}\psi_{w(\bde+\bolde_i^{1})}\psi_{\bde;1,i}.
\end{align*}
Moreover, using the inductive assumption for the third equality below, we get that $S_2$ equals
\begin{align*}
&1_{\bj^\mu}\psi_c \psi_{x_{\de^{(1)}}^\mu}
(1_\al\circ 1_{\bj^{\bar\mu}}\psi_{\bar\la;i}e_{\bar \la}\psi_{w_0^{\bar\la}})\psi_{x(\bde,2)}\dots \psi_{x(\bde,m-1)}
\\
=& 1_{\bj^\mu}\psi_c \psi_{x_{\de^{(1)}}^\mu}
(1_\al\circ 1_{\bj^{\bar\mu}}\psi_{\bar\la;i}e_{\bar \la}\psi_{w_0^{\bar\la}}\psi_{x(\bar\bde,1)}\dots \psi_{x(\bar \bde,m-2)})
\\
=& 1_{\bj^\mu}\psi_c \psi_{x_{\de^{(1)}}^\mu}
(1_\al\circ 1_{\bj^{\bar\mu}}\psi_{\bar\la;i}e_{\bar \la}\psi_{w_0^{\bar\la}}\psi_{w(\bar\bde)})
\\
=& 1_{\bj^\mu}\psi_c \psi_{x_{\de^{(1)}}^\mu}
\big(1_\al\circ \sum_{r\in[2,m]:\ \de_i^{(r)}=0} (-1)^{\sum_{s=2}^{r-1} \de_i^{(s)}}1_{\bj^{\bar\mu}} \psi_{w_0^{\bar \mu}}\psi_{w(\bar \bde+\bolde_i^{r-1})}\psi_{\bar\bde;r-1,i}\big).
\end{align*}
As $\de_i^{(1)}=0$, we have $\sum_{s=2}^{r-1} \de_i^{(s)}=\sum_{s=1}^{r-1} \de_i^{(s)}$. So
\begin{align*}
&\iota_{\al,\bar\theta}\big(1_\al\otimes \sum_{r\in[2,m]:\ \de_i^{(r)}=0} (-1)^{\sum_{s=2}^{r-1} \de_i^{(s)}}1_{\bj^{\bar\mu}} \psi_{w_0^{\bar \mu}}\psi_{w(\bar \bde+\bolde_i^{r-1})}\psi_{\bar\bde;r-1,i}\big)
\\
=&\sum_{r\in[2,m]:\ \de_i^{(r)}=0} (-1)^{\sum_{s=1}^{r-1} \de_i^{(s)}}\iota_{\al,\bar\theta}\big(1_\al\otimes 1_{\bj^{\bar\mu}} \psi_{w_0^{\bar \mu}}\bigg(\prod_{t=1}^{m-2}\psi_{x(\bar \bde+\bolde_i^{r-1},t)}\bigg)\psi_{\bar\bde;r-1,i}\big)
\\
=&\sum_{r\in[2,m]:\ \de_i^{(r)}=0} (-1)^{\sum_{s=1}^{r-1} \de_i^{(s)}}\iota_{\al,\bar\theta}\big(1_\al\otimes 1_{\bj^{\bar\mu}} \psi_{w_0^{\bar \mu}}\big)\bigg(\prod_{t=2}^{m-1}\psi_{x(\bde+\bolde_i^{r},t)}\bigg)\psi_{\bde;r,i}.
\end{align*}
Moreover,
$$
\psi_{x_{\de^{(1)}}^\mu}\prod_{t=2}^{m-1}\psi_{x(\bde+\bolde_i^{r},t)}
=\psi_{x(\bde+\bolde_i^{r},1)}\prod_{t=2}^{m-1}\psi_{x(\bde+\bolde_i^{r},t)}=\psi_{w(\bde+\bolde_i^{r})}.
$$
So $S_2$ equals
\begin{align*}
& \sum_{r\in[2,m]:\ \de_i^{(r)}=0}  (-1)^{\sum_{s=1}^{r-1} \de_i^{(s)}} 1_{\bj^\mu}\psi_c \psi_{x_{\de^{(1)}}^\mu}
\iota_{\al,\bar\theta}(1_\al\otimes 1_{\bj^{\bar\mu}} \psi_{w_0^{\bar \mu}})\bigg(\prod_{t=2}^{m-1}\psi_{x(\bde+\bolde_i^{r},t)}\bigg)\psi_{\bde;r,i}
\\
=& \sum_{r\in[2,m]:\ \de_i^{(r)}=0} (-1)^{\sum_{s=1}^{r-1} \de_i^{(s)}} 1_{\bj^\mu}\psi_{w_0^\mu} \psi_{x_{\de^{(1)}}^\mu}
\bigg(\prod_{t=2}^{m-1}\psi_{x(\bde+\bolde_i^{r},t)}\bigg)\psi_{\bde;r,i}
\\
=& \sum_{r\in[2,m]:\ \de_i^{(r)}=0} (-1)^{\sum_{s=1}^{r-1} \de_i^{(s)}} 1_{\bj^\mu}\psi_{w_0^\mu} \psi_{w(\bde+\bolde_i^{r})}\psi_{\bde;r,i},
\end{align*}
where we have used Lemma~\ref{LPull}(i) to see the first equality.
Thus $S_1+S_2$ equals the right hand side of (\ref{E150416}).
\end{proof}

The following statement means that $f$ is chain map:

\begin{Proposition} \label{fchain}
Let $\mu\in\La(n+1)$ and $\bde\in\bLa(n)$. Then
$$
\sum_{\la\in\La(n)} d_n^{\mu,\la}f^{\la,\bde}_{n}=\sum_{\bga\in\bLa(n+1)} f^{\mu,\bga}_{n+1}c_n^{\bga,\bde}.
$$
\end{Proposition}
\begin{proof}
By definition, $d_n^{\mu,\la}=0$ unless $\la=\mu-\e_i$ for some $i\in [a,b]$, and $f_n^{\la,\bde}=0$ unless $\la=\la^\bde$. On the other hand, $f_n^{\mu,\bga}=0$ unless $\mu=\la^\bga$,
and
$c_n^{\bga,\bde}=0$, unless $\bde=\bga-\bolde_i^{r}$ for some $i\in [a,b]$ and $r\in[1,m]$. So we may assume that $\mu=\la^{\bde+\bolde_i^{r}}$ for some $i\in [a,b]$ and $r\in[1,m]$ such that $\de_i^{(r)}=0$. In this case, letting $\la:=\la^{\bde}$, we have to prove
$$
d_n^{\mu,\la}f^{\la,\bde}_{n}=\sum_{r\in [1,m]:\ \de_i^{(r)}=0} f^{\mu,\bde+\bolde_i^{r}}_{n+1}c_n^{\bde+\bolde_i^{r},\bde}.
$$
By definition of the elements involved, this means
\begin{align*}
&(\sgn_{\la;i}e_\mu\psi_{\la;i}e_\la)
(\si_\bde e_\la\psi_{w_0^\la}\psi_{w(\bde)}e_\bde)
\\
=
&\sum_{r\in [1,m]:\ \de_i^{(r)}=0}
(\si_{\bde+\bolde_i^{r}} e_\mu\psi_{w_0^\mu}\psi_{w(\bde+\bolde_i^{r})}e_{\bde+\bolde_i^{r}})
(\sgn_{\bde;r,i}e_{\bde+\bolde_i^{r}}\psi_{\bde;r,i}e_\bde).
\end{align*}
Equivalently, we need to prove
\begin{align*}
\sgn_{\la;i}\si_\bde e_\mu\psi_{\la;i}e_\la
 \psi_{w_0^\la}\psi_{w(\bde)}
=
\sum_{r\in [1,m]:\ \de_i^{(r)}=0}
\si_{\bde+\bolde_i^{r}} \sgn_{\bde;r,i}
e_\mu\psi_{w_0^\mu}\psi_{w(\bde+\bolde_i^{r})}
\psi_{\bde;r,i},
\end{align*}
which, in view of Corollary~\ref{C160416}, is equivalent to the statement that
$$
\sgn_{\la^;i}\si_\bde=\si_{\bde+\bolde_i^{r}} \sgn_{\bde;r,i}(-1)^{\sum_{s=1}^{r-1} \de_i^{(s)}}
$$
for all $r\in [1,m]$ such that $\de_i^{(r)}=0$. But this is Lemma~\ref{LSigns}.
\end{proof}

\section{Verification that $g$ is a chain map}
\label{SG}
We continue with the running assumptions of Section~\ref{SDPR}.
In addition, throughout the section we fix $n\in\Z_{\geq 0}$, $\la\in\La(n)$ and $\bga=(\ga^{(1)},\dots,\ga^{(m)})\in \bLa(n+1)$.

\subsection{Special reduced expressions and a commutation lemma}\label{SSSpecialG}
Recall the notation of \S\ref{SSComb}.
Let $\mu\in\La^{\al^m}$ and $\ga\in\La^\al$ be such that $\bar\mu:=\mu-\ga\in\La^{\al^{m-1}}$. For $i\in[a,b+1]$, we denote
\begin{eqnarray}
p^i&:=&(m-1)l+(\text{the position occupied by $i$ in $\bj^{\gamma}$}),
\label{EQIG'}
\\
\label{EQIG}
q^i&:=&
\left\{
\begin{array}{ll}
r_i^-(\mu) &\hbox{if $i\neq b+1$ and $\ga^{(m)}_i=0$,}\\
r_i^+(\mu) &\hbox{if $i\neq b+1$ and $\ga^{(m)}_i=1$,}\\
r_{b+1}(\mu) &\hbox{if $i= b+1$.}
\end{array}
\right.
\end{eqnarray}
Let $Q:=\{q^a,\dots,q^{b+1}\}$.
Note that $\{p^a,\dots,p^{b+1}\}=(d-l,d]$.

Define $z_\mu^\ga\in\Si_d$ to be the permutation which maps $q^i$ to $p^i$ for all $i\in[a,b+1]$, and maps the elements of
$[1,d]\setminus Q$ increasingly to the elements of $[1,d-l]$. It is easy to see that $z_\mu^\ga$ is fully commutative, so $\psi_{z_\mu^\ga}$ is well defined, and $\psi_{z_\mu^\ga}1_{\bj^\mu}=1_{\bj^{\bar\mu}\bj^\ga}\psi_{z_\mu^\ga}$.

Now let $\mu:=\la^\bga$, and $r\in[1,m]$. Define
$$
\bga^{\leq r}:=(\ga^{(1)},\dots,\ga^{(r)})\in\bLa^{\al^{r}},\quad
\mu^{\leq r}:=\la^{\bga^{\leq r}}\in\La^{\al^{r}}.
$$
Define $z(\bga,m):=z_\mu^{\ga^{(m)}}\in\Si_d$. More generally, define for all $r=2,\dots,m$, the permutations
$$
z(\bga,r):=(x_{\mu^{\leq r}}^{\ga^{(r)}}, 1_{\Si_{(m-r)l}})\in \Si_{rl}\times \Si_{(m-r)l} \leq \Si_d.
$$

Recall the element $u(\bga)\in\Si_d$ defined in \S\ref{SSComp}. The following lemma follows from definitions:

\begin{Lemma} \label{LWXG} 
We have $u(\bga)=z(\bga,2)\cdots z(\bga,m)$ and $\ell(u(\bga))=\ell(z(\bga,1))+\dots +\ell(z(\bga,m))$.
\end{Lemma}

In view of the lemma, when convenient, we will always choose reduced decompositions so that
\begin{equation}\label{WPsiWG}
\psi_{u(\bga)}=\psi_{z(\bga,2)}\dots \psi_{z(\bga,m)}.
\end{equation}

In view of Lemma~\ref{LFIndependent}, we have
\begin{equation}\label{CWX2}
g_{n+1}^{\bga,\mu}= \tau_\bga \psi_{z(\bga,2)}\dots \psi_{z(\bga,m)}y^\mu e_\mu.
\end{equation}

\begin{Example} 
{\rm
Let $a=1$, $b=2$, $\mu=(2,2)$, and $\bga=\big((1,0),(1,1),(0,1),(0,0)\big)$, so that $m=4$ and $n=4$. Then $g_n^{\la,\bga}$ is
$$
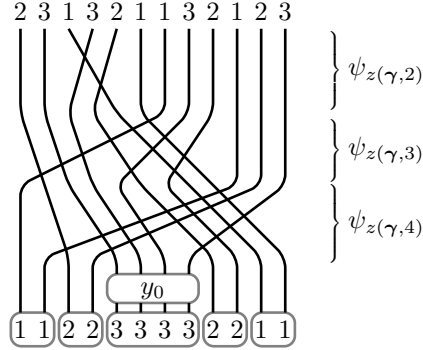
\begin{braid}\tikzset{baseline=0mm}
\draw(1,9) node[above]{\small $2$};
\draw(2,9) node[above]{\small $3$};
\draw(3,9) node[above]{\small $1$};
\draw(4,9) node[above]{\small $3$};
\draw(5,9) node[above]{\small $2$};
\draw(6,9) node[above]{\small $1$};
\draw(7,9) node[above]{\small $1$};
\draw(8,9) node[above]{\small $3$};
\draw(9,9) node[above]{\small $2$};
\draw(10,9) node[above]{\small $1$};
\draw(11,9) node[above]{\small $2$};
\draw(12,9) node[above]{\small $3$};
\draw [rounded corners,color=gray] (0.6,-2.7)--(2.4,-2.7)--(2.4,-4.1)--(0.6,-4.1)--cycle;
\draw [rounded corners,color=gray] (2.6,-2.7)--(4.4,-2.7)--(4.4,-4.1)--(2.6,-4.1)--cycle;
\draw [rounded corners,color=gray] (4.6,-2.7)--(8.4,-2.7)--(8.4,-4.1)--(4.6,-4.1)--cycle;
\draw [rounded corners,color=gray] (8.6,-2.7)--(10.4,-2.7)--(10.4,-4.1)--(8.6,-4.1)--cycle;
\draw [rounded corners,color=gray] (10.6,-2.7)--(12.4,-2.7)--(12.4,-4.1)--(10.6,-4.1)--cycle;
\draw [rounded corners,color=gray] (4.6,-1.1)--(8.4,-1.1)--(8.4,-2.3)--(4.6,-2.3)--cycle;

\draw(6.5,-2.6) node[above]{\small $y_0$};
\draw(1,9.1)--(1,5.6)--(2,2.6)--(3,-0.4)--(3,-2.7);
\draw(2,9.1)--(2,5.6)--(3,2.6)--(5,-0.4)--(5,-1.1);
\draw(3,9.1)--(5,5.6)--(8,2.6)--(11,-0.4)--(11,-2.7);
\draw(4,9.1)--(3,5.6)--(4,2.6)--(6,-0.4)--(6,-1.1);
\draw(5,9.1)--(4,5.6)--(6,2.6)--(9,-0.4)--(9,-2.7);
\draw(6,9.1)--(6,5.6)--(9,2.6)--(12,-0.4)--(12,-2.7);
\draw(7,9.1)--(7,5.6)--(1,2.6)--(1,-0.4)--(1,-2.7);
\draw(8,9.1)--(8,5.6)--(5,2.6)--(7,-0.4)--(7,-1.1);
\draw(9,9.1)--(9,5.6)--(7,2.6)--(10,-0.4)--(10,-2.7);
\draw(10,9.1)--(10,5.6)--(10,2.6)--(2,-0.4)--(2,-2.7);
\draw(11,9.1)--(11,5.6)--(11,2.6)--(4,-0.4)--(4,-2.7);
\draw(12,9.1)--(12,5.6)--(12,2.6)--(8,-0.4)--(8,-1.1);
\draw(5,-2.3)--(5,-2.7);
\draw(6,-2.3)--(6,-2.7);
\draw(7,-2.3)--(7,-2.7);
\draw(8,-2.3)--(8,-2.7);
\draw(14,7.35) node{$\left.{
\begin{matrix}
    \\
 \\
 \\
 \\
 \\
 \\
\end{matrix}
}\right\}$};
\draw(14,4.05) node{$\left.{
\begin{matrix}
    \\
 \\
 \\
 \\
 \\
\end{matrix}
}\right\}$};
\draw(14,1.0) node{$\left.{
\begin{matrix}
    \\
 \\
 \\
 \\
 \\
 \\
\end{matrix}
}\right\}$};
\draw(16.2,7.35) node{\small $\psi_{z(\bga,2)}$};
\draw(16.2,4.05) node{\small $\psi_{z(\bga,3)}$};
\draw(16.2,1.0) node{\small $\psi_{z(\bga,4)}$};
\draw(1,-3.4) node{\small $1$};
\draw(2,-3.4) node{\small $1$};
\draw(3,-3.4) node{\small $2$};
\draw(4,-3.4) node{\small $2$};
\draw(5,-3.4) node{\small $3$};
\draw(6,-3.4) node{\small $3$};
\draw(7,-3.4) node{\small $3$};
\draw(8,-3.4) node{\small $3$};
\draw(9,-3.4) node{\small $2$};
\draw(10,-3.4) node{\small $2$};
\draw(11,-3.4) node{\small $1$};
\draw(12,-3.4) node{\small $1$};
\end{braid}
$$
}
\end{Example}

\begin{Lemma} \label{L290316G} 
Let $i\in[a,b]$ with $\la_i<m$, and $\mu:=\la+\e_i$. Then
$$y^\mu e_\mu\psi_{\la;i} e_\la=\psi_{\la;i} y^\la e_\la.
$$
\end{Lemma}
\begin{proof}
Recall the notation of \S\ref{SSCommL}, we have that $y^\mu e_\mu\psi_{\la;i} e_\la$ equals
\begin{align*}
&
\Big[\prod_{j\in[a,b]} \iota^\mu_{j,-}(1_{j^{(m-\mu_j)}})
\iota^\mu_{j,+}(1_{j^{(\mu_j)}})\Big]
\iota_{b+1}^\mu(y_{0,m}1_{(b+1)^{(m)}})
\psi_{\la;i} e_\la
\\
=& \iota^\mu_{i,-}(1_{i^{(m-\mu_i)}})
\iota^\mu_{i,+}(1_{i^{(\mu_i)}})
\Big[\prod_{j\in[a,b]\setminus\{i\}} \iota^\mu_{j,-}(1_{j^{(m-\mu_j)}})
\iota^\mu_{j,+}(1_{j^{(\mu_j)}})\Big]
\\
&\times \iota_{b+1}^\mu(y_{0,m}1_{(b+1)^{(m)}})
\psi_{\la;i} e_\la
\\
=& \iota^\mu_{i,-}(1_{i^{(m-\mu_i)}})
\iota^\mu_{i,+}(1_{i^{(\mu_i)}})
\psi_{\la;i}
\Big[\prod_{j\in[a,b]\setminus\{i\}} \iota^\la_{j,-}(1_{j^{(m-\la_j)}})
\iota^\la_{j,+}(1_{j^{(\la_j)}})\Big]
\\
&\times \iota_{b+1}^\la(y_{0,m})\iota_{b+1}^\la(1_{(b+1)^{(m)}})
 e_\la
 \\
=& \iota^\mu_{i,-}(1_{i^{(m-\mu_i)}})
\iota^\mu_{j,+}(1_{j^{(\mu_j)}})
\psi_{\la;i}
y^\la e_\la=\psi_{\la;i} y^\la e_\la,
\end{align*}
where we have used Lemma~\ref{LPull} for the second equality and Lemma~\ref{LDivPower} for the last equality.
\end{proof}

\begin{Example}
We illustrate Lemma~\ref{L290316G} diagramatically.  Let $a=1$, $b=2$, $m=4$, $\mu=(2,3)$ and $\la=(1,3)$ so that $i=1$ using the notation of Lemma~\ref{L290316G}.  Then, the lemma claims the following equality.
$$
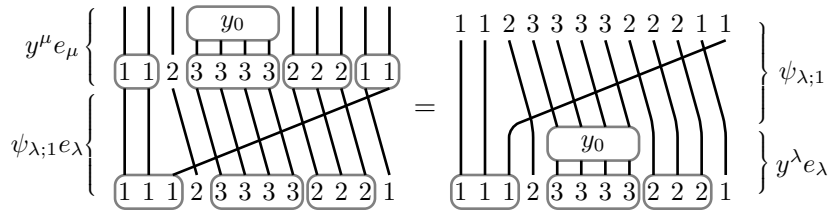
\begin{braid}\tikzset{baseline=0mm}
\draw(-0.5,2.8) node{$\left.{
\begin{matrix}
\\
\\
\\
\\
\\
\\
\end{matrix}
}\right\{$};
\draw(-0.5,-1.3) node{$\left.{
\begin{matrix}
\\
\\
\\
\\
\\
\\
\\
\\
\end{matrix}
}\right\{$};
\draw(-2,2.8) node{\small $y^\mu e_\mu$};
\draw(-2.2,-1.3) node{\small $\psi_{\la;1}e_\la$};
\draw(1,1) node[above]{\small $1$};
\draw(2,1) node[above]{\small $1$};
\draw(3,1) node[above]{\small $2$};
\draw(4,1) node[above]{\small $3$};
\draw(5,1) node[above]{\small $3$};
\draw(6,1) node[above]{\small $3$};
\draw(7,1) node[above]{\small $3$};
\draw(8,1) node[above]{\small $2$};
\draw(9,1) node[above]{\small $2$};
\draw(10,1) node[above]{\small $2$};
\draw(11,1) node[above]{\small $1$};
\draw(12,1) node[above]{\small $1$};
\draw [rounded corners,color=gray] (0.6,2.5)--(2.4,2.5)--(2.4,1.1)--(0.6,1.1)--cycle;
\draw [rounded corners,color=gray] (3.6,4.5)--(7.4,4.5)--(7.4,3.1)--(3.6,3.1)--cycle;
\draw [rounded corners,color=gray] (3.6,2.5)--(7.4,2.5)--(7.4,1.1)--(3.6,1.1)--cycle;
\draw [rounded corners,color=gray] (7.6,2.5)--(10.4,2.5)--(10.4,1.1)--(7.6,1.1)--cycle;
\draw [rounded corners,color=gray] (10.6,2.5)--(12.4,2.5)--(12.4,1.1)--(10.6,1.1)--cycle;
\draw [rounded corners,color=gray] (0.6,-2.5)--(3.4,-2.5)--(3.4,-3.9)--(0.6,-3.9)--cycle;
\draw [rounded corners,color=gray] (4.6,-2.5)--(8.4,-2.5)--(8.4,-3.9)--(4.6,-3.9)--cycle;
\draw [rounded corners,color=gray] (8.6,-2.5)--(11.4,-2.5)--(11.4,-3.9)--(8.6,-3.9)--cycle;
\draw(1,1.1)--(1,-2.5);
\draw(2,1.1)--(2,-2.5);
\draw(3,1.1)--(4,-2.5);
\draw(4,1.1)--(5,-2.5);
\draw(5,1.1)--(6,-2.5);
\draw(6,1.1)--(7,-2.5);
\draw(7,1.1)--(8,-2.5);
\draw(8,1.1)--(9,-2.5);
\draw(9,1.1)--(10,-2.5);
\draw(10,1.1)--(11,-2.5);
\draw(11,1.1)--(12,-2.5);
\draw(12,1.1)--(3,-2.5);
\draw(1,4.5)--(1,2.5);
\draw(2,4.5)--(2,2.5);
\draw(3,4.5)--(3,2.5);
\draw(4,3.1)--(4,2.5);
\draw(5,3.1)--(5,2.5);
\draw(6,3.1)--(6,2.5);
\draw(7,3.1)--(7,2.5);
\draw(8,4.5)--(8,2.5);
\draw(9,4.5)--(9,2.5);
\draw(10,4.5)--(10,2.5);
\draw(11,4.5)--(11,2.5);
\draw(12,4.5)--(12,2.5);
\draw(5.5,3) node[above]{\small $y_0$};
\draw(1,-4) node[above]{\small $1$};
\draw(2,-4) node[above]{\small $1$};
\draw(3,-4) node[above]{\small $1$};
\draw(4,-4) node[above]{\small $2$};
\draw(5,-4) node[above]{\small $3$};
\draw(6,-4) node[above]{\small $3$};
\draw(7,-4) node[above]{\small $3$};
\draw(8,-4) node[above]{\small $3$};
\draw(9,-4) node[above]{\small $2$};
\draw(10,-4) node[above]{\small $2$};
\draw(11,-4) node[above]{\small $2$};
\draw(12,-4) node[above]{\small $1$};
\end{braid}
=
\begin{braid}\tikzset{baseline=0mm}
\draw(13.5,1.7) node{$\left.{
\begin{matrix}
\\
\\
\\
\\
\\
\\
\\
\\
\end{matrix}
}\right\}$};
\draw(13.5,-2) node{$\left.{
\begin{matrix}
\\
\\
\\
\\
\\
\end{matrix}
}\right\}$};
\draw(15.2,1.7) node{\small $\psi_{\la;1}$};
\draw(15.2,-2) node{\small $y^\la e_\la$};
\draw(1,3) node[above]{\small $1$};
\draw(2,3) node[above]{\small $1$};
\draw(3,3) node[above]{\small $2$};
\draw(4,3) node[above]{\small $3$};
\draw(5,3) node[above]{\small $3$};
\draw(6,3) node[above]{\small $3$};
\draw(7,3) node[above]{\small $3$};
\draw(8,3) node[above]{\small $2$};
\draw(9,3) node[above]{\small $2$};
\draw(10,3) node[above]{\small $2$};
\draw(11,3) node[above]{\small $1$};
\draw(12,3) node[above]{\small $1$};
\draw [rounded corners,color=gray] (4.6,-0.5)--(8.4,-0.5)--(8.4,-1.9)--(4.6,-1.9)--cycle;
\draw [rounded corners,color=gray] (0.6,-2.5)--(3.4,-2.5)--(3.4,-3.9)--(0.6,-3.9)--cycle;
\draw [rounded corners,color=gray] (4.6,-2.5)--(8.4,-2.5)--(8.4,-3.9)--(4.6,-3.9)--cycle;
\draw [rounded corners,color=gray] (8.6,-2.5)--(11.4,-2.5)--(11.4,-3.9)--(8.6,-3.9)--cycle;
\draw(5,-1.9)--(5,-2.5);
\draw(6,-1.9)--(6,-2.5);
\draw(7,-1.9)--(7,-2.5);
\draw(8,-1.9)--(8,-2.5);
%
\draw(1,3.1)--(1,-0.5)--(1,-2.5);
\draw(2,3.1)--(2,-0.5)--(2,-2.5);
\draw(3,3.1)--(4,-0.5)--(4,-2.5);
\draw(4,3.1)--(5,-0.5);
\draw(5,3.1)--(6,-0.5);
\draw(6,3.1)--(7,-0.5);
\draw(7,3.1)--(8,-0.5);
\draw(8,3.1)--(9,-0.5)--(9,-2.5);
\draw(9,3.1)--(10,-0.5)--(10,-2.5);
\draw(10,3.1)--(11,-0.5)--(11,-2.5);
\draw(11,3.1)--(12,-0.5)--(12,-2.5);
\draw(12,3.1)--(3,-0.5)--(3,-2.5);
\draw(6.5,-2) node[above]{\small $y_0$};
\draw(1,-4) node[above]{\small $1$};
\draw(2,-4) node[above]{\small $1$};
\draw(3,-4) node[above]{\small $1$};
\draw(4,-4) node[above]{\small $2$};
\draw(5,-4) node[above]{\small $3$};
\draw(6,-4) node[above]{\small $3$};
\draw(7,-4) node[above]{\small $3$};
\draw(8,-4) node[above]{\small $3$};
\draw(9,-4) node[above]{\small $2$};
\draw(10,-4) node[above]{\small $2$};
\draw(11,-4) node[above]{\small $2$};
\draw(12,-4) node[above]{\small $1$};
\end{braid}
$$
\end{Example}


\subsection{$g$ is a chain map}
Recall that we have fixed  $\la\in\La(n)$ and $\bga\in \bLa(n+1)$.

\begin{Lemma} \label{L140416G} 
Suppose that $i\in [a,b]$ is such that $\la_i<m$ and $\mu:=\la+\e_i=\la^\bga$. Set $\bar\theta=(m-1)\al$, $\bar\mu:=\mu-\ga^{(m)}$, and $\bar\la:=\la-\ga^{(m)}$. Then $\psi_{z(\bga,m)}y^\mu e_\mu\psi_{\la;i}e_\la$ equals
$$
\left\{
\begin{array}{ll}
\begin{array}{ll}
\psi_{\bga-\bolde_i^{m};m,i}\psi_{z(\bga-\bolde_i^{m},m)}y^\la e_\la
\\
-(y^{\bar\mu}e_{\bar\mu}\psi_{\bar\la;i}e_{\bar\la} \circ 1_\al)\psi_{z^{\ga^{(m)}}_\la}y_{r_{b+1}(\la)}^{m-1}e_\la
\end{array}
 &\hbox{if $\ga^{(m)}_i=1$,}\\
(y^{\bar\mu}e_{\bar\mu}\psi_{\bar\la;i}e_{\bar\la} \circ 1_\al)\psi_{z^{\ga^{(m)}}_\la}y_{r_{b+1}(\la)}^{m-1}e_\la &\hbox{if $\ga^{(m)}_i=0$.}
\end{array}
\right.
$$
\end{Lemma}
\begin{proof}
Using Lemma~\ref{L290316G}, we get
$$
\psi_{z(\bga,m)}y^\mu e_\mu\psi_{\la;i}e_\la=
\psi_{z(\bga,m)}\psi_{\la;i}y^\la e_\la
$$
Recalling (\ref{EQIG'}),(\ref{EQIG}), the permutation $z(\bga,m)$ maps $q^j$ to $p^j$ for all $j\in[a,b+1]$, and the elements of
$[1,d]\setminus Q$ increasingly to the elements of $[1,d-l]$.

{\em Case 1: $\ga^{(m)}_i=1$}. 
In this case we have $q^i=r_i^+(\la)$. So the KLR diagram $D$ of $\psi_{z(\bga,m)}\psi_{\la;i}1^{\bj^\la}$ has an $i$-string $S$ from the position $r_i^-(\la)$ in the bottom to the position $p^i$ in the top, and the only $(i,i)$-crossings in $D$ will be the crossings of the string $S$ with the $\la_i$ strings which originate in the positions $R_i^+(\la)$ in the bottom.
The $(i+1)$-string $T$ in $D$ originating in the position $p_{i+1}$ in the top is to the left of all $(i,i)$ crossings. Pulling $T$ to the right produces error terms, which arise from opening $(i,i)$-crossings, but all of them, except the last one, amount to zero, when multiplied on the right by $y^\la e_\la$. The last error term is equal to
$
-\iota_{\bar\theta,\al}(\psi_{\bar\la;i} \otimes 1_\al)
\psi_{z^{\ga^{(m)}}_\la},
$
and the result of pulling $T$ past all $(i,i)$-crossings gives $\psi_{\bga-\bolde_i^{m};m,i}\psi_{z(\bga-\bolde_i^{m},m)}$.
Multiplying on the left by $y^\la e_\la$, gives
$$
-\iota_{\bar\theta,\al}(\psi_{\bar\la;i} \otimes 1_\al)
\psi_{z^{\ga^{(m)}}_\la}y^\la e_\la+
\psi_{\bga-\bolde_i^{m};m,i}\psi_{z(\bga-\bolde_i^{m},m)}y^\la e_\la,
$$
and it remains to observe using Lemma~\ref{LPull} that
$$
-\iota_{\bar\theta,\al}(\psi_{\bar\la;i} \otimes 1_\al)
\psi_{z^{\ga^{(m)}}_\la}y^\la e_\la=
-(y^{\bar\mu}e_{\bar\mu}\psi_{\bar\la;i}e_{\bar\la} \circ 1_\al)\psi_{z^{\ga^{(m)}}_\la}y_{r_{b+1}(\la)}^{m-1}e_\la.
$$

{\em Case 2: $\ga^{(m)}_i=0$.}
Let  $D$ be the KLR diagram of $\psi_{z(\bga,m)}\psi_{\la;i}1_{\bj^\la}$.
Let $S$ be the $i$-string originating in the position $r^-_{i}(\la)$ in the bottom row of $D$, and $T$ be the $(i+1)$-string originating in the position $p^{i+1}$ in the top row of $D$. The quadratic relation on these strings produces a linear combination of two diagrams, one with a dot on $S$ and the other with a dot on $T$. The term containing a dot on $T$ produces a term which is zero after multiplying on the right by $y^\la e^\la$. The term containing a dot on $S$, when multiplied on the right by $y^\la e^\la$, yields
$$
\iota_{\bar\theta,\al}(\psi_{\bar\la;i} \otimes 1_\al)\psi_{z^{\ga^{(m)}}_\la} y^\la e_\la
=(y^{\bar\mu}e_{\bar\mu}\psi_{\bar\la;i}e_{\bar\la} \circ 1_\al)\psi_{z^{\ga^{(m)}}_\la} y_{r_{b+1}(\la)}^{m-1}e_\la,
$$
where we have used Lemmas~\ref{LPull} and~\ref{LDivPower} to deduce the last equality.
\end{proof}

\begin{Example}
We illustrate Lemma~\ref{L140416G} diagramatically.  Using the notation of the lemma, let $a=1$, $b=2$, $\mu=(2,2)$, $\la=(1,2)$, $m=4$ and $i=1$.  Then, if $\bga =\big((1,0),(1,1),(0,1),(0,0)\big)$, we are in the $\gamma_i^{(m)}=0$ case and the lemma claims the following equality.
$$
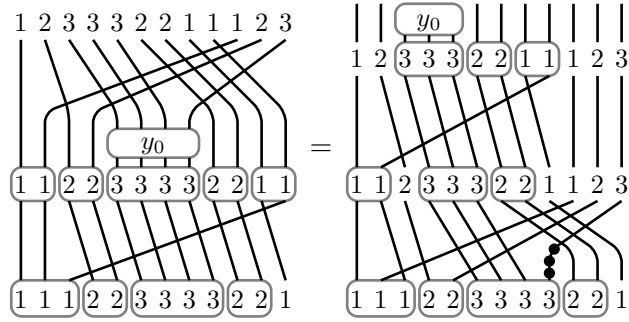
\begin{braid}\tikzset{baseline=-7mm}
\draw(1,2.5) node[above]{\small $1$};
\draw(2,2.5) node[above]{\small $2$};
\draw(3,2.5) node[above]{\small $3$};
\draw(4,2.5) node[above]{\small $3$};
\draw(5,2.5) node[above]{\small $3$};
\draw(6,2.5) node[above]{\small $2$};
\draw(7,2.5) node[above]{\small $2$};
\draw(8,2.5) node[above]{\small $1$};
\draw(9,2.5) node[above]{\small $1$};
\draw(10,2.5) node[above]{\small $1$};
\draw(11,2.5) node[above]{\small $2$};
\draw(12,2.5) node[above]{\small $3$};
\draw [rounded corners,color=gray] (0.6,-2.7)--(2.4,-2.7)--(2.4,-4.1)--(0.6,-4.1)--cycle;
\draw [rounded corners,color=gray] (2.6,-2.7)--(4.4,-2.7)--(4.4,-4.1)--(2.6,-4.1)--cycle;
\draw [rounded corners,color=gray] (4.6,-2.7)--(8.4,-2.7)--(8.4,-4.1)--(4.6,-4.1)--cycle;
\draw [rounded corners,color=gray] (8.6,-2.7)--(10.4,-2.7)--(10.4,-4.1)--(8.6,-4.1)--cycle;
\draw [rounded corners,color=gray] (10.6,-2.7)--(12.4,-2.7)--(12.4,-4.1)--(10.6,-4.1)--cycle;
\draw [rounded corners,color=gray] (4.6,-1.1)--(8.4,-1.1)--(8.4,-2.3)--(4.6,-2.3)--cycle;
\draw(6.5,-2.6) node[above]{\small $y_0$};
\draw(2,2.6)--(3,-0.4)--(3,-2.7);
\draw(3,2.6)--(5,-0.4)--(5,-1.1);
\draw(8,2.6)--(11,-0.4)--(11,-2.7);
\draw(4,2.6)--(6,-0.4)--(6,-1.1);
\draw(6,2.6)--(9,-0.4)--(9,-2.7);
\draw(9,2.6)--(12,-0.4)--(12,-2.7);
\draw(1,2.6)--(1,-0.4)--(1,-2.7);
\draw(5,2.6)--(7,-0.4)--(7,-1.1);
\draw(7,2.6)--(10,-0.4)--(10,-2.7);
\draw(10,2.6)--(2,-0.4)--(2,-2.7);
\draw(11,2.6)--(4,-0.4)--(4,-2.7);
\draw(12,2.6)--(8,-0.4)--(8,-1.1);
\draw(5,-2.3)--(5,-2.7);
\draw(6,-2.3)--(6,-2.7);
\draw(7,-2.3)--(7,-2.7);
\draw(8,-2.3)--(8,-2.7);
\draw(1,-3.4) node{\small $1$};
\draw(2,-3.4) node{\small $1$};
\draw(3,-3.4) node{\small $2$};
\draw(4,-3.4) node{\small $2$};
\draw(5,-3.4) node{\small $3$};
\draw(6,-3.4) node{\small $3$};
\draw(7,-3.4) node{\small $3$};
\draw(8,-3.4) node{\small $3$};
\draw(9,-3.4) node{\small $2$};
\draw(10,-3.4) node{\small $2$};
\draw(11,-3.4) node{\small $1$};
\draw(12,-3.4) node{\small $1$};
\draw(1,-9) node[above]{\small $1$};
\draw(2,-9) node[above]{\small $1$};
\draw(3,-9) node[above]{\small $1$};
\draw(4,-9) node[above]{\small $2$};
\draw(5,-9) node[above]{\small $2$};
\draw(6,-9) node[above]{\small $3$};
\draw(7,-9) node[above]{\small $3$};
\draw(8,-9) node[above]{\small $3$};
\draw(9,-9) node[above]{\small $3$};
\draw(10,-9) node[above]{\small $2$};
\draw(11,-9) node[above]{\small $2$};
\draw(12,-9) node[above]{\small $1$};
\draw [rounded corners,color=gray] (0.6,-8.9)--(3.4,-8.9)--(3.4,-7.4)--(0.6,-7.4)--cycle;
\draw [rounded corners,color=gray] (3.6,-8.9)--(5.4,-8.9)--(5.4,-7.4)--(3.6,-7.4)--cycle;
\draw [rounded corners,color=gray] (5.6,-8.9)--(9.4,-8.9)--(9.4,-7.4)--(5.6,-7.4)--cycle;
\draw [rounded corners,color=gray] (9.6,-8.9)--(11.4,-8.9)--(11.4,-7.4)--(9.6,-7.4)--cycle;
\draw(1,-4.1)--(1,-7.4);
\draw(2,-4.1)--(2,-7.4);
\draw(3,-4.1)--(4,-7.4);
\draw(4,-4.1)--(5,-7.4);
\draw(5,-4.1)--(6,-7.4);
\draw(6,-4.1)--(7,-7.4);
\draw(7,-4.1)--(8,-7.4);
\draw(8,-4.1)--(9,-7.4);
\draw(9,-4.1)--(10,-7.4);
\draw(10,-4.1)--(11,-7.4);
\draw(11,-4.1)--(12,-7.4);
\draw(12,-4.1)--(3,-7.4);
\end{braid}
=
\begin{braid}\tikzset{baseline=-7mm}
\draw(1,1) node[above]{\small $1$};
\draw(2,1) node[above]{\small $2$};
\draw(3,1) node[above]{\small $3$};
\draw(4,1) node[above]{\small $3$};
\draw(5,1) node[above]{\small $3$};
\draw(6,1) node[above]{\small $2$};
\draw(7,1) node[above]{\small $2$};
\draw(8,1) node[above]{\small $1$};
\draw(9,1) node[above]{\small $1$};
\draw(10,1) node[above]{\small $1$};
\draw(11,1) node[above]{\small $2$};
\draw(12,1) node[above]{\small $3$};
\draw [rounded corners,color=gray] (0.6,-2.7)--(2.4,-2.7)--(2.4,-4.1)--(0.6,-4.1)--cycle;
\draw [rounded corners,color=gray] (3.6,-2.7)--(6.4,-2.7)--(6.4,-4.1)--(3.6,-4.1)--cycle;
\draw [rounded corners,color=gray] (6.6,-2.7)--(8.4,-2.7)--(8.4,-4.1)--(6.6,-4.1)--cycle;
\draw [rounded corners,color=gray] (2.6,4.1)--(5.4,4.1)--(5.4,2.8)--(2.6,2.8)--cycle;
\draw(4,2.5) node[above]{\small $y_0$};
\draw [rounded corners,color=gray] (2.6,1.1)--(5.4,1.1)--(5.4,2.5)--(2.6,2.5)--cycle;
\draw [rounded corners,color=gray] (5.6,1.1)--(7.4,1.1)--(7.4,2.5)--(5.6,2.5)--cycle;
\draw [rounded corners,color=gray] (7.6,1.1)--(9.4,1.1)--(9.4,2.5)--(7.6,2.5)--cycle;
\draw(1,4.1)--(1,2.5);
\draw(2,4.1)--(2,2.5);
\draw(3,2.8)--(3,2.5);
\draw(4,2.8)--(4,2.5);
\draw(5,2.8)--(5,2.5);
\draw(6,4.1)--(6,2.5);
\draw(7,4.1)--(7,2.5);
\draw(8,4.1)--(8,2.5);
\draw(9,4.1)--(9,2.5);
\draw(10,4.1)--(10,2.5);
\draw(11,4.1)--(11,2.5);
\draw(12,4.1)--(12,2.5);
\draw(1,1.1)--(1,-2.7);
\draw(2,1.1)--(3,-2.7);
\draw(3,1.1)--(4,-2.7);
\draw(4,1.1)--(5,-2.7);
\draw(5,1.1)--(6,-2.7);
\draw(6,1.1)--(7,-2.7);
\draw(7,1.1)--(8,-2.7);
\draw(8,1.1)--(9,-2.7);
\draw(9,1.1)--(2,-2.7);
\draw(10,1.1)--(10,-2.7);
\draw(11,1.1)--(11,-2.7);
\draw(12,1.1)--(12,-2.7);
\draw(1,-3.4) node{\small $1$};
\draw(2,-3.4) node{\small $1$};
\draw(3,-3.4) node{\small $2$};
\draw(4,-3.4) node{\small $3$};
\draw(5,-3.4) node{\small $3$};
\draw(6,-3.4) node{\small $3$};
\draw(7,-3.4) node{\small $2$};
\draw(8,-3.4) node{\small $2$};
\draw(9,-3.4) node{\small $1$};
\draw(10,-3.4) node{\small $1$};
\draw(11,-3.4) node{\small $2$};
\draw(12,-3.4) node{\small $3$};
\draw(1,-9) node[above]{\small $1$};
\draw(2,-9) node[above]{\small $1$};
\draw(3,-9) node[above]{\small $1$};
\draw(4,-9) node[above]{\small $2$};
\draw(5,-9) node[above]{\small $2$};
\draw(6,-9) node[above]{\small $3$};
\draw(7,-9) node[above]{\small $3$};
\draw(8,-9) node[above]{\small $3$};
\draw(9,-9) node[above]{\small $3$};
\draw(10,-9) node[above]{\small $2$};
\draw(11,-9) node[above]{\small $2$};
\draw(12,-9) node[above]{\small $1$};
\draw [rounded corners,color=gray] (0.6,-8.9)--(3.4,-8.9)--(3.4,-7.4)--(0.6,-7.4)--cycle;
\draw [rounded corners,color=gray] (3.6,-8.9)--(5.4,-8.9)--(5.4,-7.4)--(3.6,-7.4)--cycle;
\draw [rounded corners,color=gray] (5.6,-8.9)--(9.4,-8.9)--(9.4,-7.4)--(5.6,-7.4)--cycle;
\draw [rounded corners,color=gray] (9.6,-8.9)--(11.4,-8.9)--(11.4,-7.4)--(9.6,-7.4)--cycle;
\draw(1,-4.1)--(1,-7.4);
\draw(2,-4.1)--(3,-7.4);
\draw(3,-4.1)--(4,-7.4);
\draw(4,-4.1)--(6,-7.4);
\draw(5,-4.1)--(7,-7.4);
\draw(6,-4.1)--(8,-7.4);
\draw(7,-4.1)--(10,-6.2)--(10,-7.4);
\draw(8,-4.1)--(11,-6.2)--(11,-7.4);
\draw(9,-4.1)--(12,-6.2)--(12,-7.4);
\draw(10,-4.1)--(2,-7.4);
\draw(11,-4.1)--(5,-7.4);
\draw(12,-4.1)--(9,-6.2)--(9,-7.4);
\blackdot(9,-7.1);
\blackdot(9,-6.6);
\blackdot(9.2,-6.1);
\end{braid}
$$
If instead we let $\bga =\big((0,0),(1,1),(0,1),(1,0)\big)$, we are in the $\gamma_i^{(m)}=1$ case and the lemma claims the following equality.
$$
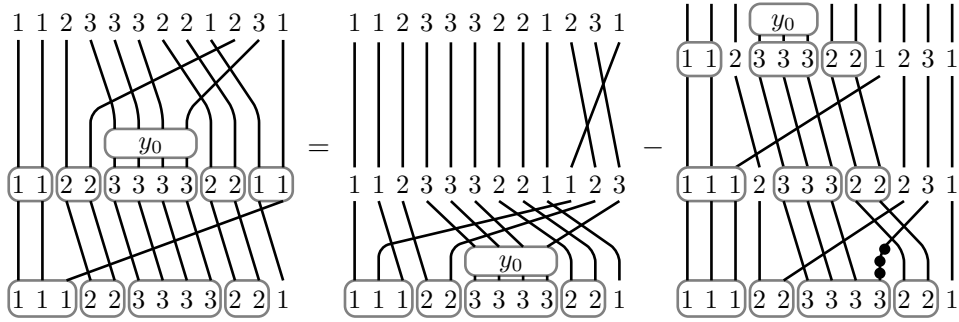
\begin{braid}\tikzset{baseline=-7mm}
\draw(1,2.5) node[above]{\small $1$};
\draw(2,2.5) node[above]{\small $1$};
\draw(3,2.5) node[above]{\small $2$};
\draw(4,2.5) node[above]{\small $3$};
\draw(5,2.5) node[above]{\small $3$};
\draw(6,2.5) node[above]{\small $3$};
\draw(7,2.5) node[above]{\small $2$};
\draw(8,2.5) node[above]{\small $2$};
\draw(9,2.5) node[above]{\small $1$};
\draw(10,2.5) node[above]{\small $2$};
\draw(11,2.5) node[above]{\small $3$};
\draw(12,2.5) node[above]{\small $1$};
\draw [rounded corners,color=gray] (0.6,-2.7)--(2.4,-2.7)--(2.4,-4.1)--(0.6,-4.1)--cycle;
\draw [rounded corners,color=gray] (2.6,-2.7)--(4.4,-2.7)--(4.4,-4.1)--(2.6,-4.1)--cycle;
\draw [rounded corners,color=gray] (4.6,-2.7)--(8.4,-2.7)--(8.4,-4.1)--(4.6,-4.1)--cycle;
\draw [rounded corners,color=gray] (8.6,-2.7)--(10.4,-2.7)--(10.4,-4.1)--(8.6,-4.1)--cycle;
\draw [rounded corners,color=gray] (10.6,-2.7)--(12.4,-2.7)--(12.4,-4.1)--(10.6,-4.1)--cycle;
\draw [rounded corners,color=gray] (4.6,-1.1)--(8.4,-1.1)--(8.4,-2.3)--(4.6,-2.3)--cycle;
\draw(6.5,-2.6) node[above]{\small $y_0$};
\draw(2,2.6)--(2,-0.4)--(2,-2.7);
\draw(3,2.6)--(3,-0.4)--(3,-2.7);
\draw(8,2.6)--(10,-0.4)--(10,-2.7);
\draw(4,2.6)--(5,-0.4)--(5,-1.1);
\draw(6,2.6)--(7,-0.4)--(7,-1.1);
\draw(9,2.6)--(11,-0.4)--(11,-2.7);
\draw(1,2.6)--(1,-0.4)--(1,-2.7);
\draw(5,2.6)--(6,-0.4)--(6,-1.1);
\draw(7,2.6)--(9,-0.4)--(9,-2.7);
\draw(10,2.6)--(4,-0.4)--(4,-2.7);
\draw(11,2.6)--(8,-0.4)--(8,-1.1);
\draw(12,2.6)--(12,-0.4)--(12,-2.7);
\draw(5,-2.3)--(5,-2.7);
\draw(6,-2.3)--(6,-2.7);
\draw(7,-2.3)--(7,-2.7);
\draw(8,-2.3)--(8,-2.7);
\draw(1,-3.4) node{\small $1$};
\draw(2,-3.4) node{\small $1$};
\draw(3,-3.4) node{\small $2$};
\draw(4,-3.4) node{\small $2$};
\draw(5,-3.4) node{\small $3$};
\draw(6,-3.4) node{\small $3$};
\draw(7,-3.4) node{\small $3$};
\draw(8,-3.4) node{\small $3$};
\draw(9,-3.4) node{\small $2$};
\draw(10,-3.4) node{\small $2$};
\draw(11,-3.4) node{\small $1$};
\draw(12,-3.4) node{\small $1$};
\draw(1,-9) node[above]{\small $1$};
\draw(2,-9) node[above]{\small $1$};
\draw(3,-9) node[above]{\small $1$};
\draw(4,-9) node[above]{\small $2$};
\draw(5,-9) node[above]{\small $2$};
\draw(6,-9) node[above]{\small $3$};
\draw(7,-9) node[above]{\small $3$};
\draw(8,-9) node[above]{\small $3$};
\draw(9,-9) node[above]{\small $3$};
\draw(10,-9) node[above]{\small $2$};
\draw(11,-9) node[above]{\small $2$};
\draw(12,-9) node[above]{\small $1$};
\draw [rounded corners,color=gray] (0.6,-8.9)--(3.4,-8.9)--(3.4,-7.4)--(0.6,-7.4)--cycle;
\draw [rounded corners,color=gray] (3.6,-8.9)--(5.4,-8.9)--(5.4,-7.4)--(3.6,-7.4)--cycle;
\draw [rounded corners,color=gray] (5.6,-8.9)--(9.4,-8.9)--(9.4,-7.4)--(5.6,-7.4)--cycle;
\draw [rounded corners,color=gray] (9.6,-8.9)--(11.4,-8.9)--(11.4,-7.4)--(9.6,-7.4)--cycle;
\draw(1,-4.1)--(1,-7.4);
\draw(2,-4.1)--(2,-7.4);
\draw(3,-4.1)--(4,-7.4);
\draw(4,-4.1)--(5,-7.4);
\draw(5,-4.1)--(6,-7.4);
\draw(6,-4.1)--(7,-7.4);
\draw(7,-4.1)--(8,-7.4);
\draw(8,-4.1)--(9,-7.4);
\draw(9,-4.1)--(10,-7.4);
\draw(10,-4.1)--(11,-7.4);
\draw(11,-4.1)--(12,-7.4);
\draw(12,-4.1)--(3,-7.4);
\end{braid}
=
\begin{braid}\tikzset{baseline=-7mm}
\draw(1,2.5) node[above]{\small $1$};
\draw(2,2.5) node[above]{\small $1$};
\draw(3,2.5) node[above]{\small $2$};
\draw(4,2.5) node[above]{\small $3$};
\draw(5,2.5) node[above]{\small $3$};
\draw(6,2.5) node[above]{\small $3$};
\draw(7,2.5) node[above]{\small $2$};
\draw(8,2.5) node[above]{\small $2$};
\draw(9,2.5) node[above]{\small $1$};
\draw(10,2.5) node[above]{\small $2$};
\draw(11,2.5) node[above]{\small $3$};
\draw(12,2.5) node[above]{\small $1$};
\draw(1,2.5)--(1,-2.7);
\draw(2,2.5)--(2,-2.7);
\draw(3,2.5)--(3,-2.7);
\draw(4,2.5)--(4,-2.7);
\draw(5,2.5)--(5,-2.7);
\draw(6,2.5)--(6,-2.7);
\draw(7,2.5)--(7,-2.7);
\draw(8,2.5)--(8,-2.7);
\draw(9,2.5)--(9,-2.7);
\draw(10,2.5)--(11,-2.7);
\draw(11,2.5)--(12,-2.7);
\draw(12,2.5)--(10,-2.7);
\draw(1,-3.4) node{\small $1$};
\draw(2,-3.4) node{\small $1$};
\draw(3,-3.4) node{\small $2$};
\draw(4,-3.4) node{\small $3$};
\draw(5,-3.4) node{\small $3$};
\draw(6,-3.4) node{\small $3$};
\draw(7,-3.4) node{\small $2$};
\draw(8,-3.4) node{\small $2$};
\draw(9,-3.4) node{\small $1$};
\draw(10,-3.4) node{\small $1$};
\draw(11,-3.4) node{\small $2$};
\draw(12,-3.4) node{\small $3$};
\draw [rounded corners,color=gray] (5.6,-6)--(9.4,-6)--(9.4,-7.2)--(5.6,-7.2)--cycle;
\draw(7.5,-7.5) node[above]{\small $y_0$};
\draw(1,-9) node[above]{\small $1$};
\draw(2,-9) node[above]{\small $1$};
\draw(3,-9) node[above]{\small $1$};
\draw(4,-9) node[above]{\small $2$};
\draw(5,-9) node[above]{\small $2$};
\draw(6,-9) node[above]{\small $3$};
\draw(7,-9) node[above]{\small $3$};
\draw(8,-9) node[above]{\small $3$};
\draw(9,-9) node[above]{\small $3$};
\draw(10,-9) node[above]{\small $2$};
\draw(11,-9) node[above]{\small $2$};
\draw(12,-9) node[above]{\small $1$};
\draw [rounded corners,color=gray] (0.6,-8.9)--(3.4,-8.9)--(3.4,-7.4)--(0.6,-7.4)--cycle;
\draw [rounded corners,color=gray] (3.6,-8.9)--(5.4,-8.9)--(5.4,-7.4)--(3.6,-7.4)--cycle;
\draw [rounded corners,color=gray] (5.6,-8.9)--(9.4,-8.9)--(9.4,-7.4)--(5.6,-7.4)--cycle;
\draw [rounded corners,color=gray] (9.6,-8.9)--(11.4,-8.9)--(11.4,-7.4)--(9.6,-7.4)--cycle;
\draw(1,-4.1)--(1,-7.4);
\draw(2,-4.1)--(3,-7.4);
\draw(3,-4.1)--(4,-7.4);
\draw(4,-4.1)--(6,-6);
\draw(5,-4.1)--(7,-6);
\draw(6,-4.1)--(8,-6);
\draw(7,-4.1)--(10,-6.2)--(10,-7.4);
\draw(8,-4.1)--(11,-6.2)--(11,-7.4);
\draw(9,-4.1)--(12,-6.2)--(12,-7.4);
\draw(10,-4.1)--(2,-6)--(2,-7.4);
\draw(11,-4.1)--(5,-6)--(5,-7.4);
\draw(12,-4.1)--(9,-6);
\draw(6,-7.2)--(6,-7.4);
\draw(7,-7.2)--(7,-7.4);
\draw(8,-7.2)--(8,-7.4);
\draw(9,-7.2)--(9,-7.4);
\end{braid}
-
\begin{braid}\tikzset{baseline=-7mm}
\draw(1,1) node[above]{\small $1$};
\draw(2,1) node[above]{\small $1$};
\draw(3,1) node[above]{\small $2$};
\draw(4,1) node[above]{\small $3$};
\draw(5,1) node[above]{\small $3$};
\draw(6,1) node[above]{\small $3$};
\draw(7,1) node[above]{\small $2$};
\draw(8,1) node[above]{\small $2$};
\draw(9,1) node[above]{\small $1$};
\draw(10,1) node[above]{\small $2$};
\draw(11,1) node[above]{\small $3$};
\draw(12,1) node[above]{\small $1$};
\draw [rounded corners,color=gray] (3.6,4.1)--(6.4,4.1)--(6.4,2.8)--(3.6,2.8)--cycle;
\draw(5,2.5) node[above]{\small $y_0$};
\draw(1,4.1)--(1,2.5);
\draw(2,4.1)--(2,2.5);
\draw(3,4.1)--(3,2.5);
\draw(4,2.8)--(4,2.5);
\draw(5,2.8)--(5,2.5);
\draw(6,2.8)--(6,2.5);
\draw(7,4.1)--(7,2.5);
\draw(8,4.1)--(8,2.5);
\draw(9,4.1)--(9,2.5);
\draw(10,4.1)--(10,2.5);
\draw(11,4.1)--(11,2.5);
\draw(12,4.1)--(12,2.5);
\draw(1,1.1)--(1,-2.7);
\draw(2,1.1)--(2,-2.7);
\draw(3,1.1)--(4,-2.7);
\draw(4,1.1)--(5,-2.7);
\draw(5,1.1)--(6,-2.7);
\draw(6,1.1)--(7,-2.7);
\draw(7,1.1)--(8,-2.7);
\draw(8,1.1)--(9,-2.7);
\draw(9,1.1)--(3,-2.7);
\draw(10,1.1)--(10,-2.7);
\draw(11,1.1)--(11,-2.7);
\draw(12,1.1)--(12,-2.7);
\draw(1,-3.4) node{\small $1$};
\draw(2,-3.4) node{\small $1$};
\draw(3,-3.4) node{\small $1$};
\draw(4,-3.4) node{\small $2$};
\draw(5,-3.4) node{\small $3$};
\draw(6,-3.4) node{\small $3$};
\draw(7,-3.4) node{\small $3$};
\draw(8,-3.4) node{\small $2$};
\draw(9,-3.4) node{\small $2$};
\draw(10,-3.4) node{\small $2$};
\draw(11,-3.4) node{\small $3$};
\draw(12,-3.4) node{\small $1$};
\draw [rounded corners,color=gray] (0.6,-4.1)--(3.4,-4.1)--(3.4,-2.7)--(0.6,-2.7)--cycle;
\draw [rounded corners,color=gray] (4.6,-4.1)--(7.4,-4.1)--(7.4,-2.7)--(4.6,-2.7)--cycle;
\draw [rounded corners,color=gray] (7.6,-4.1)--(9.4,-4.1)--(9.4,-2.7)--(7.6,-2.7)--cycle;
\draw(1,-9) node[above]{\small $1$};
\draw(2,-9) node[above]{\small $1$};
\draw(3,-9) node[above]{\small $1$};
\draw(4,-9) node[above]{\small $2$};
\draw(5,-9) node[above]{\small $2$};
\draw(6,-9) node[above]{\small $3$};
\draw(7,-9) node[above]{\small $3$};
\draw(8,-9) node[above]{\small $3$};
\draw(9,-9) node[above]{\small $3$};
\draw(10,-9) node[above]{\small $2$};
\draw(11,-9) node[above]{\small $2$};
\draw(12,-9) node[above]{\small $1$};
\draw [rounded corners,color=gray] (0.6,1.1)--(2.4,1.1)--(2.4,2.5)--(0.6,2.5)--cycle;
\draw [rounded corners,color=gray] (3.6,1.1)--(6.4,1.1)--(6.4,2.5)--(3.6,2.5)--cycle;
\draw [rounded corners,color=gray] (6.6,1.1)--(8.4,1.1)--(8.4,2.5)--(6.6,2.5)--cycle;
\draw [rounded corners,color=gray] (0.6,-8.9)--(3.4,-8.9)--(3.4,-7.4)--(0.6,-7.4)--cycle;
\draw [rounded corners,color=gray] (3.6,-8.9)--(5.4,-8.9)--(5.4,-7.4)--(3.6,-7.4)--cycle;
\draw [rounded corners,color=gray] (5.6,-8.9)--(9.4,-8.9)--(9.4,-7.4)--(5.6,-7.4)--cycle;
\draw [rounded corners,color=gray] (9.6,-8.9)--(11.4,-8.9)--(11.4,-7.4)--(9.6,-7.4)--cycle;
\draw(1,-4.1)--(1,-7.4);
\draw(2,-4.1)--(2,-7.4);
\draw(3,-4.1)--(3,-7.4);
\draw(4,-4.1)--(4,-7.4);
\draw(5,-4.1)--(6,-7.4);
\draw(6,-4.1)--(7,-7.4);
\draw(7,-4.1)--(8,-7.4);
\draw(8,-4.1)--(10,-6.2)--(10,-7.4);
\draw(9,-4.1)--(11,-6.2)--(11,-7.4);
\draw(10,-4.1)--(5,-7.4);
\draw(12,-4.1)--(12,-7.4);
\draw(11,-4.1)--(9,-6.2)--(9,-7.4);
\blackdot(9,-7.1);
\blackdot(9,-6.6);
\blackdot(9.2,-6.1);
\end{braid}
$$
\end{Example}

\begin{Corollary} \label{C160416G} 
Suppose that $i\in [a,b]$ is such that $\la_i<m$ and $\mu:=\la+\e_i=\la^\bga$. Then
\begin{equation*}
\psi_{u(\bga)}y^\mu e_\mu\psi_{\la;i}e_\la
=\sum_{r\in[1,m]:\ \ga_i^{(r)}=1} (-1)^{\sum_{s=r+1}^{m}\ga_i^{(s)}}\psi_{\bga-\bolde_i^{r};r,i}\psi_{u(\bga-\bolde_i^{r})}y^\la e_\la.
\end{equation*}
\end{Corollary}
\begin{proof}
The proof is by induction on $m$, the induction base $m=1$ being obvious.  Let $\bar\ga:=(\ga^{(1)},\dots,\ga^{(m-1)})$, $\bar\theta=(m-1)\al$, $\bar\mu:=\mu-\ga^{(m)}$, and $\bar\la:=\la-\ga^{(m)}$.
By~(\ref{WPsiWG}), we have
\begin{equation}\label{E200416G}
\psi_{u(\bga)}y^\mu e_\mu\psi_{\la;i}e_\la
=\psi_{z(\bga,2)}\dots \psi_{z(\bga,m)}
y^\mu e_\mu\psi_{\la;i}e_\la.
\end{equation}
Now we apply Lemma~\ref{L140416G}. We consider the case $\ga_i^{(m)}=1$, the case $\ga_i^{(m)}=0$ being similar. Then we get the following expression for the right hand side of (\ref{E200416G}):
\begin{align*}
&\psi_{z(\bga,2)}\dots \psi_{z(\bga,m-1)}\big(\psi_{\bga-\bolde_i^{m};m,i}\psi_{z(\bga-\bolde_i^{m},m)}y^\la e_\la
\\
-&(y^{\bar\mu}e_{\bar\mu}
\psi_{\bar\la;i}e_{\bar\la} \circ 1_\al)\psi_{z^{\ga^{(m)}}_\la} y_{r_{b+1}(\la)}^{m-1}e_\la\big)
.
\end{align*}
Opening parentheses,  we get two summands $S_1+S_2$.
Note that
\begin{align*}
S_1&=\psi_{z(\bga,2)}\dots \psi_{z(\bga,m-1)}
\psi_{\bga-\bolde_i^{m};m,i}\psi_{z(\bga-\bolde_i^{m},m)}y^\la e_\la
\\
&=  \psi_{\bga-\bolde_i^{m};m,i}\psi_{z(\bga,2)}\dots \psi_{z(\bga,m-1)}
\psi_{z(\bga-\bolde_i^{m},m)}y^\la e_\la
\\
&=  \psi_{\bga-\bolde_i^{m};m,i}\psi_{u(\bga-\bolde_i^{m})}y^\la e_\la.
\end{align*}
Moreover,
using the inductive assumption,
for the third equality below,
 $S_2$ equals
\begin{align*}
&-\psi_{z(\bga,2)}\dots \psi_{z(\bga,m-1)}
(y^{\bar\mu}e_{\bar\mu}\psi_{\bar\la;i}e_{\bar\la} \circ 1_\al)\psi_{z^{\ga^{(m)}}_\la}y_{r_{b+1}(\la)}^{m-1}e_\la
\\
=& -
\big(\psi_{z(\bar\bga,2)}\dots \psi_{z(\bar\bga,m-1)}y^{\bar\mu}e_{\bar\mu}\psi_{\bar\la;i}e_{\bar\la} \circ 1_\al\big)\psi_{z^{\ga^{(m)}}_\la}y_{r_{b+1}(\la)}^{m-1}e_\la
\\
=& -
(\psi_{u(\bar\bga)}y^{\bar\mu}e_{\bar\mu}\psi_{\bar\la;i}e_{\bar\la} \circ 1_\al)\psi_{z^{\ga^{(m)}}_\la}y_{r_{b+1}(\la)}^{m-1}e_\la
\\
=& -
\Big[\sum_{r\in[1,m-1]:\ \ga_i^{(r)}=1} \hspace{-7mm}(-1)^{\sum_{s=r+1}^{m-1}\ga_i^{(s)}}\psi_{\bar\bga-\bolde_i^{r};r,i}\psi_{u(\bar\bga-\bolde_i^{r})}y^{\bar \la} e_{\bar \la} \circ 1_\al\Big]
 \psi_{z^{\ga^{(m)}}_\la} y_{r_{b+1}(\la)}^{m-1}e_\la
\\
=& \sum_{r\in[1,m-1]:\ \ga_i^{(r)}=1} (-1)^{\sum_{s=r+1}^m \ga_i^{(s)}}\psi_{\bga-\bolde_i^{r};r,i}
\psi_{z(\bga-\bolde_i^{r},2)}\dots \psi_{z(\bga-\bolde_i^{r},m-1)}
\\
&\times \Big[y^{\bar \la} e_{\bar \la} \circ 1_\al\Big]\psi_{z(\bga-\bolde_i^{r},m)}y_{r_{b+1}(\la)}^{m-1}e_\la
\\
=& \sum_{r\in[1,m-1]:\ \ga_i^{(r)}=1} (-1)^{\sum_{s=r+1}^m \ga_i^{(s)}}\psi_{\bga-\bolde_i^{r};r,i}
\psi_{z(\bga-\bolde_i^{r},2)}\dots \psi_{z(\bga-\bolde_i^{r},m)}
\\
&\times \Big[y^{\bar \la} e_{\bar \la} \circ 1_\al\Big]y_{r_{b+1}(\la)}^{m-1}e_\la
\\
=& \sum_{r\in[1,m-1]:\ \ga_i^{(r)}=1} (-1)^{\sum_{s=r+1}^m \ga_i^{(s)}}\psi_{\bga-\bolde_i^{r};r,i}
\psi_{u(\bga-\bolde_i^{r})} y^\la e_\la.
\end{align*}
Thus $S_1+S_2$ equals the right hand side of the expression in the corollary.
\end{proof}

The following statement means that $g$ is chain map:

\begin{Proposition} \label{gchain}
Let $\la\in\La(n)$ and $\bga\in\bLa(n+1)$. Then
$$
\sum_{\mu\in\La(n+1)} g^{\bga,\mu}_{n+1}d_n^{\mu,\la}=\sum_{\bde\in\bLa(n)} c_n^{\bga,\bde}g^{\bde,\la}_{n}.
$$
\end{Proposition}
\begin{proof}
By definition, $d_n^{\mu,\la}=0$ unless $\mu=\la+\e_i$ for some $i\in [a,b]$, and $g^{\bga,\mu}_{n+1}=0$ unless $\mu=\la^\bga$. On the other hand, $g^{\bde,\la}_{n}=0$ unless $\la=\la^\bde$,
and
$c_n^{\bga,\bde}=0$, unless $\bde=\bga-\bolde_i^{r}$ for some $i\in [a,b]$ and $r\in[1,m]$. So we may assume that there exists $i\in [a,b]$ such that $\la^\bga=\la+\e_i$, in which case, setting $\mu:=\la^\bga$, we have to prove
$$
g^{\bga,\mu}_{n+1}d_n^{\mu,\la}=\sum_{r\in [1,m]:\ \ga_i^{(r)}=1}
c_n^{\bga,\bga-\bolde_i^{r}}g^{\bga-\bolde_i^{r},\la}_{n}.
$$
By definition of the elements involved, this means
\begin{align*}
&(\tau_\bga e_\bga\psi_{u(\bga)}y^\mu e_\mu)
(\sgn_{\la;i}e_\mu\psi_{\la;i}e_\la)
\\
=
&\sum_{r\in [1,m]:\ \ga_i^{(r)}=1}
(\sgn_{\bga-\bolde_i^{r};r,i}e_{\bga}\psi_{\bga-\bolde_i^{r};r,i}e_{\bga-\bolde_i^{r}})
(\tau_{\bga-\bolde_i^{r}} e_{\bga-\bolde_i^{r}}\psi_{u(\bga-\bolde_i^{r})}y^\la e_\la).
\end{align*}
Equivalently, we need to prove
\begin{align*}
\tau_\bga \sgn_{\la;i} \psi_{u(\bga)}y^\mu  e_\mu\psi_{\la;i}e_\la
=
\sum_{r\in [1,m]:\ \ga_i^{(r)}=1}
\sgn_{\bga-\bolde_i^{r};r,i} \tau_{\bga-\bolde_i^{r}}
\psi_{\bga-\bolde_i^{r};r,i}\psi_{u(\bga-\bolde_i^{r})}y^\la e_\la,
\end{align*}
which, in view of Corollary~\ref{C160416G}, is equivalent to the statement that
$$
\tau_\bga \sgn_{\la;i}=\sgn_{\bga-\bolde_i^{r};r,i} \tau_{\bga-\bolde_i^{r}}(-1)^{\sum_{s=r+1}^m \ga_i^{(s)}}
$$
for all $r\in [1,m]$ such that $\ga_i^{(r)}=1$. But this is Lemma~\ref{LSignsG}.
\end{proof}

\end{document}